\documentclass[a4paper,12pt]{amsart}
\usepackage[utf8]{inputenc}
\usepackage{csquotes}
\usepackage{amsmath, amssymb, amsthm, xcolor, enumerate, enumitem}
\usepackage[english]{babel}
\usepackage{tikz-cd}
\usepackage[all]{xy}
\usepackage{microtype}
\usepackage[colorlinks=true, citecolor=blue, linkcolor=blue, bookmarks=true]{hyperref}
\usepackage{bbm,stmaryrd}
\usepackage{graphicx}
\usepackage{adjustbox}
\usepackage{comment}
\usepackage{hyperref}
\usepackage[capitalize]{cleveref}
\usepackage{mathtools}

\newcounter{theorem}
\newtheorem{theorem}[theorem]{Theorem}
\newtheorem{lemma}[theorem]{Lemma}
\newtheorem{prop}[theorem]{Proposition}
\newtheorem{cor}[theorem]{Corollary}

\theoremstyle{definition}
\newtheorem{defn}[theorem]{Definition}
\newtheorem{notation}[theorem]{Notation}

\theoremstyle{remark}
\newtheorem*{remark*}{Remark}
\newtheorem{rmk}[theorem]{Remark}
\newtheorem{example}[theorem]{Example}

\numberwithin{equation}{section}
\allowdisplaybreaks

\crefname{theorem}{Theorem}{Theorems}
\crefname{lemma}{Lemma}{Lemmas}
\crefname{prop}{Proposition}{Propositions}
\crefname{cor}{Corollary}{Corollaries}
\crefname{defn}{Definition}{Definitions}
\crefname{notation}{Notation}{Notations}
\crefname{rmk}{Remark}{Remarks}
\crefname{example}{Example}{Examples}
\crefname{equation}{}{}

\newcommand{\KK}{\mathrm{KK}}
\newcommand{\K}{\mathrm{K}}
\newcommand{\asy}{\mathrm{asymp}}
\newcommand{\Bim}{\mathrm{Bim}}

\newcommand{\End}{\operatorname{End}}
\newcommand{\Ad}{\mathrm{Ad}}
\newcommand{\Hilb}{\mathrm{Hilb}}
\newcommand{\Irr}{\mathrm{Irr}}

\newcommand{\Hom}{\mathrm{Hom}}
\newcommand{\spa}{\mathrm{span}}
\newcommand{\au}{\approx_{\mathrm{u}}}

\newcommand{\id}{\mathrm{id}}
\newcommand{\op}{\mathrm{op}}
\newcommand{\ev}{\mathrm{ev}}

\newcommand{\bC}{\mathbb{C}}
\newcommand{\bE}{\mathbb{E}}
\newcommand{\bK}{\mathbb{K}}
\newcommand{\bN}{\mathbb{N}}

\newcommand{\bU}{\mathbb{U}}

\newcommand{\bZ}{\mathbb{Z}}
\newcommand{\bu}{\mathbbm{u}}
\newcommand{\bv}{\mathbbm{v}}
\newcommand{\bw}{\mathbbm{w}}

\newcommand{\cA}{\mathcal{A}}
\newcommand{\cB}{\mathcal{B}}
\newcommand{\cC}{\mathcal{C}}
\newcommand{\cH}{\mathcal{H}}
\newcommand{\cJ}{\mathcal{J}}
\newcommand{\cK}{\mathcal{K}}
\newcommand{\cL}{\mathcal{L}}
\newcommand{\cF}{\mathcal{F}}
\newcommand{\cG}{\mathcal{G}}
\newcommand{\cM}{\mathcal{M}}
\newcommand{\cU}{\mathcal{U}}
\newcommand{\cO}{\mathcal{O}}

\newcommand{\fC}{\mathfrak{C}}
\newcommand{\fD}{\mathfrak{D}}

\newcommand{\fu}{\mathfrak{u}}
\newcommand{\fv}{\mathfrak{v}}

\newcommand{\rs}{\mathrm{s}}
\newcommand{\fw}{\mathfrak{w}}

\newcommand{\bfE}{\mathbf{E}}

\newcommand{\nn}[1]{\textcolor{red}{[[#1]]}}

\title[stable uniqueness for KK$^{\mathcal{C}}$]{The stable uniqueness theorem for unitary tensor category equivariant KK-theory}

\author{Sergio Girón Pacheco}
\address{\hskip-\parindent Sergio Girón Pacheco, Department of mathematics, KU Leuven, Celestijnenlaan 200B, 3001, Leuven, Belgium.}
\email{sergio.gironpacheco@kuleuven.be}

\author{Kan Kitamura}
\address{\hskip-\parindent Kan Kitamura, RIKEN iTHEMS, 2-1 Hirosawa, Wako Saitama, 351-0198, Japan.}
\email{kan.kitamura@riken.jp}

\author{Robert Neagu}
\address{\hskip-\parindent Robert Neagu, Department of mathematics, KU Leuven, Celestijnenlaan 200B, 3001, Leuven, Belgium.}
\email{robert.neagu@kuleuven.be}

\thanks{Funded by the European Union. Views and opinions expressed are however those of the authors only and do not necessarily reflect those of the European Union or the European Research Council. Neither the EU nor the ERC can be held responsible for them. For the purpose of open access, the authors have applied a CC BY license to any author accepted manuscript arising from this submission.}

\begin{document}

\begin{abstract}
We introduce the Cuntz–Thomsen picture of $\cC$-equivariant Kasparov theory, denoted $\KK^{\cC}$, for a unitary tensor category $\cC$ with countably many isomorphism classes of simple objects. We use this description of $\KK^{\cC}$ to prove the stable uniqueness theorem in this setting.
\end{abstract}

\maketitle

\setcounter{tocdepth}{1}
\tableofcontents

\numberwithin{theorem}{section}	

\section*{Introduction}
\renewcommand*{\thetheorem}{\Alph{theorem}}

The study of K-theory, K-homology, and their interactions has proven to be of fundamental importance in the development of C$^*$-algebra theory. Notably, these ideas were exploited early in the work of Atiyah and Singer (\cite{ATSI68}), and further developed by Brown, Douglas, and Fillmore (\cite{BDF1,BDF2}). Furthermore, K-theory has proven to be an essential ingredient in the classification programme for C$^*$-algebras, particularly since Elliott's classification of AF-algebras (\cite{elliott}). In a tour de force (\cite{KA88}), Kasparov introduced $\KK$-theory, a bivariant framework that unifies and extends $\K$-theory and $\K$-homology. From its inception, $\KK$-theory was formulated in the general setting of C$^*$-dynamical systems, i.e. C$^*$-algebras carrying a group action, a level of generality crucial for Kasparov's applications to the Novikov conjecture. 

The first major application of $\KK$-theory in C$^*$-classification is the Kirchberg--Phillips classification (\cite{Kirch00,phillipsclass,kirchbergclass}), establishing that a certain class of C$^*$-algebras exhibits a form of homotopy rigidity with respect to the $\KK$-bifunctor. In subsequent approaches to the Kirchberg--Phillips classification theorem (\cite{DADEIL02,oinftyclass}), an existence and uniqueness strategy is employed (see \cite{DADEIL02} for a discussion on such a strategy). A crucial component of this strategy, is showing that two $^*$-homomorphisms between Kirchberg algebras that induce the same map under the $\KK$-functor are asymptotically unitarily equivalent. Uniqueness theorems for $\KK$-theory of this type have been central to following breakthrough classification results (\cite{oinftyclass,CS20,classif}).

One of the features that makes $\KK$-theory widely applicable is the existence of many equivalent formulations, each suited for a different purpose. One such description of $\KK$-theory is due to Cuntz (\cite{Cu83,Cu84}), with a group equivariant counterpart due to Thomsen (\cite{Th98}). In Cuntz's picture, elements of the Kasparov group $\KK(A,B)$ are realised as homotopy classes of pairs of $^*$-homomorphisms $\phi,\psi\colon A \to \cM(B \otimes \mathbb{K})$ such that $\phi(a)-\psi(a)\in B\otimes \bK$ for all $a \in A$. Such a pair $(\phi,\psi)$ is commonly called a Cuntz pair. One prominent theorem which requires this picture is the so-called \emph{stable uniqueness theorem} of Dadarlat-Eilers (\cite{DADEIL01}) and Lin (\cite{Lin02}), which roughly states that if a Cuntz pair $(\phi,\psi)$ induces the zero element in $\KK(A,B)$, then after stabilising  both $\phi$ and $\psi$ with another $^*$-homomorphism, they become homotopic via a well-behaved unitary path. In the past decade, this theorem was not only a key tool to a more recent proof of the Kirchberg--Phillips theorem (\cite{oinftyclass}), but it has also played a crucial role in various other groundbreaking results such as \cite{TWW,CS20,classif}. Moreover, the group equivariant version of this theorem proved by Gabe and Szabó in \cite{GASZ22}, was used by the same authors in their breakthrough classification of group actions on Kirchberg algebras (\cite{DynamicalKP}).

Due to the great success of Jones' subfactor theory (\cite{jonessubfactors}), and in particular Popa's seminal classification results for amenable subfactors (\cite{Popa94}), there has been growing interest in studying quantum symmetries on amenable C$^*$-algebras. A quantum symmetry on a C$^*$-algebra is the action of a unitary tensor category $\cC$ on $A$, or more precisely, a C$^*$-tensor functor from $\cC$ into the category of bimodules $\Bim(A)$. Following the recent introduction of the unitary tensor category equivariant $\KK$-theory (\cite{ARKIKU23}), denoted $\KK^{\cC}$, it is natural to examine to what extent $\KK^{\cC}$ may be suitable for classification purposes. In an effort to shed light on $\cC$-equivariant $\KK$-uniqueness, we prove the $\cC$-equivariant version of the stable uniqueness theorem.

\par To do this, we start by introducing a Cuntz-Thomsen formulation of $\KK^{\cC}$, showing its equivalence to the construction in \cite{ARKIKU23} (see Theorem \ref{thm:Cuntzpicture}).\ As in \cite{Th98}, to introduce this picture of $\KK^{\cC}((\alpha,\fu),(\beta,\fv))$ for $\cC$-C$^*$-algebras $(A,\alpha,\fu)$ and $(B,\beta,\fv)$, it is important to consider the right notion of a $\cC$-equivariant representation from $A$ to $B$, often called a cocycle representation. In our setting, this coincides with that of a $\cC$-equivariant structure on a correspondence isomorphic to $B$ as a right module. A crucial observation that underpins most of the results in this paper is that such a $\cC$-equivariant representation can be repackaged as a family of compatible linear maps indexed by the elements in the category. More precisely, as shown in Proposition \ref{prop:phiXmaps}, for $\cC$-C$^*$-algebras $(A,\alpha,\fu)$ and $(B,\beta,\fv)$, a $\cC$-cocycle representation $\phi\colon (A,\alpha,\fu)\to (B,\beta,\fv)$ consists of a family of linear maps $\{\phi_X\colon\alpha(X)\to\cL(B,\beta(X))\}_{X\in\cC}$ satisfying certain compatibility conditions. With this viewpoint in hand, we can state a simplified version of the $\cC$-equivariant stable uniqueness theorem.

\begin{theorem}[Theorem \ref{thm: StableUniq}]\label{thm:A}
Let $\cC$ be a unitary tensor category with countably many isomorphism classes of simple objects, $(A,\alpha,\fu)$ be a separable $\cC$-C$^*$-algebra, and $(B,\beta,\fv)$ be a $\sigma$-unital $\cC$-C$^*$-algebra. Let \[\phi,\psi\colon(A,\alpha,\fu)\to (B\otimes\bK,\beta\otimes\id_{\bK},\fv\otimes 1)\] be two $\cC$-cocycle representations that form a $\cC$-Cuntz pair. Then the pair $(\phi,\psi)$ represents the zero element in $\KK^{\cC}((\alpha,\fu),(\beta,\fv))$ if and only if there exists another $\cC$-cocycle representation $\theta$ and a norm-continuous path of unitaries $u\colon[0,\infty)\to\cU(1+B\otimes\bK)$ with $u_0=1$ such that for all $X\in\cC$ and $\xi\in\alpha(X)$ \[(\psi_X\oplus\theta_X)(\xi)=\lim\limits_{t\to\infty}u_t\rhd (\phi_X\oplus\theta_X)(\xi)\lhd u_t^*.\footnote{We denote by $\rhd$ and $\lhd$ the actions of $\cM(B)$ on the $\cM(B)$-$\cM(B)$-correspondence $\cL(B,\beta(X))$. The direct sum above is also understood in this sense by taking isometries forming a copy of $\cO_2$ in $\cM(\bK)$. The notation of this theorem is carefully explained in Sections \ref{sect: Prelim} and \ref{sect:Cuntzpicture}.}\] 
\end{theorem}

Due to the fact that many unitary tensor categories do not admit actions on the complex numbers, and the actions we consider are by correspondences rather than endomorphisms, one cannot directly employ the machinery developed in the group equivariant setting (\cite{GASZ22}). To circumvent these issues, we need to introduce the right counterparts for notions such as Cuntz pairs, weak containment, or asymptotic unitary equivalence, which do not restrict to the analogous notions in \cite{GASZ22}. Indeed, the categorical notion of equivariant morphism between actions of tensor categories leads us to handle the underlying $^*$-homomorphism and unitary cocycles into a combined map, rather than separate compatible objects.
This leads to different but appropriate notions in our setting.
Thus, the reader should be warned that our stable uniqueness theorem differs to the one in \cite{GASZ22}, even in the group equivariant case. Futhermore, a fundamental difference in our proofs is the use of duality in a unitary tensor category to replace the invertibility in a group. We set up the necessary framework to make use of duality in Section \ref{sect: Prelim}.
\par An important milestone in the proof of Theorem \ref{thm: StableUniq} is showing the existence of absorbing representations in the $\cC$-equivariant setting (see Theorem \ref{thm:absorbingexists}). In the past decade, through a result of Elliott and Kucerovsky (\cite{EllKu01}), verifying when a representation is absorbing proved to be a key step in major results regarding structure and classification of C$^*$-algebras (\cite{ChrisTWW,CS20,oinftyclass,classif}). Therefore, an appealing result that we would like to point out is Lemma \ref{lem:biglemsec3}, which contains sufficient conditions to determine when a given $\cC$-cocycle representation absorbs another $\cC$-cocycle representation. This result is in the spirit of (\cite{Ka80,GASZ22}) and Voiculescu's theorem (\cite{Voi76}). As absorbing representations exist, we may get control over the structure of the stabilisation $\theta$ in Theorem \ref{thm:A} (in fact any absorbing representation will suffice). This enhancement of the stable uniqueness theorem is crucial for its applications in \cite{DynamicalKP}.

\subsection*{Acknowledgements}
We would like to thank Gabór Szabó for useful discussions on the topic of this paper. This work started during the authors' stay at the International Centre for Mathematical Sciences (ICMS) for the 'Twinned Conference on C$^*$-Algebras and Tensor Categories' in November 2023. We thank ICMS and the organisers for the hospitality.

\par SGP was supported by projects G085020N and 1249225N funded by the Research Foundation Flanders (FWO). KK was partly supported by JSPS KAKENHI Grant Number JP21J21283 and RIKEN Special Postdoctoral Researcher Program.
RN was supported by the EPSRC grant EP/R513295/1, by the European Research Council under the European Union's Horizon Europe research and innovation programme (ERC grant AMEN-101124789), and by the postdoctoral fellowship 1204626N of the Research Foundation Flanders (FWO).
\allowdisplaybreaks

\section{Preliminaries}\label{sect: Prelim}
\numberwithin{theorem}{section}

We denote by $\bK$ the C$^*$-algebra of compact operators on an infinite dimensional separable Hilbert space. For a C$^*$-algebra $A$ we denote its stabilization by $A^\rs=A\otimes\bK$. For a Banach space $B$, we write $B_1$ for its unit ball and for another Banach space $E$ we write $\cB(E,B)$ for the Banach space of bounded linear operators from $E$ to $B$. For $x,y\in B$ and $\varepsilon>0$, we denote $x=_\varepsilon y$ if $\|x-y\|\leq \varepsilon$. For a closed subspace $J\subset B$, we say $x\equiv y$ mod $J$ if $x-y\in J$ and $x\equiv_{\varepsilon} y$ mod $J$ if $x-y\in J$ and $\|x-y\|\leq \varepsilon$.
\subsection{Hilbert modules}
\par We assume standard facts about Hilbert modules and C$^*$-corresponde\-nces as can be found in \cite{Hilbertmodules}. We will often simply write correspondence to mean C$^*$-correspondence. In our case, Hilbert modules will have inner products on the right. If $E$ and $H$ are Hilbert $B$-modules, we denote by $\cL_B(E,H)$ the Banach space of adjointable $B$-module maps between $E$ and $H$ and by $\cK_B(E,H)$ the compact operators from $E$ to $H$, i.e., the norm closure of the finite rank operators in $\cL_B(E,H)$.
If $B$ is clear from the context, we simply write $\cL(E,H)$ and $\cK(E,H)$. Recall that $\cL(E,E)$, which we simply denote by $\cL(E)$, is a C$^*$-algebra. If $E$ and $H$ are in fact $A$-$B$-correspondences, then we denote by $\Hom(E,H)$ those operators in $\cL(E,H)$ that also commute with the left $A$-action and denote $\Hom(E,E)$ by $\End(E)$. For an $A_1$-$B_1$-correspondence $E_1$ and an $A_2$-$B_2$-correspondence $E_2$, we denote their external tensor product by $E_1\boxtimes E_2$. Moreover, if $B_1=A_2$, we denote their internal tensor product by $E_1\otimes_{A_2} E_2$ (often omitting the subscript). We denote by $E_1^{s}=E_1\boxtimes \bK$. To emphasise the left and right actions on an $A$-$B$-correspondence $E$, we often denote these by $a\rhd \xi$ and $\xi\lhd b$ for $a\in A$, $b\in B$ and $\xi\in E$.
\par We say an $A$-$B$-correspondence $E$ is \emph{non-degenerate} if $\overline{A\rhd E}=A\rhd E=E$ (the first equality is by virtue of Cohen factorisation). For any non-degenerate $A$-$B$-correspondence $E$, we denote by $l_E\in \Hom(E,A\otimes E)$ and $r_E\in \Hom(E,E\otimes B)$ the canonical natural unitary isomorphisms given by using the approximate units of $A$ and $B$, respectively (see the discussion above \cite[Lemma~1.11]{intertwining}).\footnote{Here we are viewing $A$ and $B$ as C$^*$-correspondences over themselves} We often omit the maps $l_E$ and $r_E$ from notation when they are implied. We call an $A$-$B$-correspondence $E$ \emph{proper} if the left action consists of compact operators on $E$. We denote by $\Bim(B)$ the C$^*$-tensor category of non-degenerate C$^*$-correpondences over $B$ under the internal tensor product.
%\par For a Hilbert $B$-module $E$, we sometimes write $\langle E|\coloneqq \cK(E,B)$, which is a $B$-$\cK(E)$-correspondence and $|E\rangle\coloneqq E$ \nn{Do we actually use this notation? SG}\nn{I searched all `$\rangle$' and `$\langle$' to see there is no occurrence of $|Hibert\ module\rangle$ nor $\langle Hibert\ module|$}. 
%Similarly, we denote by $\langle\xi|\in\cK(E,B)$ for $\xi\in E$ the operator given by $\langle \xi| \eta=\langle \xi,\eta\rangle$ for $\eta\in E$ and so for another Hilbert $B$-module $E'$ and $\zeta\in E'$ we denote by $|\zeta\rangle \langle \xi|\in \cK(E,E')$ the rank one operator defined by $|\zeta\rangle \langle \xi|(\eta)=\zeta\lhd \langle \xi,\eta\rangle$. \nn{Also, I see the only occurrence of this notation was in the proof of \cref{lem:compactstensoridentity}.}

\par For an $A$-$B$-correspondence $(E,\phi)$, with $\phi\colon A\rightarrow \cL(E)$ the left action, and a Hilbert $A$-module $E'$ with $\xi\in E'$, we write 
\begin{align*}
T_\xi\colon &E\mapsto E'\otimes_A E, \\
& \eta\mapsto \xi\otimes \eta.
\end{align*}
Note that $T_{\xi}\in \cL(E,E'\otimes E)$ with adjoint given by $T_{\xi}^*(\zeta\otimes\eta)=\phi(\langle \xi,\zeta\rangle)\eta$. We sometimes write $T_\xi^\phi$ instead of $T_\xi$ to emphasise the dependence on $\phi$. We collect the following result about proper correspondences.

\begin{lemma}[cf.~{\cite[Proposition 4.7]{Hilbertmodules}}]\label{lem:compactstensoridentity}
    Let $E$ and $E'$ be Hilbert $A$-modules, $M$ be a proper $A$-$B$-correspondence and $T\in\cL(E,E')$. Then if $T\in \cK(E,E')$ then $T\otimes \id_{M}\in \cK(E\otimes M,E'\otimes M)$. Thus, if $M$ is a Morita equivalence between $A$ and $B$ (i.e. $M$ is invertible under the internal tensor product), then $T\otimes \id_{M}\in \cK(E\otimes M,E'\otimes M)$ implies that $T\in \cK(E,E')$. 
\end{lemma}
\begin{comment}
    As any compact operator is in the closure of the span of rank one operators it suffices to show that, for any rank one operator $|\zeta\rangle\langle\xi|$ for $\zeta\in E'$ and $\xi\in E$, $|\zeta\rangle\langle \xi|\otimes\id_M\in \cK(E\otimes M,E'\otimes M)$. Moreover, as $E'$ has a right $A$-valued inner product it is non-degenerate as a right Hilbert $A$-module (see \cite[Proposition 2.16]{KWP04}) and hence by Cohen factorisation we may assume that $\zeta=\zeta'\lhd a$ for some $\zeta'\in E'$ and $a\in A$. Now $|\zeta\rangle\langle \xi|\otimes\id_M=|\zeta'\lhd a\rangle\langle \xi|\otimes \id_M=T_{\zeta'}\circ \varphi(a)\circ T_{\xi}^*\in \cK(E\otimes M,E'\otimes M)$ by the ideal property of compact operators, where $\varphi\colon A\rightarrow \cK(M)$ denotes the left action on $M$.
    \par If $M$ is invertible under the internal tensor product let $M^{-1}$ be a $B$-$A$-correspondence such that $M\otimes M^{-1}\cong A$.  If $T\in \cL(E\otimes M,E'\otimes M)$ is such that $T\otimes \id_M\in \cK(E\otimes M,E'\otimes M)$ we may apply the result proven in the previous paragraph to imply that $T\otimes\id_M\otimes \id_{M^{-1}}\in \cK(E\otimes M\otimes M^{-1},E'\otimes M\otimes M^{-1})$. By pre and poscomposing by the isomorphism between $M\otimes M^{-1}$ and $A$ we have that $T\in \cK(E,E')$.
\end{comment}
\par For a family of $A$-$B$-correspondences $(E_\lambda)_\lambda$, we write $\bigoplus_\lambda E_\lambda$ for the $\ell^2$-completion of their direct sum. 
In particular, we write $E^{\infty}$ for the $\ell^2$-completion of direct sum of countably infinitely many copies of a correspondence $E$. 
We note the following technical lemma, which is a special case of \cite[Lemma 4.3]{MEY00}, to be used in the proof of Theorem \ref{thm:Cuntzpicture}. 
%For a separable right Hilbert $A$-module $E$ we canonically identify its external tensor product $\bK\boxtimes E$ with its infinite $l^2$ direct sum $E^{\infty}$ throughout this paper. \nn{Is this disclaimer valid?}
%Let $(E,\phi,\bu)$ be a $\cC$-$(A,B)$-correspondence. For $X\in\cC$, we put the well-defined $^*$-homomorphism $\phi_X\colon \cK(\alpha(X))\to \cL(E\otimes_B\beta(X))$ defined by $\phi_X(|\xi\ket\bra\eta|)\coloneqq (\bu_X T_\xi^{\phi})(\bu_X T_\eta^{\phi})^*$ for $\xi,\eta\in\alpha(X)$, that the image of $\phi_X$ is adjointable follows from \cite[Lemma 2.31]{ARKIKU23}. 

\begin{lemma}[cf.~{\cite[Lemma 4.3]{MEY00}}]\label{lem:strictpathconnisomet}
Let $B$ be a $\sigma$-unital C$^*$-algebra and $E$ be a countably generated Hilbert $B$-module. For any isometries $S_+,S_-\in \cL(E,B^{\infty})$, by viewing both $E\boxtimes C[0,1]$ and $B^{\infty}\boxtimes C[0,1]$ as Hilbert $B\otimes C[0,1]$-modules, there is an isometry $(S_t)_{t\in[0,1]}\in \cL_{B\otimes C[0,1]}(E\boxtimes C[0,1],B^{\infty}\boxtimes C[0,1])$ whose evaluations at $t=0,1$ satisfy $S_0=S_+$ and $S_1=S_-$. 
\end{lemma}
 %(To answer the question therein, $s_t\colon L^2[0,1]\ni f\mapsto [s\mapsto t^{-1}\mathbbm{1}_{[0,1]}(t^{-1}s)f(t^{-1}s) ]$ for $t\in[1/2,1]$ gives a strictly continuous path of isometries with $s_1=\id$ and $s_{1/2}$ being the isometry $L^2[0,1]\to L^2[0,1/2](\subset L^2[0,1])$. Similarly, we have a strictly continuous path between $\id$ and the isometry $L^2[0,1]\to L^2[1/2,1]$.) 

\begin{comment}
	Take isometries $s_\pm\in\cB(\ell^2(\bN))$ with $s_+^*s_-=0$. 
	There are strictly continuous paths of isometries between $s_\pm$ and $\id_{\ell^2(\bN)}$ \nn{Is there a reference for isometries in $\cB(\cH)$ being path connected in the strict topology? SG}. 
	Thus, there are strictly continuous paths of adjointable isometries between $S_\pm$ and $(s_\pm\otimes1_B)S_\pm$. 
	Finally, $\sqrt{1-t^2}(s_+\otimes1_B)S_+ +t(s_-\otimes1_B)S_-$ for $t\in[0,1]$ gives a norm continuous path between $(s_+\otimes1_B)S_+$ and $(s_-\otimes1_B)S_-$. 
\end{comment}
%\todo{Unify the convention ``$(A,B)$-something" v.s. ``$A$-$B$-something".}

Given a C$^*$-algebra $B$ and Hilbert $B$-modules $E$ and $E'$, we shall adopt the following convention: for $a\in\cL(E)$, $b\in \cL(E',E)$, $c\in\cL(E,E')$, and $d\in \cL(E')$, we let $\left(\begin{array}{cc} a& b \\ c& d \end{array}\right) \in \cL(E\oplus E')$ denote the element $x\in \cL(E\oplus E')$ such that $p_E x p_E=a$, $p_E x p_{E'} =b$, $p_{E'} x p_E =c$, and $p_{E'} xp_{E'}=d$, where $p_E, p_{E'}\in\cL(E\oplus E')$ are the projections onto $E$ and $E'$ respectively. 

\subsection{Actions of unitary tensor categories}
\begin{comment}For C$^*$-tensor categories $(\mathcal{C},\otimes)$ and $(\mathcal{D},\boxtimes)$, $\alpha\colon \mathcal{C}\to\mathcal{D}$ is said to be a C$^*$-tensor functor if it is a C$^*$-functor such that $\alpha(1_{\cC})=1_{\mathcal{D}}$ and there exists unitary natural isomorphisms $J_{X,Y}\colon \alpha(X)\boxtimes \alpha(Y)\to \alpha(X\otimes Y)$ such that 

\begin{equation}\label{diagramJmaps}
\begin{tikzcd}[column sep=5em]
(\alpha(X)\boxtimes \alpha(Y))\boxtimes \alpha(Z)\arrow[r]\arrow[d,"J_{X,Y}\boxtimes\id_{\alpha(Z)}"]  & \alpha(X)\boxtimes (\alpha(Y) \boxtimes \alpha(Z)) \arrow[d,"\id_{\alpha(X)}\boxtimes J_{Y,Z}"] \\
\alpha(X \otimes Y)\boxtimes \alpha(Z) \arrow[d,"J_{X \otimes Y,Z}"] & \alpha(X)\boxtimes \alpha(Y \otimes Z) \arrow[d,"J_{X,Y \otimes Z}"]\\
\alpha((X \otimes Y) \otimes Z )\arrow[r] & \alpha(X \otimes (Y \otimes Z))
\end{tikzcd}
\end{equation}
commutes for all $X,Y,Z\in \mathcal{C}$.
\end{comment}
Throughout this paper, we will denote by $\cC$ a unitary tensor category, by $\otimes$ its monoidal product, by capital letters $X,Y,Z$ objects in $\cC$ and by $1_{\cC}$ the tensor unit. We will assume knowledge of the basic definitions in the theory of rigid C$^*$-tensor categories, which we call here \emph{unitary tensor categories} following the conventions of \cite{CHHPJOPE22} for example, and refer to \cite[Sections 2.1 and 2.2]{NETU13} for a detailed account of the topic. 
We always assume that $1_\cC$ is simple and 
that $\cC$ has countably many isomorphism classes of simple objects. 
We denote by $\Irr(\cC)$ a choice of isomorphism class representatives for the simple objects in $\cC$. Therefore $\Irr(\cC)$ is a countable set.
\par We choose standard solutions to the conjugate equations. For any object $X\in \cC$, these consist of a dual object $\overline{X}\in \cC$ and morphisms $\overline{R}_X\colon 1_{\cC}\rightarrow X\otimes \overline{X}$, $R_X\colon 1_{\cC}\rightarrow \overline{X}\otimes X$. Recall that, as these are solutions to the conjugate equations, the equalities
\begin{equation*}
    (\overline{R}_X^*\otimes \id_X)(\id_X\otimes R_X)=\id_X
\end{equation*}
and 
\begin{equation*}
    (R_X^*\otimes \id_{\overline{X}})(\id_{\overline{X}}\otimes \overline{R}_X)=\id_{\overline{X}}
\end{equation*}
hold for all $X\in \cC$. Moreover, standardness of the solutions is requiring that the left and right traces coincide, i.e., that for any $T\in \End(X)$ for $X\in \cC$
\begin{equation*}
    R_X^*(\id_{\overline{X}}\otimes T)R_X=\overline{R}_X^*(T\otimes\id_{\overline{X}})\overline{R}_X.
\end{equation*}
In particular, standard solutions satisfy $\|R_X\|=\|\overline{R}_X\|=d_X^{1/2}$, where $d_X$ is called the \emph{intrinsic dimension} of $X$. We denote by $\mu_X\colon X\rightarrow \overline{\overline{X}}$ the canonical pivotal structure defined by
\begin{equation}\label{eqn:muX}
\mu_X=(\id_{\overline{\overline{X}}}\otimes R_X^*)(R_{\overline{X}}\otimes \id_X)=(\overline{R}_X^*\otimes \id_{\overline{\overline{X}}})(\id_X\otimes \overline{R}_{\overline{X}}).\footnote{See \cite[Lemma 7.6]{SEL11} for this equality.}
\end{equation}
that is a unitary isomorphism. We now briefly recall important notions and set up notation about actions of C$^*$-tensor categories on C$^*$-algebras. We refer the reader to \cite{intertwining} or \cite{ARKIKU23} for more details.
\begin{defn}\label{defn: categaction}
An \emph{action} of a C$^*$-tensor category $\cC$ on a C$^*$-algebra $A$ is a C$^*$-tensor functor $(\alpha,\fu)\colon\mathcal{C}\to \Bim(A)$, where \[\fu\coloneqq\{\fu_{X,Y}\colon \alpha(X)\otimes \alpha(Y)\rightarrow \alpha(X\otimes Y)\}_{X,Y\in\cC}\] is the coherent family of unitary natural isomorphisms associated to the functor $\alpha$. We will often denote such an action by $(\alpha,\fu)\colon\mathcal{C}\curvearrowright A$. 
In this case, we say that $(A,\alpha,\fu)$ is a \emph{$\cC$-C$^*$-algebra}. 
%if each $\alpha(X)$ is isomorphic to $A$ as a right Hilbert $A$-module, in which case we denote by $\alpha_X\colon A\rightarrow \cM(A)$ the endomorphism defining the left action.
\end{defn}
\begin{rmk}
In this paper, $\cC$ will always be a unitary tensor category. In such a situation, whenever $A$ is separable, $\alpha(X)$ has a countable dense subset (see \cite[Corollary 2.24]{KWP04} and also \cite[(3) Lemma 2.26]{ARKIKU23}). Moreover, each $\alpha(X)$ is proper (this follows from \cite{KWP04}, see also \cite[(2) Lemma 2.26]{ARKIKU23}).
\end{rmk}

\begin{rmk}
    For a $\cC$-C$^*$-algebra $(A,\alpha,\fu)$ and $X\in\cC$, we denote \[R_X^{\alpha}=\fu_{\overline{X},X}^*\alpha(R_X)\ \text{and}\ \overline{R}_X^{\alpha}=\fu_{X,\overline{X}}^*\alpha(\overline{R}_X).\] Then $\alpha(\overline{X})$ is the dual of $\alpha(X)$ and $R_X^{\alpha}$, $\overline{R}_X^{\alpha}$ are also solutions to the conjugate equations. 
\end{rmk}

\begin{rmk}
    When $(\alpha,\fu)$ is an action of $\cC$ on a C$^*$-algebra $A$, and $B$ is another C$^*$-algebra, we can induce an action $(\alpha\otimes \id_{B},\fu\otimes 1)$ on the minimal tensor product $A\otimes B$ by 
\[(\alpha\otimes \id_{B})(X)=\alpha(X)\boxtimes B\] 
\[\fu_{X,Y}((\xi\otimes b)\otimes(\eta\otimes b')=\fu_{X,Y}(\xi\otimes \eta)\boxtimes bb'\] 
with $X,Y\in \cC$, $\xi\in \alpha(X)$, $\eta\in \alpha(Y)$ and $b,b'\in B$. In this paper, we will mainly use this construction with $B=\bK$ in which case we simply denote $(\alpha\otimes \id_{\bK},\fu\otimes 1)$ by $(\alpha^\rs,\fu^\rs)$. In the same way, one can also make sense of $(\id_B\otimes \alpha,1\otimes\fu)$ which is an action of $\cC$ on $B\otimes A$.
\end{rmk}

 For a non-degenerate $A$-$B$-correspondence $E$, we equip $\cL(B,E)$ with the structure of a $\cM(A)$-$\cM(B)$-correspondence with the left and right actions defined by $f\lhd g=f\circ g$ and $h\rhd f=l_E^*(h\otimes \id_{E})l_E \circ f$ for any $f\in \cL(B,E)$, $g\in \cM(B)$, $h\in\cM(A)$ and the right inner product defined by $\langle f,g\rangle=f^*\circ g$ for any $f,g\in \cL(B,E)$. 
 In these formulae, we identify $\cM(B)$ with $\cL(B)$. Similarly, we equip $\cL(E,B)$ with the structure of an (algebraic) $\cM(B)$-$\cM(A)$-bimodule by $h\rhd f=h\circ f$ and $f\lhd g=f\circ l_E^{-1}(g\otimes \id_E)l_E$ for $h\in \cM(B)$, $g\in \cM(A)$ and $f\in \cL(E,B)$. We now set up the conjugate operations defined on those bimodules that appear as the image of an action of $\cC$ on a C$^*$-algebra. This operation is involutive up to $\mu_X$, where $\mu_X$ is the pivotal structure of \eqref{eqn:muX}. 

\begin{lemma}\label{lemma:barnotation}
    \par Let $(A,\alpha,\fu)$ be a $\cC$-C$^*$-algebra. For $X\in \cC$ and $\xi\in \cL(A,\alpha(X))$, there is unique $\overline{\xi}\in\cL(A,\alpha(\overline{X}))$ such that 
    \begin{align}\label{eqn:barinnerprod}
        &
        \langle\xi,\eta\rangle=\xi^*\eta=(R_X^{\alpha})^*(\overline{\xi}\otimes_A \id_{\alpha(X)}) \eta
    \end{align}
    for all $\eta \in \cL(A,\alpha(X))$. 
    This gives an anti-linear bijection 
    \begin{align*}
        \cL(A,\alpha(X))\ni \xi &\mapsto \overline{\xi} \in \cL(A,\alpha(\overline{X})),
    \end{align*}
    which restricts to a bijection from $\cK(A,\alpha(X))$ to $\cK(A,\alpha(\overline{X}))$. 
    Moreover, one has that 
    $\overline{x\rhd\xi\lhd y}=y^*\rhd \overline{\xi}\lhd x^*$, 
    $\overline{\overline{\xi}}=\alpha(\mu_X)\xi$, and $\|\overline{\xi}\|\leq d_X^{1/2}\|\xi\|$ for all $x,y\in \cM(A)$ and $\xi\in \cL(A,\alpha(X))$. 
\end{lemma}
\begin{proof}
Consider the following linear maps 
\begin{align*}
    S_X\colon \cL(\alpha(\overline{X}),A)&\rightarrow \cL(A,\alpha(X))\\
    \xi&\mapsto (\xi\otimes \id_{\alpha(X)})\circ R_{X}^{\alpha},
\end{align*}
\begin{align*}
    \overline{S}_X\colon \cL(\alpha(X),A)&\rightarrow \cL(A,\alpha(\overline{X}))\\
    \xi &\mapsto (\xi\otimes \id_{\alpha(\overline{X})})\circ \overline{R}_{X}^{\alpha}.
\end{align*}
By definition, both $S_X$ and $\overline{S}_X$ are left $\cM(A)$-linear, and it holds $\|\overline{S}_X\|\leq \| \fu_{X,\overline{X}} \alpha(\overline{R}_{X})\|\leq d_X^{1/2}$.

Using the conjugate equations, it is routine to check that the assignment $\xi\mapsto (S_X(\xi^*))^*$ for $\xi\in\cL(A,\alpha(\overline{X}))$ defines the inverse to $\overline{S}_X$, which restricts to a bijection from $\cK(\alpha(X),A)$ to $\cK(A,\alpha(\overline{X}))$ due to the properness of $\alpha(X)$ and $\alpha(\overline{X})$ by \cite[Lemma 2.26 (2)]{ARKIKU23}. 
Since $S_X$ is left $\cM(A)$-linear, we see that $\overline{S}_X^{-1}$ and thus $\overline{S}_X$ are also right $\cM(A)$-linear. 
Note that as $\langle \xi,\eta\rangle=\xi^*\eta$ for $\xi,\eta\in \cL(A,\alpha(X))$ is a positive definite $\cM(A)$ valued inner product, the existence of $\overline{\xi}$ satisfying condition \eqref{eqn:barinnerprod} would imply that $\xi^*=(R_X^{\alpha})^*(\overline{\xi}\otimes_A \id_{\alpha(X)})=\overline{S}_X^{-1}(\overline{\xi})$. In particular, if such $\overline{\xi}$ exists, it is unique and 
\begin{align}\label{eqn:barinnerprod2}
    \overline{\xi}=\overline{S}_X(\xi^*) = (\xi^*\otimes \id_{\alpha(\overline{X})})\circ \overline{R}_{X}^{\alpha}.
\end{align}
Therefore it suffices to show that $\overline{\xi}\coloneqq\overline{S}_X(\xi^*)$ satisfies the required conditions. Firstly, for all $\xi,\eta\in\cL(A,\alpha(X))$ it follows from the conjugate equations that
\begin{align*}
    &(R_{X}^{\alpha})^* (\overline{S}_X(\xi^*) \otimes \id_{\alpha(X)}) \eta 
    = 
    (\xi^*\otimes (R_{X}^{\alpha})^*) (\overline{R}_{X}^{\alpha}\otimes \id_{\alpha(X)}) \eta 
    = 
    \xi^*\eta. 
\end{align*}
Moreover, it follows from the rightmost side of \eqref{eqn:muX} that \[\overline{\overline{\xi}}=\overline{S}_{\overline{X}}(\overline{S}_X(\xi^*)^*) = \alpha(\mu_X)\circ\xi.\] As $\overline{S}_X$ is $\cM(A)$-bilinear and $\|\overline{S}_X\|\leq d_X^{1/2}$, it also follows that $\overline{x\rhd \xi\lhd y}=y^*\rhd \overline{\xi}\lhd x^*$ and $\|\overline{\xi}\|\leq d_X^{1/2}\|\xi\|$. 
\end{proof}
\begin{rmk}
    As the bijection between $\cL(A,\alpha(X))$ and $\cL(A,\alpha(\overline{X}))$ of Lemma \ref{lemma:barnotation} sends compact operators to compact operators, it also makes sense to write $\overline{\xi}\in \alpha(\overline{X})$ for any $\xi\in \alpha(X)$. Indeed, one can compose the canonical isomorphisms $\alpha(X)\cong \cK(A,\alpha(X))$ with the construction of Lemma \ref{lemma:barnotation}.
\end{rmk}

\subsection{Cocycle representations revisited}
Next, we introduce an equivalent formulation for $\cC$-cocycle representations between $\cC$-C$^*$-algebras, as defined in \cite{intertwining}, which will play a crucial role in our Cuntz--Thomsen construction of $\KK^{\cC}$. Throughout this subsection we let $(A,\alpha,\fu)$ and $(B,\beta,\fv)$ be $\cC$-C$^*$-algebras for a unitary tensor category $\cC$. We start by recalling the definition of a $\cC$-equivariant structure on a correspondence (see e.g. \cite[Definition 3.2]{AFclass} or \cite[Definition 2.12]{ARKIKU23} and \cite[Definition 3.1]{intertwining} where the same notion appears under different names).
\begin{defn}\label{defn: cocyclerep}
  %Let $(\alpha,\fu)\colon \mathcal{C}\curvearrowright A$ and $(\beta,\fv)\colon\mathcal{C}\curvearrowright B$ be actions of a unitary tensor category $\mathcal{C}$ on C$^*$-algebras $A$ and $B$. 
  Let $\cC$ be a unitary tensor category and $(A, \alpha,\fu)$, $(B, \beta,\fv)$ be $\cC$-C$^*$-algebras. Let $E$ be an $A$-$B$-correspondence. A \emph{$\cC$-equivariant structure} on $E$ consists of a natural family of (not necessarily adjointable) $A$-$B$-bilinear isometries $\{\mathbbm{u}_X\colon \alpha(X)\otimes E\to E\otimes\beta(X)\}_{X\in\cC}$ 
  such that for all $X,Y\in\mathcal{C}$, the following pentagon diagram 
\begin{equation}\label{cocyclemorphismdiagram}
\begin{adjustbox}{max width=\textwidth}
\begin{tikzcd}
& \alpha(X)\otimes \alpha(Y)\otimes E 
\ar{dr}{\fu_{X,Y}\otimes\id_{E}}
\ar[swap]{dl}{\id_{\alpha(X)}\otimes\mathbbm{u}_Y} &    
\\   
\alpha(X)\otimes E \otimes\beta(Y) \ar[swap]{dd}{\mathbbm{u}_X\otimes\id_{\beta(Y)}} 
& &\alpha(X\otimes Y)\otimes E \ar{dd}{\mathbbm{u}_{X\otimes Y}}
\\
& &
\\
E\otimes\beta(X)\otimes \beta(Y)
\ar{rr}{\id_{E}\otimes \fv_{X,Y}} 
&    
& E \otimes\beta(X\otimes Y), 
\end{tikzcd}
\end{adjustbox}
\end{equation}commutes. In this case, we say that $(E,\phi,\mathbbm{u})$ is a $\cC$-$A$-$B$-correspondence, where $\phi$ is the left $A$-action on $E$.
When $E=B$ with the left action given by a $^*$-homomorphism $\phi\colon A\rightarrow \cM(B)$, a pair $\phi$ along with a $\cC$-equivariant structure $\bv_X$ on ${}_\phi B$, denoted simply by $(\phi,\bv)$, is called a \emph{$\cC$-cocycle representation}. We denote the collection of $\cC$-cocycle representations from $(A,\alpha,\fu)$ to $(B,\beta,\fv)$ by $\Hom^{\cC}((\alpha,\fu),(\beta,\fv))$. If $\phi$ as in the previous sentence is in fact valued in $B$ then we call $(\phi,\bv)$ a \emph{$\cC$-cocycle morphism}.
\end{defn}
Note that if $(\phi,\bu)$ is a $\cC$-cocycle representation, then $\mathbbm{u}_{1_\cC}\colon A\otimes {}_{\phi}B\to {}_{\phi} B\otimes B\cong {}_{\phi}B$ is automatically given by $\mathbbm{u}_{1_\cC}(a\otimes x)=\phi(a)x$ for any $a\in A$ and $x\in B$ (see \cite[Remark 2.14]{ARKIKU23}). 
\begin{rmk}\label{rmk:composition}
By \cite[Lemma 3.8, Remark 3.9]{intertwining} one can compose $\cC$-cocycle representations \[(\phi,\bu)\colon (A,\alpha,\fu)\rightarrow (B,\beta,\fv)\ \text{and} \ (\psi,\bv)\colon (B,\beta,\fv)\rightarrow (C,\gamma,\fw)\] whenever $\psi\colon B\rightarrow \cM(C)$ is an extendible $^*$-homomorphism (i.e. that for an approximate unit $e_{\lambda}$ of $B$ the net $\psi(e_{\lambda})$ converges to some projection in $\cM(C)$). We denote this composition by $(\psi,\bv)\circ (\phi,\bu)$.
\end{rmk}
\par In the spirit of \cite[Lemma 3.8]{AFclass} and \cite[Definition A]{intertwining}, we can package the information of a $\cC$-cocycle representation into that of a family of linear maps indexed by the objects in the category. 

%Let $\phi\colon A\rightarrow \cM(B)$ be a $^*$-homomorphism. If $\bu_X\colon \alpha(X)\otimes{}_\phi B\rightarrow {}_\phi B\otimes \beta(X)$ is a collection of isometries making $(\phi,\bu)$ a cocycle representation, then one can associate a family of linear maps 
%\begin{align*}
    %\phi_X\colon \alpha(X)&\rightarrow \cL(B,\beta(X))\\
    %\xi&\mapsto \bu_X T_{\xi}^{\phi},
%\end{align*}
%which are well-defined by \cite[Lemma 2.31 (2)]{ARKIKU23}. Moreover, for any family of bounded linear maps $\phi_X\colon \alpha(X)\rightarrow \cL(B,\beta(X))$ satisfying certain compatibility conditions, we show that one can reconstruct a $\cC$-equivariant structure on the $^*$-homomorphism $\phi$.

\begin{prop}\label{prop:phiXmaps}
Let $(A,\alpha,\fu)$ and $(B,\beta,\fv)$ be $\cC$-C$^*$-algebras for a unitary tensor category $\cC$. Let $\phi\colon A\rightarrow \cM(B)$ be a $^*$-homomorphism. Then the families of bilinear isometries $\bu_X\colon \alpha(X)\otimes{}_{\phi}B\rightarrow {}_\phi B\otimes \beta(X)$ for $X\in \cC$ such that $(\phi,\bu)$ is a $\cC$-cocycle representation are in one-to-one correspondence with families of linear maps \[\{\phi_X\colon \alpha(X)\rightarrow \cL(B,\beta(X))\}_{X\in \cC}\] such that for any $X,Y\in \cC$, $a\in A$, $\xi,\eta\in\alpha(X)$, and $\zeta\in\alpha(Y)$ we have that
\begin{enumerate}[label=\textit{(\roman*)}]
    \item $\phi_{1_{\cC}}(a)=\phi(a)$, \label{item:unit}
    \item $\phi_X(a\rhd\xi)=\phi(a)\rhd \phi_X(\xi)$ and $\phi_X(\xi\lhd a)=\phi_X(\xi)\lhd \phi(a)$, \label{item:bimodular}
    \item $\phi(\langle \eta,\xi\rangle)=\langle \phi_X(\eta),\phi_X(\xi)\rangle$,\label{item:isometric}
    \item $\phi_{X\otimes Y}(\fu_{X,Y} (\xi\otimes \zeta))=\fv_{X,Y}\circ(\phi_X(\xi)\otimes \id_{\beta(Y)})\circ \phi_Y(\zeta)$,\label{item:coherence}
     \item $\phi_Y\circ \alpha(f)=\beta(f)\circ \phi_X$ for any $f\in \Hom(X,Y)$.\label{item:naturality}
    
\end{enumerate}
\end{prop}
\begin{proof}
As much of the proof follows in a similar spirit to \cite[Lemma 3.8]{AFclass} and \cite[Lemma 3.11]{intertwining}, we will omit the detailed computations. Suppose we have a family of linear maps $\{\phi_X\}_{X\in\cC}$ satisfying the conditions above. For $X\in \cC$ we let $\bu_X\colon \alpha(X)\otimes{}_\phi B\rightarrow {}_\phi B\otimes \beta(X)$ be defined by
\begin{equation}
    \bu_X(\xi\otimes b)=(l_{\beta(X)}\circ \phi_X(\xi))(b).
\end{equation}
It follows from \ref{item:bimodular} that this is well-defined and bilinear on the algebraic tensor product $\alpha(X)\odot {}_\phi B$. It follows from \ref{item:isometric} that $\bu_X$ is isometric and so it is also well defined on all of $A\otimes {}_\phi B$ (that $\bu_X$ is isometric follows easily for elementary tensors, for the general case it follows by the same argument as in the proof of \cite[Lemma 3.11]{intertwining}). Naturality of $\bu_X$ immediately follows from \ref{item:naturality} and the naturality of $l$. Similarly, combining the naturality of $l$ and \ref{item:coherence}, the coherence of $\bu_X$ follows. We denote the mapping $\{\phi_X\}\mapsto \{\bu_X\}$ by $\Phi$.
\par Suppose $(\phi,\bu)\colon(A,\alpha,\fu)\rightarrow (B,\beta,\fv)$ is a $\cC$-cocycle representation. We define linear maps $\phi_X\colon \alpha(X)\rightarrow \cL(B,\beta(X))$ by 
\begin{align*}
    \phi_X\colon \alpha(X)&\rightarrow \cL(B,\beta(X))\\
    \xi&\mapsto l_{\beta(X)}^*\bu_X T_{\xi}^{\phi}.
\end{align*}
Firstly, we note that $\phi_X(\xi)$ as defined above is indeed adjointable for all $\xi\in \alpha(X)$ by \cite[Lemma 2.31 (2)]{ARKIKU23}. 
%Indeed, the adjoint is given by $\phi_X(\xi)^*=(\id_{\phi}\otimes R_{X}^{\beta})^*(\bu_{\overline{X}}\otimes \id_{\beta(X)})T^{\phi} _{\overline{\xi}}l_{\beta(X)}$ \todo{Do we actually need the formula of the adjoint? Maybe in Proposition \ref{prop:dualrep}? SG}. 
As for any $a,a'\in A$ and $\xi\in \alpha(X)$ one has $ T_{a\rhd\xi\lhd a'}^{\phi}=(a\rhd)\circ T_{\xi}^{\phi}\circ (a'\rhd)$ and $\bu_X$ is bilinear, \ref{item:bimodular} follows. Condition \ref{item:unit} holds as noted after \cref{defn: cocyclerep}. Conditions \ref{item:naturality} and \ref{item:coherence} follow from the naturality and coherence of $\bu_X$ respectively. Condition \ref{item:isometric} follows as $l_{\beta(X)}$ and $\bu_X$ preserve inner products. That $\Psi$ and $\Phi$ are mutually inverse follows from a straightforward computation.
\end{proof}
In the previous proposition we have denoted the use of the map $l_{\beta(X)}$ and its inverse. We will often omit these from computations hereinafter.
\begin{rmk}\label{rmk: LinearMapsNotation}
For ease of notation, we will denote a $\cC$-cocycle representation by $\phi\colon (A,\alpha,\fu)\to (B,\beta,\fv)$, where $\phi$ stands for the collection of linear maps $\{\phi_X\}_{X\in\cC}$ satisfying the conditions of Proposition \ref{prop:phiXmaps} with respect to the $^*$-homomorphism $\phi_{1_{\cC}}$.
\end{rmk}

\begin{rmk}\label{rmk: GroupCase}
If $(A,\alpha)$ and $(B,\beta)$ are actions of a countable discrete group $G$, then we may canonically identify them with actions of the opposite unitary tensor category of $G$-graded finite dimensional Hilbert spaces denoted $\Hilb(G)^{\op}$. Recall that, as discussed in \cite{intertwining}, a cocycle representation $(\phi,\mathbbm{u})\colon (A,\alpha)\rightarrow (B,\beta)$ in the sense of \cite{cocyclecategszabo,GASZ22,DynamicalKP}, canonically induces a $\Hilb(G)^{\op}$-cocycle representation $(A,\alpha)\rightarrow (B,\beta)$, where by abusing notation we regard $\bu^*$ as the natural family of bilinear isometries determined by ${}_{\alpha_g}A\otimes {}_\phi B\to {}_{\phi\alpha_g}B\xrightarrow{\bu_g^*} {}_{\beta_g\phi}B \cong {}_{\phi}B\otimes {}_{\beta_g}B$ on each $g\in G\cong \Irr(\Hilb(G^{\op}))$. In this case, it follows from a similar calculation to the one in \cite[Example 3.13]{intertwining}, that the associated family of maps to the cocycle representation $(\phi,\bu)$ as in Proposition \ref{prop:phiXmaps} is $\{\phi_g\,|\, g\in G\}$ with the map $\phi_g\colon A\to\mathcal{M}(B)$ for $g\in G$ given by $\phi_g(a)=\mathbbm{u}_g^*\phi(a)$ for any $a\in A$.
\par More generally, it follows from Proposition \ref{prop:phiXmaps} that a $\Hilb(G)^{\op}$-cocycle representation $\phi\colon (A,\alpha)\rightarrow (B,\beta)$ consists of a family of linear maps $\{\phi_g\colon A\rightarrow\cM(B)\}_{g\in G}$ satisfying
\begin{enumerate}
    \item $\phi_1\colon A\rightarrow \cM(B)$ is a $^*$-homomorphism,
    \item $\phi_g(\alpha_g(a_1)a_2)=\beta_g(\phi_1(a_1))\phi_g(a_2)$, $\phi_g(a_1a_2)=\phi_g(a_1)\phi_1(a_2)$,
    \item $\phi_1(a_1^*a_2)=\phi_g(a_1)^*\phi_g(a_2)$,
    \item $\phi_{gh}(\alpha_g(a_1)a_2)=\beta_g(\phi_h(a_1))\phi_g(a_2)$,
\end{enumerate}
for all $a_1,a_2\in A$ and $g,h\in G$.
\end{rmk}

\begin{rmk}\label{rmk:compnondegcocyclereps}
 Let $\theta\colon (B,\beta,\fv)\rightarrow (C,\gamma,\fw)$ be a $\cC$-cocycle representation. We call $\theta$ \emph{non-degenerate} if $\theta_{1_{\cC}}\colon B\rightarrow \cM(C)$ is a non-degenerate $^*$-homomorphism. Firstly, through the canonical identification of the correspondences $\cK(B,\beta(X))$ and $\beta(X)$ one can see $\theta_X$ as a linear map $\cK(B,\beta(X))\rightarrow \cL(C,\gamma(X))$. Using the non-degeneracy of $\theta_{1_{\cC}}$, for each $X\in \cC$, one can extend $\theta_X$ to a linear map $\overline{\theta}_X\colon \cL(B,\beta(X))\rightarrow \cL(C,\gamma(X))$ by $\overline{\theta}_X(\xi)(\theta_{1_{\cC}}(b)c)=\theta_X(\xi(b))(c)$. That this is well-defined follows from a standard argument using both approximate units and the continuity of $\theta_X$. Moreover, it follows that the extension $\overline{\theta}_X$ satisfies $\theta_{1_{\cC}}(b)\rhd \overline{\theta}_X(\xi)=\overline{\theta}_X(b\rhd \xi)$ and $\overline{\theta}_X(\xi)\lhd \theta_{1_{\cC}}(b)=\overline{\theta}_X(\xi\lhd b)$ for all $\xi\in \cL(B,\beta(X))$ and $b\in B$ and is the unique linear extension of $\theta_X$ satisfying these conditions. Hereinafter, we will denote $\overline{\theta}_X$ also by $\theta_X$.
    \par Let $\phi\colon (A,\alpha,\fu)\rightarrow (B,\beta,\fv)$ be another (possibly degenerate) $\cC$-cocycle representation. It follows from the proof of Proposition \ref{prop:phiXmaps} that $(\theta\circ \phi)_X(\xi)(c)=\theta_X(\phi_X(\xi))(c)$ for any $X\in \cC$ and $c\in C$, where we've extended $\theta_X$ to all of $\cL(B,\beta(X))$ in the right hand side.
\end{rmk}
In the following example we will canonically identify $\cK(B,E)$ with $E$ for a Hilbert $B$-module $E$.
\begin{example}\label{example: cocyclerep}
    Let $(A,\alpha,\fu)$ be a $\cC$-C$^*$-algebra. 
    We will frequently use $\cC$-cocycle representations of the following forms. 
    \begin{enumerate}
        \item 
        For an isometry $v\in\cM(A)$, the family of maps 
        $\alpha(X)\ni \xi\mapsto v\rhd \xi\lhd v^* \in \alpha(X)$ over $X\in\cC$ satisfies the conditions of \cref{prop:phiXmaps} and thus gives a $\cC$-cocycle representation denoted by $\Ad(v)\colon (A,\alpha,\fu)\to (A,\alpha,\fu)$. 
        \item 
        For C$^*$-algebras $B,C$ and a $^*$-homomorphism $\phi\colon B\to \cM(C)$, the family of maps $\{ \id_{\alpha(X)}\otimes \phi \}_{X\in\cC}$ satisfies the conditions of \cref{prop:phiXmaps} and thus gives a $\cC$-cocycle representation denoted by $\id_A\otimes\phi\colon (A\otimes B,\alpha\otimes \id,\fu\otimes 1)\to (A\otimes C,\alpha\otimes \id,\fu\otimes 1)$. 
        When $\phi$ is $\ev_t\colon C[0,1]\to\bC$, the evaluation at a given point $t\in[0,1]$, sometimes we still write $\ev_t$ to indicate the $\cC$-cocycle representation $\id_B\otimes\ev_t\colon (B\otimes C[0,1],\beta\otimes\id,\bv\otimes1)\to (B,\beta,\bv)$ by abusing notation. 
    \end{enumerate}
\end{example}
%If $\phi\colon (A,\alpha,\fu)\rightarrow (B,\beta,\fv)$ and $\psi\colon (B,\beta,\fv)\rightarrow (C,\gamma,\fw)$ are cocycle representations with $\psi_{1_{\cC}}$ extendible then we denote their composition by $\psi\circ\phi$ (this composition descends from the composition of cocycle representations when viewed as pairs $(\phi,\bu)$ and $(\psi,\bv)$ see Remark \ref{rmk:composition}). 
Two $\cC$-cocycle representations $\phi,\psi\colon (A,\alpha,\fu)\rightarrow (B,\beta,\fv)$ are called \emph{unitarily equivalent}, if there exists a unitary $u\in \cM(B)$ with $\Ad(u)\circ\phi=\psi$. We end this subsection by collecting some further properties of $\cC$-cocycle representations.

\begin{prop}\label{prop:dualrep}
    Let $(A,\alpha,\fu)$ and $(B,\beta,\fv)$ be $\cC$-C$^*$-algebras for a unitary tensor category $\cC$. Let $\phi\colon (A,\alpha,\fu)\to (B,\beta,\fv)$ be a $\cC$-cocycle representation. Then one has that
    \[\overline{\phi_X(\xi)}=\phi_{\overline{X}}(\overline{\xi})\] for any $X\in\cC$ and $\xi\in\alpha(X)$.
\end{prop}
\begin{proof}
    As $\alpha(X)$ is non-degenerate, we only have to show the desired equality for elements of the form $a\rhd \xi$ for $X\in \cC$, $\xi\in\alpha(X)$ and $a\in A$. We have that $\overline{\phi_X(a\rhd \xi)}=\overline{\phi_X(\xi)}\lhd\phi_{1_\cC}(a^*)$ and $\phi_{\overline{X}}(\overline{a\rhd \xi})=\phi_{\overline{X}}(\overline{\xi}\lhd a^*)$. 
    Thus, it suffices to show the equality 
    \begin{align}\label{eqn:prop:dualrep}
        &
        \overline{\phi_X(\xi)}(\phi_{1_\cC}(a^*)b) = \phi_{\overline{X}}(\overline{\xi} \lhd a^*)(b)
    \end{align}
    in $\beta(\overline{X})$ for all $b\in B$. For the remainder of the proof we let $\bu_X$ be the family of isometries corresponding to $\phi_X$ (see \cref{prop:phiXmaps}) and denote  $\phi_{1_{\cC}}$ simply by $\phi$. By naturality and coherence of $\bu_X$ we have the following commutative diagram
    \begin{equation*}
    \begin{tikzcd}[cells={nodes={font=\small}}]
    {A\otimes}{}_\phi B \arrow[d,"\alpha(\overline{R}_X)\otimes\id_{\phi}"] \arrow[r,"\bu_{1_{\cC}}"] &\overline{\phi(A) B} \otimes B \arrow[d,"\id_{\phi} \otimes\beta(\overline{R}_X)"]\ar[r,"l_{B}"]& B\ar[dd,"R_X^{\beta}"]
    \\
    \alpha(X\otimes \overline{X})\otimes{}_\phi B \arrow[r,"\bu_{X\otimes\overline{X}}"] \arrow[dd,"\fu_{X,\overline{X}}^*\otimes \id_{\phi}"']&
    \overline{\phi(A) B} \otimes \beta(X\otimes \overline{X})\arrow[d,"\id_{\phi}\otimes \fv_{X,\overline{X}}^*"]
    \\
    & 
    \overline{\phi(A) B} \otimes \beta(X)\otimes\beta(\overline{X}) \arrow[d,"\bu_{X}^*\otimes\id_{\beta(\overline{X})}"]\ar[r,"l_{\beta(X)\otimes \beta(\overline{X})}"]&\beta(X)\otimes \beta(\overline{X})\ar[dd,"\phi_X(\xi)^*\otimes \id_{\beta(\overline{X})}" swap]
    \\
    \alpha(X)\otimes\alpha(\overline{X})\otimes{}_\phi B \arrow[d,"T_\xi^*"'] \arrow[r,"\id_{\alpha(X)}\otimes\bu_{\overline{X}}"]&
    \alpha(X)\otimes {}_\phi B\otimes\beta(\overline{X}) \arrow[d,"T_\xi^*"]&
    \\
    \alpha(\overline{X})\otimes {}_\phi B \arrow[r,"\bu_{\overline{X}}"]&
    B\otimes\beta(\overline{X}) \ar[r,equal]& B\otimes \beta(\overline{X})
    \end{tikzcd}
    \end{equation*}
    %\begin{equation*}
    %\begin{tikzcd}
    %{A\otimes}{}_\phi B \arrow[r,"\alpha(\overline{R}_X)\otimes1"] \arrow[d,"\bu_1"] &
    %\alpha(X\otimes \overline{X})\otimes{}_\phi B \arrow[ddl,"\bu_{X\otimes\overline{X}}"] \arrow[r,"\fu_{X,\overline{X}}^*\otimes 1"]&
    %\alpha(X)\otimes\alpha(\overline{X})\otimes{}_\phi B \arrow[dd,"T_\xi^*"] \arrow[dddl,"1\otimes\bu_{\overline{X}}"] \\
    %\overline{\phi(A) B} \otimes B \arrow[d,"\beta(\overline{R}_X)"]\\
    %\overline{\phi(A) B} \otimes \beta(X\otimes \overline{X})\arrow[d,"\id_{\phi}\otimes \fv_{X,\overline{X}}^*"] && \alpha(\overline{X})\otimes {}_\phi B \arrow[d,"\bu_{\overline{X}}"]
    %\\
    %\overline{\phi(A) B} \otimes \beta(X)\otimes\beta(\overline{X}) \arrow[r,"\bu_{X}^*\otimes1"] &
    %\alpha(X)\otimes {}_\phi B\otimes\beta(\overline{X}) \arrow[r,"T_\xi^*"] & B\otimes\beta(\overline{X}) 
    %\end{tikzcd}
    %\end{equation*}
    where we make sense of $\bu_X^*\otimes \id_{\beta(\overline{X})}$ in the diagram above since $\bu_X$ is a unitary isomorphism from $\alpha(X)\otimes{}_\phi B$ to the closed linear span of $\phi(A)\rhd (B\otimes \beta(X))$ (see \cite[Lemma 2.31 (1)]{ARKIKU23}). Now, the image of $a^*\otimes b \in A\otimes {}_\phi B$ in $\beta(\overline{X})$ via the lower left composition in the diagram followed by the identification of $B\otimes \beta(\overline{X})$ with $\beta(\overline{X})$ coincides with
    \begin{align*}
        \bu_{\overline{X}}(T_\xi^*R_X^{\alpha}(a^*)\otimes b)&=  \bu_{\overline{X}}(T_{\overline{\xi}\lhd a^*}(b))\\
        &=\phi_{\overline{X}}(\overline{\xi}\lhd a^*)(b).
    \end{align*}
    and its image via the upper rightmost composition followed again by the identification of $B\otimes \beta(\overline{X})$ with $\beta(\overline{X})$ coincides by definition with $\overline{\phi_X(\xi)}(\phi(a^*)b)$. Thus \eqref{eqn:prop:dualrep} holds as desired. 
\end{proof}

This yields the following relation that will be used later. 

\begin{lemma}\label{lem:usingisometryofstar}
     Let $(A,\alpha,\fu)$ and $(B,\beta,\fv)$ be $\cC$-C$^*$-algebras for a unitary tensor category $\cC$. Let $\phi\colon (A,\alpha,\fu)\rightarrow (B,\beta,\fv)$ be a $\cC$-cocycle representation.\ Then for any $x,y\in \cM(B)$, $X\in \cC,$ and $\xi\in \alpha(X)$ we have that
     \[x\rhd \phi_{X}(\xi)\lhd y^*=\beta(\mu_X^{-1}) \circ \bigl( \overline{y\rhd \phi_{\overline{X}}(\overline{\xi})\lhd x^*} \bigr).\] 
\end{lemma}
\begin{proof}
It follows from \cref{lemma:barnotation,prop:phiXmaps,prop:dualrep} that
 \begin{align*}
     \overline{y\rhd \phi_{\overline{X}}(\overline{\xi})\lhd x^*}
     &=x\rhd \overline{\overline{\phi_{X}(\xi)}}\lhd y^*\\
     &=x\rhd (\beta(\mu_X)\circ \phi_{X}(\xi) )\lhd y^*\\
     &=\beta(\mu_X)\circ (x\rhd \phi_{X}(\xi) \lhd y^*)
 \end{align*}   
 where in the last line we have used that $\beta(\mu_X)$ is left $B$-linear and bounded, so commutes with the left action of $\cM(B)$.
\end{proof}

\subsection{Paschke dual type algebras}

This subsection introduces $\cC$-equivariant variants of the classical Paschke dual construction commonly appearing in the study of $\KK$-theory (see for example \cite[Definition 5.1.1]{HigsonRoe}, \cite[Section 3]{Tho01}, \cite[Section 6]{THO05}).
\begin{rmk}\label{rmk: CommutantNotation}
    Let $E$ be an $A$-$A$-correspondence. 
    For $\xi\in \cL(A,E)$ and $a\in \cM(A)$, we write $[a,\xi]\coloneqq a\rhd \xi - \xi\lhd a \in \cL(A,E)$ and $[\xi,a]\coloneqq -[a,\xi]$. 
    Also, for $u\in \cU(\cM(A))$, we write $\Ad(u)(\xi)\coloneqq u\rhd \xi\lhd u^* \in \cL(A,E)$. 
    For a subspace $S\subset E$, we denote by $[a,S]=\{ [a, \xi] \,|\, \xi\in S\}$ for $a\in A$. 
\end{rmk}
\begin{lemma}\label{lem:fD*alg}
Let $A$ be a $\cC$-C$^*$-algebra and for $X\in \cC$ let $D_X\subset \cL(A,\alpha(X))$ be a subspace such that $\overline{D_X}\subset D_{\overline{X}}$. 
Then 
\[\fD_D \coloneqq \{ a\in \cM(A) \,|\, [a,D_X]\subset \cK(A,\alpha(X)), \forall X\in\cC \}\]
is a unital C$^*$-subalgebra of $\cM(A)$. %\footnote{For an $A$-$A$ C$^*$-correspondence $E$ and a subspace $S\subset E$ we denote by $[a,S]=a\rhd \xi-\xi\lhd a$ for $a\in A$ and $\xi\in D$.}
\end{lemma}

\begin{proof}
A standard check shows that $\fD_D$ is a norm-closed unital subalgebra. To see that $\fD_D$ is self-adjoint, for $a\in \fD_D$, $X\in \cC$, and $\xi\in D_{X}$, using $\overline{\xi}\in D_{\overline{X}}$ and Lemma~\ref{lemma:barnotation}, we see that
\begin{align*}
[a^*,\xi] &= a^*\rhd \xi - \xi\lhd a^* = \alpha(\mu_X^{-1}) \Bigl( \overline{\overline{\xi} \lhd a} - \overline{a \rhd \overline{\xi}} \Bigr) \\
&\in \alpha(\mu_X^{-1}) \circ \cK(A,\alpha(\overline{\overline{X}}))=\cK(A,\alpha(X)). 
\qedhere\end{align*}
\end{proof}

\begin{defn}\label{defn: DPhi}
For $\cC$-C$^*$-algebras $(A,\alpha,\fu)$, $(B,\beta,\fv)$ and a $\cC$-cocycle representation $\phi\in\Hom^\cC((\alpha,\fu),(\beta,\fv))$, we define a unital C$^*$-subalgebra of $\cM(B)$ by
\[\fD_{\phi}\coloneqq\{ x\in \cM(B) \,|\, [x,\phi_X(\alpha(X))]\subset\cK(B,\beta(X)) , X\in\cC \}. \]
\end{defn}
Note that this is a well-defined C$^*$-subalgebra by \cite[Lemma 3.7]{ARKIKU23}, which is also a particular case of \cref{lem:fD*alg} for $D_X\coloneqq \phi_X(\alpha(X))$ in view of \cref{prop:dualrep} above. 
\begin{rmk}\label{rmk: PaschkeDualGroup}
If $(A,\alpha)$ and $(B,\beta)$ are actions of a countable discrete group $G$, and $(\phi,\mathbbm{u})$ is a cocycle representation in the sense of \cite{cocyclecategszabo}, then \[\fD_{\phi}\coloneqq\{ x\in \cM(B) \,|\,\beta_g(x)\mathbbm{u}_g^*\phi(a)-\mathbbm{u}_g^*\phi(a)x\in B, \ g\in G, \ a\in A\}.\] Putting $g=1$, we see that $[x,\phi(A)]\subseteq B$, so we get that \[\fD_{\phi}\coloneqq\{ x\in \cM(B) \,|\,(\Ad(\bu_g)\beta_g(x)-x)\phi(a)\in B, \ g\in G, \ a\in A\}.\] 
\end{rmk}

%For a cocycle representation $\phi\colon (A,\alpha,\fu)\rightarrow (B,\beta, \fv)$ we write $\overline{\phi}_X(\xi)=\overline{\phi_X(\xi)}\in \cL(B,\beta(\overline{X}))$ for $X\in\cC$ and $\xi\in \alpha(X)$. 

We finish this subsection with a lemma about approximate units that will play a key role in Sections \ref{sec:absorbing} and \ref{sect: AsymptEquiv}. Essentially, it is the $\cC$-equivariant version of \cite[Lemma 1.14]{KA88} and  follows from \cite[Lemma B.2]{ARKIKU23}. We shall include its proof for the convenience of the reader.

\begin{lemma}\label{lem:approxunit}
    Let $(B,\beta,\fv)$ be a $\sigma$-unital $\cC$-C$^*$-algebra for a unitary tensor category $\cC$ with countably many isomorphism classes of simple objects and $D_X\subset\cL(B,\beta(X))$ be a separable closed linear subspace for each $X\in\Irr(\cC)$. Then there exists an increasing approximate unit $e_n\in B$ such that 
    \begin{enumerate}
        \item $\|[e_n,\xi]\|\longrightarrow 0$ for all $X\in \Irr(\cC)$, $\xi\in D_X$. 
        \item $e_n\rhd T\longrightarrow T$ and $T\lhd e_n\longrightarrow T$ for all $T\in \cK(B,\beta(X))$, $X\in \Irr(\cC)$.
    \end{enumerate}
\end{lemma}

\begin{proof}
    For $n\in\bN$, take $X_n\in\Irr(\cC)$ and $\xi_n \in D_{X_n}$ such that $\|\xi_n\|\xrightarrow{n\to\infty}0$ and the linear span of $\{ \xi_n \,|\, n\in\bN, X_n=X \}$ is norm-dense in $D_X$ for all $X\in\Irr(\cC)$. Let $A\coloneqq \bigoplus_{X\in\Irr(\cC)} \cK(\beta(X)\oplus B)$ and $a=(a_X)_{X\in\Irr(\cC)}\in A$ be a strictly positive contraction which exists as $B$ is $\sigma$-unital. By writing $\beta_X\colon B\to \cK(\beta(X))$ for the non-degenerate $^*$-homomorphism of the left action, consider the bounded linear map 
    $\Phi\colon \cM(B) \to (c_0(\bN)\otimes A) \oplus A$ defined by \begin{align*}
        &\Phi(x)\coloneqq \left(\left(\begin{array}{cc}
             0& \delta_{X,X_n}[x,\xi_n] \\
             0& 0
        \end{array}\right)_{X\in\Irr(\cC),n\in\bN} , (a_X-\beta_{X\oplus1_\cC}(x)a_X)_{X\in\Irr(\cC)} \right) . 
    \end{align*}
    It follows from the non-degeneracy of each $\beta_{X\oplus1_\cC}$ that $\Phi$ is strictly continuous on the unit ball of $\cM(B)$. 
    Recall that for any positive contraction $b\in B$ such that $\|b\|<1$, the set $E_b\coloneqq \{ e\in B \,|\, b\leq e, \|e\|<1 \}$ is an approximate unit of $B$. 
    Since $\Phi(1_{\mathcal{M}(B)})=0$ and any state on a C$^*$-algebra is strictly continuous on bounded sets,\footnote{This is a consequence of \cite[Corollary 5.7]{Hilbertmodules}. Note that strictness in this result is implied from the fact that the image of an approximate unit under a state converges to 1 (see \cite[Theorem 3.3.3]{MU90}).} it follows that $0$ is contained in the closure $\Phi(E_b)$ with respect to the weak topology of the Banach space $(c_0(\bN)\otimes A)\oplus A$ for any positive $b\in B$ with $\|b\|<1$. Moreover, as $E_b$ is convex it follows from the Hahn--Banach separation theorem that the weak closure and norm closure of $\Phi(E_b)$ coincide, so $0$ is contained in the norm closure of $\Phi(E_b)$.
    Therefore, starting from $e_0=0$, we can inductively construct a sequence of positive contractions $e_n\in B$ such that $\|e_n\|<1$, $e_{n-1}\leq e_n$, and $\|\Phi(e_n)\|<n^{-1}$. 
    Now it is routine to check this family $\{ e_n \}_{n\in \bN}$ is the desired approximate unit. 
\end{proof}

%We typically apply this lemma when $D_X$ contains $\phi_X(\alpha(X))$ for a cocycle representation $\phi\colon (A,\alpha,\fu)\to (B,\beta,\fv)$. 

\section{Cuntz--Thomsen pictures for \texorpdfstring{$\KK^\cC$}{KKC}}\label{sect:Cuntzpicture}

The $\cC$-equivariant Kasparov KK-theory was introduced in \cite{ARKIKU23}. In this section, we introduce a Cuntz--Thomsen picture for $\cC$-equivariant KK-theory. We start by recalling the crucial definitions of the Fredholm module picture for $\KK^{\cC}$. 
\begin{defn}\label{defn:Kasparovmodules}
    Let $(A,\alpha,\fu)$ be a separable $\cC$-C$^*$-algebra and $(B,\beta,\fv)$ be a $\sigma$-unital $\cC$-C$^*$-algebra. 
    \begin{enumerate}
    \item 
    A $\cC$-Kasparov $A$-$B$-bimodule is a quadruple $(E,\phi,\bv,F)$, where $(E,\phi,\bv)$ is a $\cC$-$A$-$B$-correspondence  of the form $(E_+,\phi_+,\bu_+)\oplus (E_-,\phi_-,\bv_-)$ and $F\in \cL(E)$ is an oddly graded operator with respect to $E=E_+\oplus E_-$ such that \[(F^2-1)\phi(A),\ (F-F^*)\phi(A)\subset \cK(E)\] and $[F,T_\xi]\coloneqq (F\otimes \id_{\beta(x)})\bv_X(\xi\otimes(-)) - \bv_X(\xi\otimes F(-))\in \cK(E,E\otimes_B\beta(X))$ for all $X\in\cC$ and $\xi\in\alpha(X)$.
    \item 
    We write $\bfE^\cC(A,B)$ for the class of $\cC$-Kasparov $A$-$B$-bimodules and $\simeq$ for the equivalence relation on $\bfE^\cC(A,B)$ generated by $\cC$-equivariant homotopy and $\cC$-equivariant unitary equivalence (i.e., $E_0,E_1\in\bfE^\cC(A,B)$ satisfy $E_0\simeq E_1$ if and only if there is some $E\in\bfE^\cC(A,B\otimes C[0,1])$ whose evaluations at $0$ and $1$ are $\cC$-equivariantly unitary equivalent to $E_0$ and $E_1$, respectively). Then the $\cC$-equivariant Kasparov group is $\KK^\cC(A,B)\coloneqq \bfE^\cC(A,B)/{\simeq}$.
    \end{enumerate}
\end{defn}
In the following, we shall write $\bfE^\cC((\alpha,\fu),(\beta,\fv))$ and $\KK^{\cC}((\alpha,\fu),(\beta,\fv))$ to emphasise the actions of $\cC$.

\begin{defn}\label{defn:Cuntzpair} Let $(A,\alpha,\fu)$ and $(B,\beta,\fv)$ be two $\cC$-C$^*$-algebras with $A$ separable and $B$ $\sigma$-unital. We call $\left(\phi,\psi\right)$ a $\cC$-Cuntz pair from $(\alpha,\fu)$ to $(\beta,\fv)$ if $\phi,\psi\in \Hom^{\cC}((\alpha,\fu),(\beta^\rs,\fv^\rs))$ and 
\[\phi_X(\xi)-\psi_X(\xi)\in \cK(B^\rs,\beta^\rs(X)),\
 \forall \ \xi\in \alpha(X), \ X\in \cC.\] 
A $\cC$-Cuntz pair of the form $\left(\phi,\phi\right)$ will be called \emph{degenerate}. 
\end{defn}

\begin{rmk}\label{rmk: CuntzPairsGroups}
Let $(A,\alpha)$ and $(B,\beta)$ be actions of a countable discrete group $G$, and $(\phi,\mathbbm{u}),(\psi,\mathbbm{v})$ be cocycle representations in the sense of \cite{cocyclecategszabo}. Then $((\phi,\mathbbm{u}^*),(\psi,\mathbbm{v}^*))$ is a $\Hilb(G)^{\op}$-Cuntz pair from $(\alpha,1)$ to $(\beta,1)$ if \[\mathbbm{u}_g^*\phi(a)-\mathbbm{v}_g^*\psi(a)\in B\otimes \bK\] for any $g\in G$ and $a\in A$. By taking adjoints this is precisely that
\[\phi(a)\bu_g-\psi(a)\bv_g\in B\otimes \bK\]
for all $a\in A$ and $g\in G$.
\end{rmk}
%\todo{I feel we should discuss homotopy more carefully because usually we check it by constructing a strictly continuous path of Cuntz pairs $(\phi_X^{(t)},\psi_X^{(t)})$. We should maybe discuss how a strictly continuous path of cocycle representations induces a genuine cocycle representation $\Phi_X\colon (A,\alpha)\rightarrow (B^\rs,\beta^\rs)$ maybe referencing the conditions in Proposition \ref{prop:phiXmaps}. SG}

We denote the collection of $\cC$-Cuntz pairs from $(\alpha,\fu)$ to $(\beta,\fv)$ by $\bE^{\cC}((\alpha,\fu),(\beta,\fv))$. We now define an equivalence relation on $\bE^{\cC}((\alpha,\fu),(\beta,\fv))$. For elements $x_0,x_1\in\bE^{\cC}((\alpha,\fu),(\beta,\fv))$, we write $x_0\simeq x_1$ if there exists $y\in \bE^{\cC}((\alpha,\fu),(\beta\otimes \id_{C[0,1]},\fv\otimes1))$ which restricts to $x_0$ upon evaluation at $0$ and to $x_1$ upon evaluation at $1$. i.e. that if $y=(\Phi,\Psi)$ then $(\ev_0\circ \Phi,\ev_0\circ \Psi)=x_0$ and $(\ev_1\circ \Phi,\ev_1\circ \Psi)=x_1$ where $\ev_t\colon B\otimes C[0,1]\rightarrow B$ for $t\in [0,1]$ is the non-degenerate $\cC$-cocycle representation discussed in Example \ref{example: cocyclerep} (in fact it is a cocycle morphism in the sense of \cite[Definition A]{intertwining}). 

We now define an additive operation on the set of $\cC$-Cuntz pairs from $(\alpha,\fu)$ to $(\beta,\fv)$. Let $\phi,\psi\in\Hom^{\cC}((\alpha,\fu),(\beta^\rs,\fv^\rs))$ be two $\cC$-cocycle representations and let $s_1,s_2\in \cM(B^\rs)$ be two isometries which generate a copy of $\cO_2$. We may define a $\cC$-cocycle representation \[\phi\oplus_{s_1,s_2}\psi\coloneqq s_1\rhd\phi\lhd s_1^*+s_2\rhd\psi\lhd s_2^*.\] Moreover, the $\cC$-cocycle representation $\phi\oplus_{s_1,s_2}\psi$ does not depend on the choice of isometries $s_1$ and $s_2$ up to unitary equivalence. Indeed, if the pairs of isometries $s_1,s_2$ and $t_1,t_2$ both generate a copy of $\mathcal{O}_2$ in $\cM(B^\rs)$, the unitary $u=t_1s_1^*+t_2s_2^*\in \cM(B^\rs)$ induces a unitary equivalence between the $\cC$-cocycle representations $\phi\oplus_{s_1,s_2}\psi$ and $\phi\oplus_{t_1,t_2}\psi$.

\begin{defn}\label{defn:CuntzDirectSum}
 Let $(A,\alpha,\fu)$ and $(B,\beta,\fv)$ be $\cC$-C$^*$-algebras with $A$ separable and $B$ $\sigma$-unital, and let $(\phi_1,\psi_1), (\phi_2,\psi_2) \in \bE^\cC((\alpha,\fu),(\beta,\fv))$. Let $s_1,s_2\in\mathcal{M}(B^\rs)$ be two isometries generating a copy of $\mathcal{O}_2$ in $\cM(B^\rs)$. Then, we define the Cuntz sum of $(\phi_1,\psi_1)$ and $(\phi_2,\psi_2)$ by \[(\phi_1,\psi_1)\oplus_{s_1,s_2} (\phi_2,\psi_2)\coloneqq (\phi_1\oplus_{s_1,s_2}\phi_2, \psi_1\oplus_{s_1,s_2}\psi_2).\]
\end{defn}
\begin{rmk}
If $(A,\alpha)$ and $(B,\beta)$ are actions of a countable discrete group $G$, with $B$ a stable C$^*$-algebra, and $(\phi,\bu),(\psi,\bv)\colon (A,\alpha)\rightarrow (B,\beta)$ are cocycle representations in the sense of \cite{cocyclecategszabo}, then their $\Hilb(G)^{\op}$-Cuntz sum $(\phi,\bu^*)\oplus_{s_1,s_2}(\psi,\bv^*)$ with respect to $s_1,s_2\in \cM(B)$ corresponds to the cocycle representation $(\phi\oplus_{s_1,s_2}\psi,\bu\oplus_{s_1,s_2}\bv)$ with 
\begin{equation}\phi\oplus_{s_1,s_2}\psi(a)=s_1\phi(a)s_1^*+s_2\psi(a)s_2^*,\end{equation} 
\begin{equation}\label{eqn:nonfixedcuntzsum}(\bu\oplus_{s_1,s_2}\bv)_g=s_1\bu_g\beta_g(s_1^*)+s_2\bv_g\beta_g(s_2^*).
\end{equation} 
Which is precisely the Cuntz sum as in \cite[Definition 1.6]{GASZ22} when $s_1,s_2\in \cM(B)^{\beta}$.
\par The reason that the authors may assume that the isometries used to form a Cuntz sum are in the fixed point algebra is that by \cite[Proposition 1.4]{DynamicalKP} the action $\beta$ is cocycle conjugate to $\beta\otimes \id_{\bK}$. Let $(\Phi,\bU)\colon (B,\beta)\rightarrow (B\otimes\bK,\beta\otimes \id_{\bK})$ be a cocycle conjugacy and $s_1,s_2\in \cM(\bK)\subset \cM(B\otimes \bK)^{\beta\otimes \id_{\bK}}$. Then one may pullback the Cuntz sum $(\Phi,\bU)\circ(\phi,\bu)\oplus_{s_1,s_2}(\Phi,\bU)\circ (\psi,\bv)$ to a cocycle representation from $(A,\alpha)$ to $(B,\beta)$ by postcomposing with $(\Phi^{-1},\Phi^{-1}(\bU))$; we denote the resulting cocycle representation by $(\varphi,\bw)$. It is a straightforward computation that $\varphi(a)=\Phi^{-1}(s_1)\phi(a)\Phi^{-1}(s_1^*)+\Phi^{-1}(s_2)\phi(a)\Phi^{-1}(s_2^*)$ for all $a\in A$ and that $\bw_g=\Phi^{-1}(s_1)\bu_g\Phi^{-1}(\bU_gs_1^*\bU_g^*)+\Phi^{-1}(s_2)\bv_g\Phi^{-1}(\bU_gs_2^*\bU_g^*)$. As $(\Phi,\bU)$ is a cocycle representation and $s_1,s_2$ are fixed by $\beta\otimes \id_{\bK}$ it follows that $\beta_g(\Phi^{-1}(s_i))=\Ad(\Phi^{-1}(\bU_g))\Phi^{-1}(s_i)$ for $i=1,2$ and so $(\varphi,\bw)$ is precisely $(\phi\oplus_{\Phi^{-1}(s_1),\Phi^{-1}(s_2)}\psi,\bu\oplus_{\Phi^{-1}(s_1),\Phi^{-1}(s_2)}\bv)$ as in \ref{eqn:nonfixedcuntzsum}.

\end{rmk}
\begin{rmk}\label{rmk:sums}
    Let $s_1,s_2\in \cM(B^\rs)$ be two isometries which generate a copy of $\cO_2$.
    In the alternative picture, where $\cC$-cocycle representations are given by pairs $(\phi,\bu),(\psi,\bv)\colon (A,\alpha,\fu)\rightarrow (B^\rs,\beta^\rs,\fv^\rs)$, their direct sum $(\phi,\bu)\oplus_{s_1,s_2}(\psi,\bv)$ coincides with the pair $(\Ad(s_1)\phi+\Ad(s_2)\psi,(s_1\otimes \id_{\beta(X)})\bu_X (\id_{\alpha(X)}\otimes s_1^*)+(s_2\otimes \id_{\beta(X)})\bv_X (\id_{\alpha(X)}\otimes s_2^*))$. Where for any $X\in \cC$, 
    \begin{align*}
    & g_X\circ\left(\bu_X\oplus \bv_X\right)\circ f_X\\
    &= (s_1\otimes \id_{\beta(X)})\bu_X (\id_{\alpha(X)}\otimes s_1^*)+(s_2\otimes \id_{\beta(X)})\bv_X (\id_{\alpha(X)}\otimes s_2^*)
    \end{align*} with
\begin{align*}
f_X\colon \alpha(X)\otimes {}_{\Ad(s_1)\phi+\Ad(s_2)\psi} B^s\rightarrow \alpha(X)\otimes ({}_\phi B^s\oplus {}_\psi B^s)
\end{align*} given by $f_X(\xi\otimes b)= \xi\otimes (s_1^*b,s_2^*b)$ for any $\xi\in\alpha(X)$ and any $b\in B^\rs$ and
\begin{align*}
    g_X\colon ({}_\phi B^s\oplus {}_\psi B^s)\otimes \beta(X)\rightarrow {}_{\Ad(s_1)\phi+\Ad(s_2)\psi} B^s\otimes \beta(X)
\end{align*} given by $g_X\left((b,b')\otimes \eta\right)= (s_1b+s_2b')\otimes \eta$
for any $b,b'\in B^s$ and any $\eta\in \beta(X)$. 
\end{rmk}
The following lemma follows exactly as in the group equivariant case (see for example \cite[Lemma 1.10]{GASZ22}). In fact, its content is an immediate consequence of Theorem \ref{thm:Cuntzpicture}.
\begin{lemma}\label{lemma: CuntzSumGroup}
 Let $(A,\alpha,\fu)$ and $(B,\beta,\fv)$ be $\cC$-C$^*$-algebras with $A$ separable and $B$ $\sigma$-unital.
\begin{enumerate}[label=\textit{(\roman*)}]
    \item Any degenerate $\cC$-Cuntz pair from $(\alpha,\fu)$ to $(\beta,\fv)$ is homotopic to the zero pair.\label{item: DegHom0}
    \item The set $\bE^\cC((\alpha,\fu),(\beta,\fv))/{\simeq}$ becomes an abelian group when equipped with the Cuntz sum.\label{item: CuntzGrp}
\end{enumerate}
\end{lemma}

We want to show that the abelian group $\bE^{\cC}((\alpha,\fu),(\beta,\fv))/{\simeq}$ is isomorphic to $\KK^{\cC}((\alpha,\fu),(\beta,\fv))$. To show this we will require an equivalent but slight variation of Definition \ref{defn:Cuntzpair}. Firstly, recall from Proposition \ref{prop:phiXmaps} that the data of a $\cC$-cocycle representation $\phi=(\phi_X)_X\colon(A,\alpha,\fu)\rightarrow (B^\rs,\beta^\rs,\fv^\rs)$ coincides with the data $(\phi,\bu)$ of a $^*$-homomorphism $\phi\colon A\rightarrow \cM(B^\rs)$ and a $\cC$-equivariant structure $\bu$ on the correspondence ${}_{\phi} B^\rs$ with respect to the actions $(\alpha,\fu)$ and $(\beta^\rs,\fv^\rs)$. However, $(\phi,\bu)$ also induces a $\cC$-equivariant structure on the correspondence ${}_\phi B^{\infty}$ through the composition
\begin{equation}\label{eqn:BstoBinf}
\begin{tikzcd}
    \alpha(X)\otimes {}_\phi B^{\infty}\ar{d}{\id_{\alpha(X)}\otimes l_{B^{\infty}}} \arrow[rr, "\widehat{\bu}_X"]& &{}_\phi B^{\infty}\otimes \beta(X)\\
    \alpha(X)\otimes {}_\phi B^\rs\otimes B^{\infty}\ar{r}{\bu_X} & {}_\phi B^\rs\otimes \beta^\rs(X)\otimes B^{\infty}\ar{r} & {}_\phi B^\rs\otimes \beta(X)^{\infty}\ar{u}{}
\end{tikzcd}
\end{equation}
where the last map is the canonical unitary isomorphism $\beta(X)^{\infty}\cong B^{\infty}\otimes \beta(X)$ followed by $l_{B^{\infty}}^*\otimes \id_{\beta(X)}$ and the second to last map is also a unitary isomorphism defined by sending an elementary tensor $(\xi\lhd a\boxtimes T)\otimes (b_n)\in \beta^\rs(X)\otimes B^{\infty}\mapsto (\xi\lhd (a\boxtimes T)(b_n))_{n\in \bN}\in \beta(X)^{\infty}$ (here we have used Cohen factorisation). Similarly, a $\cC$-equivariant structure $\bv$ on ${}_\phi B^{\infty}$ induces a $\cC$-equivariant structure $\widehat{\bv}$ on ${}_\phi B^\rs$ through the composition
\begin{equation}\label{eqn:BinftoBs}
\begin{tikzcd}
    \alpha(X)\otimes {}_\phi B^{s}\ar{d} \arrow[r, mapsto, "\widehat{\bv}_X"]&{}_\phi B^{s}\otimes \beta^{s}(X)\\
    \alpha(X)\otimes {}_\phi B^{\infty}\otimes \overline{B^{\infty}}\ar{r}{\bv_X} & {}_\phi B^{\infty}\otimes \beta(X)\otimes \overline{B^{\infty}}\ar{u} 
\end{tikzcd}
\end{equation}
with $\overline{B^{\infty}}$ the contragredient correspondence of the $B^\rs$--$B$ Morita equivalence $B^{\infty}$. In the above composition we have used the isomorphisms $B^{\infty}\otimes \overline{B^{\infty}}\rightarrow B^\rs$ and $\overline{B^{\infty}}\otimes B^{\infty}\rightarrow B$ given by the left and right valued inner products. A direct computation shows that the two maps defined in \eqref{eqn:BstoBinf} and \eqref{eqn:BinftoBs} are inverse to one another.
\par Finally, we note that a pair $(\phi,\bu),(\psi,\bv)\in \Hom^{\cC}((\alpha,\fu),(\beta^\rs,\fv^\rs))$ forms an $((\alpha,\fu),(\beta,\fv))$-Cuntz-pair if and only if
\begin{equation}\label{eqn:alternativecuntzpair}
    \widehat{\bu}_XT_{\xi}^{\phi}-\widehat{\bv}_XT_{\xi}^{\psi}\in \cK(B^{\infty},\beta(X)^{\infty})
\end{equation}
where we have canonically identified $B^{\infty}\otimes \beta(X)$ with $\beta(X)^{\infty}$ in the codomain. Indeed, using Lemma \ref{lem:compactstensoridentity} we have that
\begin{align*}
    &\phi_X(\xi)-\psi_X(\xi)\in \cK(B^\rs,\beta^\rs(X))\\
    \iff{}& \bu_XT_{\xi}-\bv_XT_{\xi}\in \cK(B^\rs,B^\rs\otimes \beta^{s}(X))\\
    \iff{}& (\bu_XT_{\xi}-\bv_XT_\xi)\otimes \id_{B^{\infty}}\in \cK(B^\rs\otimes B^{\infty},B^\rs\otimes \beta^\rs(X)\otimes B^{\infty})\\
    \iff{}& \widehat{\bu}_XT_{\xi}^{\phi}-\widehat{\bv}_XT_{\xi}^{\psi}\in \cK(B^{\infty},\beta(X)^{\infty}).
\end{align*}
We summarise the above discussion in the following lemma.
\begin{lemma}\label{lem:cuntpairBinf}
    Let $(A,\alpha,\fu)$ and $(B,\beta,\fv)$ be two $\cC$-C$^*$-algebras with $A$ separable and $B$ $\sigma$-unital. The set $\bE^{\cC}((\alpha,\fu),(\beta,\fv))$ is in bijective correspondence with pairs $(\phi,\bu),(\psi,\bv)$ with $\phi,\psi\colon A\rightarrow B^\rs$ and $\bu,\bv$ $\cC$-equivariant structures on ${}_\phi B^{\infty}$ and ${}_\psi B^{\infty}$ respectively, such that
    \begin{equation}
    \bu_XT_{\xi}^{\phi}-\bv_XT_{\xi}^{\psi}\in \cK(B^{\infty},\beta(X)^{\infty}).  
    \end{equation}
     Under this bijection, the operation $\oplus$ is understood in a similar way as in Remark \ref{rmk:sums} and the equivalence relation $\simeq$ descends to the canonical definition of homotopy. 
\end{lemma}
For the remainder of this section we will refer to $\cC$-Cuntz pairs as suggested by Lemma \ref{lem:cuntpairBinf}.
\begin{rmk}\label{rmk:canonicalmaps}
It follows from Lemma \ref{lem:cuntpairBinf} that we have a canonical homomorphism of abelian groups $\Phi\colon\bE^\cC((\alpha,\fu),(\beta,\fv))/{\simeq}\to \KK^\cC((\alpha,\fu),(\beta,\fv))$ sending $\left((\phi,\bu),(\psi,\bv)\right)$ to $\left((B^{\infty})_{+}\oplus (B^{\infty})_{-},\phi\oplus\psi,\bu\oplus\bv,\left(\begin{array}{cc}0&1\\1&0\end{array}\right)\right)$ which is a $\cC$-Kasparov module. A homotopy of $\cC$-Cuntz pairs will induce a homotopy of the respective Kasparov modules, so this canonical map induces a well-defined map $\bE^\cC((\alpha,\fu),(\beta,\fv))/{\simeq}\to \KK^\cC((\alpha,\fu),(\beta,\fv))$. 
%The same construction yields a well-defined map $\bE_{\mathrm{e}}^\cC((\alpha,\fu),(\beta,\fv))/{\simeq_{\mathrm{e}}}\to \KK^\cC((\alpha,\fu),(\beta,\fv))$. 
%These maps send the Cuntz sum of two (extendible) Cuntz pairs to the sum of their images in $\KK^\cC((\alpha,\fu),(\beta,\fv))$, where the latter does not depend on the choice of isometries $s_1,s_2\in \cM(B\otimes\bK)$ generating $\cO_2$. \todo{explain this last bit more carefully}
\end{rmk}

To show that the mapping $\bE^\cC((\alpha,\fu),(\beta,\fv))/{\simeq}\to \KK^\cC((\alpha,\fu),(\beta,\fv))$ defined in Remark \ref{rmk:canonicalmaps} is a group isomorphism, we factor through a variant of $\KK^\cC((\alpha,\fu),(\beta,\fv))$.

\begin{defn}\label{defn: KKVariants}
Let $(A,\alpha,\fu)$ and $(B,\beta,\fv)$ be $\cC$-C$^*$-algebras with $A$ separable and $B$ $\sigma$-unital. Let 
\[\bfE_{\mathrm{s}}^\cC((\alpha,\fu),(\beta,\fv)) \coloneqq \{ (E,\phi,\bu,F)\in\bfE^\cC((\alpha,\fu),(\beta,\fv)) \,|\, F^*=F, F^2=1 \}\] and 
\[\bfE_{\mathrm{c}}^\cC((\alpha,\fu),(\beta,\fv)) \ = \{ (E,\phi,\bu,F)\in\bfE^\cC((\alpha,\fu),(\beta,\fv)) \,|\, F^*=F\text{ contractive} \}.\]
%\item $\bfE_{\mathrm{s}}^\cC((\alpha,\fu),(\beta,\fv))\hspace{0.17cm} = \{ (E,\phi,\bu,F)\in\bfE^\cC((\alpha,\fu),(\beta,\fv)) \,|\, F^*=F\text{ unitary} \},$
%\item $\bfE_{\nd}^\cC((\alpha,\fu),(\beta,\fv)) = \{ (E,\phi,\bu,F)\in\bfE^\cC((\alpha,\fu),(\beta,\fv))\,|\, \text{$\phi$ non-degenerate} \}$, 
%\item $\bfE_{\mathrm{e}}^\cC((\alpha,\fu),(\beta,\fv))\hspace{0.17cm} = \{ (E,\phi,\bu,F)\in\bfE^\cC((\alpha,\fu),(\beta,\fv))\,|\, \text{$\phi$ extendible}\},$ 
%\item $\bfE_{\ec}^\cC((\alpha,\fu),(\beta,\fv))\hspace{0.1cm} = \bfE_{\mathrm{e}}^\cC((\alpha,\fu),(\beta,\fv))\cap \bfE_{\mathrm{c}}^\cC((\alpha,\fu),(\beta,\fv))$,
%\item $\bfE_{\es}^\cC((\alpha,\fu),(\beta,\fv))\hspace{0.09cm} = \bfE_{\mathrm{e}}^\cC((\alpha,\fu),(\beta,\fv))\cap \bfE_{\mathrm{s}}^\cC((\alpha,\fu),(\beta,\fv))$. 
We then define \[\KK^\cC_{\mathrm{s}}((\alpha,\fu),(\beta,\fv))\coloneqq\bfE^\cC_{\mathrm{s}}((\alpha,\fu),(\beta,\fv)) /{\simeq_{\mathrm{s}}},\] 
and
\[
\KK^\cC_{\mathrm{c}}((\alpha,\fu),(\beta,\fv))\coloneqq\bfE^\cC_{\mathrm{c}}((\alpha,\fu),(\beta,\fv)) /{\simeq_{\mathrm{c}}},
\]
    where for $x_0,x_1\in \bfE_{\mathrm{s}}^\cC((\alpha,\fu),(\beta,\fv))$ we write $x_0\simeq_{\mathrm{s}} x_1$ if there exists an element $y\in \bfE^\cC_{\mathrm{s}}((\alpha,\fu),(\beta\otimes\id_{C[0,1]},\fv\otimes 1))$ such that $y\circ \ev_1$ and $y \circ \ev_0$ are $\cC$-equivariantly unitarily equivalent to $x_1$ and $x_0$ respectively. We define $\simeq_{\mathrm{c}}$ in the same manner.    
    %Likewise, we can define $\simeq_{\mathrm{s}}$ and $\KK^\cC_{\mathrm{s}}((\alpha,\fu),(\beta,\fv))$. 
    %$\KK^\cC_{\nd}((\alpha,\fu),(\beta,\fv))$, $\KK^\cC_{\mathrm{e}}((\alpha,\fu),(\beta,\fv))$, $\KK^\cC_{\ec}((\alpha,\fu),(\beta,\fv))$, and $\KK^\cC_{\es}((\alpha,\fu),(\beta,\fv))$.
\end{defn}

Consider the sequence of maps
\[ \bE^\cC ((\alpha,\fu),(\beta,\fv))/{\simeq} \to \KK^\cC_{\mathrm{s}}((\alpha,\fu),(\beta,\fv)) %\to \KK^\cC_{\mathrm{c}} 
\to \KK^\cC((\alpha,\fu),(\beta,\fv)), \]
%\[ \bE^\cC_{\mathrm{e}}/{\simeq_{\mathrm{e}}} \to \KK^\cC_{\es} \to \KK^\cC_{\ec} \to \KK^\cC_{\mathrm{e}} \to \KK^\cC \]
where we use that the map of \cref{rmk:canonicalmaps} actually falls into $\KK^\cC_{\mathrm{s}}$. We first show that the canonical map $\KK_{\mathrm{s}}^\cC((\alpha,\fu),(\beta,\fv))\to \KK^\cC((\alpha,\fu),(\beta,\fv))$ is an isomorphism. Thus it will remain to show that $\bE^\cC ((\alpha,\fu),(\beta,\fv))/{\simeq} \to \KK^\cC_{\mathrm{s}}((\alpha,\fu),(\beta,\fv))$ is an isomorphism. The lemma below follows from a well-known argument in non-equivariant $\KK$-theory (see \cite[17.4]{blackadar}),
%We first show that the rightmost map in the below is an isomorphism.

	%Indeed, this gives a well-defined assignment $\bfE^\cC((\alpha,\fu),(\beta,\fv))\to \KK^\cC_{\nd}((\alpha,\fu),(\beta,\fv))$ since the Kasparov product is unique up to an operator homotopy \nn{what is meant here? SG}, and for any element in $\bfE^\cC(A,C[0,1]\otimes B)$, its Kasparov product with $(A,\id_A,0)$ gives a homotopy between its evaluations at $0$ and $1$, so the map $\KK^\cC((\alpha,\fu),(\beta,\fv))\to\KK_{\nd}^\cC((\alpha,\fu),(\beta,\fv))$ is indeed well-defined, which is the desired inverse by construction. \nn{Doesn't this follow from \cite[Lemma 3.16]{ARKIKU23}? Rather than citing Blackadar which does not discuss $\cC$-equivariance? SG}

\begin{lemma}\label{lemma: unitaryKK}
Let $(A,\alpha,\fu)$ and $(B,\beta,\fv)$ be $\cC$-C$^*$-algebras with $A$ separable and $B$ $\sigma$-unital. Then the canonical map $\Phi\colon \KK_{\mathrm{s}}^\cC((\alpha,\fu),(\beta,\fv))\to \KK^\cC((\alpha,\fu),(\beta,\fv))$ is a group isomorphism. 
\end{lemma}

\begin{proof} 
Since $\Phi$ is the composition of the two maps $\KK_{\mathrm{s}}^\cC((\alpha,\fu),(\beta,\fv))\to \KK_{\mathrm{c}}^\cC((\alpha,\fu),(\beta,\fv))$ and $\KK_{\mathrm{c}}^\cC((\alpha,\fu),(\beta,\fv))\to \KK^\cC((\alpha,\fu),(\beta,\fv))$, it suffices to show that they are both isomorphisms. The map \[\KK_{\mathrm{c}}^\cC((\alpha,\fu),(\beta,\fv))\to \KK^\cC((\alpha,\fu),(\beta,\fv))=\bfE^\cC((\alpha,\fu),(\beta,\fv))/{\simeq}\] is indeed an isomorphism by \cite[Remark 3.8]{ARKIKU23} (see also \cite[Proposition 17.4.3]{blackadar}). 
%Its inverse given by the functional calculus of $(F+F^*)/2$ with respect to the function 
%\[t\mapsto\left\{ \begin{array}{cl} -1&(\text{if $t\leq -1$})\\ t&(\text{if $-1\leq t\leq 1$})\\ 1&(\text{if $1\leq t$}).\end{array} \right.\]
Moreover, the canonical map $\Theta\colon\KK_{\mathrm{s}}^\cC((\alpha,\fu),(\beta,\fv))\to \KK_{\mathrm{c}}^\cC((\alpha,\fu),(\beta,\fv))$ is an isomorphism with the inverse $\Psi$ given by the assignment
\begin{align*}
	&
	 (E,\phi,\bu,F)\mapsto \left(E\oplus E^{\op},\phi\oplus0,\bu\oplus0,\left(\begin{array}{cc}
		F&\sqrt{1-F^2}\\\sqrt{1-F^2}&-F
	\end{array}\right)\right).
\end{align*}
Indeed, notice that any $(E,\phi,\bu,F)\in\bfE_{\mathrm{s}}^\cC((\alpha,\fu),(\beta,\fv))$ is homotopic to $\Psi\circ\Theta(E,\phi,\bu,F)$ via an element in $\bfE_{\mathrm{s}}^\cC(A,C[0,1]\otimes B)$ of the form
\begin{align*}
	&
	\left((C[0,1]\boxtimes E)\oplus (C_0(0,1]\boxtimes E^{\op}),\phi\oplus0,\bu\oplus0,\left(\begin{array}{cc}
		F&0\\0&-F
	\end{array}\right)\right).
\end{align*}
On the other hand, any $(E,\phi,\bu,F)\in\bfE_{\mathrm{c}}^\cC((\alpha,\fu),(\beta,\fv))$ is homotopic to $\Theta\circ\Psi(E,\phi,\bu,F)$ via an element in $\bfE_{\mathrm{c}}^\cC(A,C[0,1]\otimes B)$ of the form
\begin{align*}
	&
	\left((C[0,1]\boxtimes E)\oplus (C_0(0,1]\boxtimes E^{\op}),\phi\oplus0,\bu\oplus0,\left(\begin{array}{cc}
		F&G^*\\G&-F
	\end{array}\right)\right),
\end{align*}
where $G\colon[0,1] \ni t\mapsto t\sqrt{1-F^2} \in \cL(E,E^{\op})$. 
Thus, it follows that the canonical map $\Phi\colon \KK_{\mathrm{s}}^\cC((\alpha,\fu),(\beta,\fv))\to \KK^\cC((\alpha,\fu),(\beta,\fv))$ is an isomorphism. 
\end{proof}

\begin{theorem}\label{thm:Cuntzpicture}
Let $(A,\alpha,\fu)$ and $(B,\beta,\fv)$ be $\cC$-C$^*$-algebras with $A$ separable and $B$ $\sigma$-unital. Then the canonical map %$\bE_{\mathrm{e}}^\cC((\alpha,\fu),(\beta,\fv))/{\simeq_{\mathrm{e}}}\to \KK^\cC((\alpha,\fu),(\beta,\fv))$ and 
$\bE^\cC((\alpha,\fu),(\beta,\fv))/{\simeq}\to \KK^\cC_{s}((\alpha,\fu),(\beta,\fv))$ is bijective and hence the map $\bE^\cC((\alpha,\fu),(\beta,\fv))/{\simeq}\to \KK^\cC((\alpha,\fu),(\beta,\fv))$ is bijective.
\end{theorem}

As noted in \cref{lemma: CuntzSumGroup}, $\bE^\cC((\alpha,\fu),(\beta,\fv))/{\simeq}$ is a group under Cuntz sum, so the maps of Theorem \ref{thm:Cuntzpicture} are isomorphisms of groups. 

\begin{proof}
By Lemma~\ref{lemma: unitaryKK}, it suffices to show that the canonical map \[\bE^\cC((\alpha,\fu),(\beta,\fv))/{\simeq}\to \KK^\cC_{s}((\alpha,\fu),(\beta,\fv))\]is an isomorphism.
Note that in this proof we shall denote a generic $\cC$-Cuntz pair in $\bE^\cC((\alpha,\fu),(\beta,\fv))$ by $(\phi_\pm,\bu_\pm)$ instead of $((\phi_+,\bu_+),(\phi_-,\bu_-))$.
Let us first construct a homomorphism \[\KK_{\mathrm{s}}^\cC((\alpha,\fu),(\beta,\fv))\to \bE^\cC((\alpha,\fu),(\beta,\fv))/{\simeq}.\]
Indeed, given $(E,\phi,\bu,F)\in\bfE_{\mathrm{s}}^\cC((\alpha,\fu),(\beta,\fv))$, we denote by $(E_+,\phi_+,\bu_+)$ and $(E_-,\phi_-,\bu_-)$ its positive and negatively graded parts.
Let $U\in\cU(E_-,E_+)$ be the unitary such that $F=(\begin{smallmatrix}0&U\\U^*&0\end{smallmatrix})$.
Given an adjointable isometry $S\in\cL(E_+, B^{\infty})$, set $S_+\coloneqq S$ and $S_-\coloneqq SU$.
Note that such an isometry exists by Kasparov's stabilisation theorem since one may let $T\colon E_{+}\oplus B^{\infty}\rightarrow B^{\infty}$ be a right $B$-module isomorphism and $S=T\circ \iota$ with $\iota\colon E_{+}\rightarrow E_{+}\oplus B^{\infty}$ defined by $\iota(\xi)=\xi\oplus 0$.
We then define an element of $\bE^\cC((\alpha,\fu),(\beta,\fv))$ by
    \begin{equation}\label{eq: Psi}
        \Psi(E,\phi,\bu,F;S)\coloneqq (\Ad(S_{\pm})\phi_\pm,(S_\pm\otimes\id_{\beta(X)}) \bu_\pm (\id_{\alpha(X)}\otimes S_\pm^*)). 
    \end{equation}

    We claim that the map $\Psi$ is well-defined, i.e. that it does not depend on the choice of the isometry $S$. Indeed, \cref{lem:strictpathconnisomet} shows that any two isometries $S_0$ and $S_1$ constructed as above, can be connected with a strictly continuous path of adjointable isometries $S_t\in\cL(E,B^{\infty})$ for $t\in[0,1]$. Letting $S_{+,t}\coloneqq S_t$ and $S_{-,t}\coloneqq S_tU$, we see that \[\Psi(E,\phi,\bu,F;S_0)\simeq\Psi(E,\phi,\bu,F;S_1).\] 
    Thus, $\Psi$ gives a well-defined map $\bfE_{\mathrm{s}}^\cC((\alpha,\fu),(\beta,\fv))\rightarrow \bE^\cC((\alpha,\fu),(\beta,\fv))$. 
    By construction, $\Psi$ induces a well-defined map $\KK_{\mathrm{s}}^\cC((\alpha,\fu),(\beta,\fv))\to \bE^\cC((\alpha,\fu),(\beta,\fv))/{\simeq}$ which will also be denoted by $\Psi$. 
	
    If $\Phi$ is the map defined in Remark \ref{rmk:canonicalmaps}, then the isomorphism $T\colon E_{+}\oplus B^{\infty}\rightarrow B^{\infty}$ induces an isomorphism of Kasparov modules \[\Phi\circ\Psi(E,\phi,\bu,F)\cong (E,\phi,\bu,F)\oplus \left((B^{\infty})_+\oplus (B^{\infty})_-,0,0,0\right).\] Since $\left(B^{\infty}_+\oplus B^{\infty}_-,0,0,0\right)$ is a degenerate module, it follows that \[\Phi\circ\Psi(E,\phi,\bu,F;S)\simeq (E,\phi,\bu,F)\] in $\KK^\cC((\alpha,\fu),(\beta,\fv))$. Hence, $\Phi\circ\Psi(E,\phi,\bu,F)\simeq_{\mathrm{s}} (E,\phi,\bu,F)$ so $\Phi\circ\Psi$ is the identity map on $\KK_{\mathrm{s}}^\cC((\alpha,\fu),(\beta,\fv))$.
 
    \par We claim that $\Psi\circ\Phi$ is the identity map on $\bE^\cC((\alpha,\fu),(\beta,\fv))/{\simeq}$. Note first that $\Psi\circ\Phi(\phi_\pm,\bu_\pm)=(\phi_\pm,\bu_\pm)$ by choosing $S=\id_{B^{\infty}}$. As the homotopy class of the $\cC$-Cuntz pair $\Psi\circ\Phi(\phi_\pm,\bu_\pm)$ does not depend on the choice of $S$, the claim follows. Hence the canonical map $\bE^\cC((\alpha,\fu),(\beta,\fv))/{\simeq}\to\KK_{\mathrm{s}}^\cC((\alpha,\fu),(\beta,\fv))$ is an isomorphism and so the conclusion follows by Lemma~\ref{lemma: unitaryKK}.
\end{proof}

\section{Stable operator homotopy}\label{sec:operatorhom}

In this section we show the $\cC$-equivariant version of \cite[Lemma~2.8]{GASZ22} i.e. that if the class of a $\cC$-Cuntz pair is $0$ in $\KK$-theory, then any pair of $\cC$-cocycle representations yielding this Cuntz-pair are stably operator homotopic. We start by defining the $\cC$-equivariant analogue of stable operator homotopy.

\begin{defn}
Let $(A,\alpha,\fu)$ and $(B,\beta,\fv)$ be $\cC$-C$^*$-algebras, and $\phi,\psi\in\Hom^\cC((\alpha,\fu), (\beta^\rs,\fv^\rs))$. 
We say that $\phi$ and $\psi$ are \emph{$\cC$-operator homotopic} 
if there is a norm-continuous path $u_t\in\cU(\fD_{\phi})$ over $t\in[0,1]$ such that $u_0=1$ and $\Ad(u_1) \circ \phi=\psi$.\footnote{Recall the definition of $\fD_\phi$ from Definition \ref{defn: DPhi}.} We further say that 
$\phi$ and $\psi$ are \emph{stably $\cC$-operator homotopic} 
if there is another $\cC$-cocycle representation $\kappa$ such that $\phi\oplus \kappa$ and $\psi\oplus \kappa$ are $\cC$-operator homotopic.
\end{defn}
\begin{comment}
    In the definition, note that for every object $X\in\cC$ with a unitary $u\colon X\cong \bigoplus_{k=1}^{n}X_k$ for some $X_k\in\Irr(\cC)$, we have 
\begin{align*}
    &[\fD_\phi,\phi_X(\alpha(X))] 
    = \beta(u^*) \{ ([x,\phi_{X_k}(\alpha(X_k))])_{k=1}^{n} \,|\, x\in \fD_\phi \} \\
    \subset{}& \beta(u^*) \bigoplus_{k=1}^{n} \cK(B,\beta(X_k)) = \cK(B,\beta(X_k)) . 
\end{align*}
\end{comment}

\begin{rmk}
The relation of $\cC$-operator homotopy is an equivalence relation. Symmetry follows from the fact that $\fD_{\phi}\subset\fD_{\psi}$ if $\phi$ and $\psi$ are $\cC$-operator homotopic. To see this, let $x\in \fD_{\phi}$, $X\in\cC$, $\xi\in\alpha(X)$, and $u_t\in\cU(\fD_{\phi})$ be the path of unitaries realising the $\cC$-operator homotopy with $1=u_0$ and $u\coloneqq u_1$. Then, since $u\in \fD_{\phi}$, we have that $\psi_X(\xi)\equiv\phi_X(\xi)$ mod $\cK(B,\beta(X))$. But $x\in \fD_{\phi}$, so it follows that $[x,\psi_X(\xi)]\equiv 0$ mod $\cK(B,\beta(X))$. This gives that $\fD_{\phi}\subset\fD_{\psi}$. Transitivity follows by a similar argument.
\end{rmk}

\begin{notation}
Let $(e_n)_{n\geq 0}$ be an orthonormal basis of $\ell^2(\bZ_{\geq 0})$ and regard $\ell^2(\bZ_{\geq1})$ as the closed subspace of $\ell^2(\bZ_{\geq 0})$ generated by $(e_n)_{n\geq 1}$. 
By \cite[Theorem 1.15]{GASZ22} (see also \cite[Lemma 3.26]{CMR07}), we have a unital $^*$-homomorphism $\theta\colon C[0,1]\to\cB(\ell^2(\bZ_{\geq 1}))$ and a norm-continuous path of unitaries $w_t\in \cU(\cB(\ell^2(\bZ_{\geq 0})))$ such that \[w_0=1,\ \Ad(w_1) \circ \theta_0=\theta_1,\ \text{and}\ [w_t,\theta_0(C[0,1])]\subset\cK(\ell^2(\bZ_{\geq 0})),\] for any $t\in[0,1]$, where $\theta_t$ denotes the map $\theta_t\coloneqq\ev_t\oplus \theta\colon C[0,1]\to\cB(\ell^2(\bZ_{\geq 0}))$, whose direct sum is taken with respect to the isometries $\bC\cong \bC e_0\subset \ell^2(\bZ_{\geq 0})$ and $\ell^2(\bZ_{\geq 1})\subset \ell^2(\bZ_{\geq 0})$. 
Moreover, given a $\cC$-C$^*$-algebra $(B,\beta,\fv)$, we set up some notation.
\begin{itemize}[leftmargin=*]%\setlength{\leftskip}{-2em}
\item
Fix isometries $r_n\in\cB(\ell^2(\bZ_{\geq 0}))$ for $n\in\bZ_{\geq 0}$ with pairwise orthogonal ranges such that $\sum_{n=0}^{\infty} r_nr_n^* =1$ in the strict topology, and let $r_\infty\coloneqq \sum_{n=0}^{\infty}r_{n+1}r_n^*$. 
For $n\in\bZ_{\geq 0}\cup\{\infty\}$, we write $r_n^B\coloneqq 1_{\cM(B)}\otimes r_n\in \cM(B^\rs)$. 
We regard $r=(r_n)_{n=0}^{\infty}\in\cU(\ell^2(\bZ_{\geq 0})^{\otimes 2},\ell^2(\bZ_{\geq 0}))$ as a unitary sending $\xi\otimes e_n$ to $r_n\xi$ for all $n\in\bZ_{\geq 0}$ and $\xi\in\ell^2(\bZ_{\geq 0})$. We also let $r^B= 1_{\cM(B)}\otimes r$.
\item
We define $\theta^B\in \Hom^\cC((\beta^\rs\otimes \id_{C[0,1]},\fv^\rs\otimes 1),(\beta^\rs,\fv^\rs))$ by $\theta^{B}\coloneqq \Ad(r^B) \circ (\id_{B^\rs}\otimes \theta)=\id_{B}\otimes \Ad(r)\circ (\id_{\bK}\otimes \theta)$. 
This is indeed a well-defined $\cC$-cocycle representaton by \cref{example: cocyclerep}.
\item
Similarly, we define $\theta^B_t\coloneqq \Ad(r^B)\circ (\id_{B^\rs}\otimes \theta_t)=(\id_{B}\otimes \Ad(r))\circ (\id_{\bK}\otimes \theta_t) \in \Hom^\cC((\beta^\rs\otimes \id_{C[0,1]},\fv^\rs\otimes 1),(\beta^\rs,\fv^\rs))$ for $t\in [0,1]$. 
Note that $\theta^B_t=(\id_{B^{\rs}}\otimes\ev_t)\oplus_{r_0^B,r_\infty^B}\theta^B$. 
\item
We let $w^B_t\coloneqq \Ad(r^B) (1_{\cM(B^\rs)}\boxtimes w_t)\in \cU(\cM(B^\rs))$. 
Note that $\Ad(w^B_1)\circ \theta^B_0 = \theta^B_1$ as $\cC$-cocycle representations and $[w^B_t , \theta^B_{0,X}(\xi)] \in \cK(B^\rs,\beta^\rs(X))$ for all $t\in[0,1]$, $X\in\Irr(\cC)$, and $\xi\in\beta^\rs(X)\boxtimes C[0,1]$. 
\end{itemize}
%For $\phi,\psi\in\Hom^\cC((\alpha,\fu), (\beta^\rs,\fv^\rs))$, we denote by $\phi\oplus_{r^B_0,r^B_\infty}\psi\in \Hom^\cC((\alpha,\fu), (\beta^\rs,\fv^\rs))$ the Cuntz sum obtained using the pair of isometries $(r^B_0,r^B_\infty)$.
\end{notation}
The following lemma is the key step in the proof of the main result of this section, Theorem \ref{thm:KKCtostableoh}. Similarly to the argument employed in \cite[Lemma~2.8]{GASZ22}, a crucial ingredient is the $\cC$-equivariant version of Kasparov's technical theorem (\cite[Theorem B.1]{ARKIKU23}).

\begin{lemma}\label{lem:KKarantine}
For $\cC$-C$^*$-algebras $(A,\alpha,\fu)$ and $(B,\beta,\fv)$, with $A$ separable and $B$ $\sigma$-unital,
let \[(\Phi,\Psi)\in\bE^{\cC}((\alpha,\fu),(\beta\otimes\id_{C[0,1]},\fv\otimes 1))\] be such that \[\phi \coloneqq \ev_0\circ \Phi = \ev_1\circ \Phi = \ev_0\circ \Psi.\] 
If we put $\psi \coloneqq \ev_1\circ \Psi $ and \[\kappa \coloneqq \theta^B\circ \Psi \oplus_{r^B_0,r^B_\infty} \theta^B_0\circ \Phi \in \Hom^\cC((\alpha,\fu), (\beta^\rs,\fv^\rs)),\] then 
$\phi \oplus_{r^B_0,r^B_\infty} \kappa$ is $\cC$-operator homotopic to  $\psi \oplus_{r^B_0,r^B_\infty} \kappa$. 
\end{lemma}

\begin{proof}
We put $\Theta\coloneqq \theta^B_0\circ \Psi \oplus_{s_1,s_2} \theta^B_1\circ \Phi=\theta^B_0\circ \Psi \oplus_{s_1,s_2} \theta^B_0\circ \Phi = \phi \oplus_{r^B_0,r^B_\infty} \kappa$ for the pair of isometries $(s_1,s_2)\coloneqq (r^B_0r^{B*}_0 + r^{B}_{\infty}r^B_0r^{B*}_{\infty},(r^B_\infty)^2)$ that generate a copy of $\cO_2$. Observe that $\psi \oplus_{r^B_0,r^B_\infty} \kappa = \Ad (w^B_1\oplus_{s_1,s_2} w^{B*}_1)\circ \Theta$. 
We first claim that $w^B_1\in \fD_{\theta^B_0\circ\Phi}\cap \fD_{\theta^B_0\circ\Psi}$. Then, it would follow that 
$w^B_t\oplus_{s_1,s_2} w^{B*}_{t'}\in \fD_\Theta$
for any $t,t'\in\{0,1\}$ as $w_0^B=1$. To see the claim, note that
\[\Ad(w^B_1)\circ\theta^B_0\circ \Phi =\theta^B_1\circ \Phi =\theta^B_0\circ \Phi\] and hence, since $(\phi,\psi)\in\bE^\cC((\alpha,\fu),(\beta^\rs,\fv^\rs))$, one has that  
\begin{align*}
\Ad(w^B_1) ((\theta^{B}_{0}\circ \Psi)_{X}(\xi)) 
={}&
(\theta^{B}_{1}\circ \Psi)_X(\xi)\\
={}&
\psi_X(\xi)\oplus_{r^B_0,r^B_\infty}(\theta^{B} \circ \Psi)_X(\xi)\\ 
\equiv{}& 
\phi_X(\xi)\oplus_{r^B_0,r^B_\infty}(\theta^{B}\circ \Psi)_X(\xi)\\ 
={}& 
(\theta^{B}_{0}\circ \Psi)_X(\xi)
\end{align*}
mod $\cK(B^{\rs},\beta^\rs(X))\oplus_{r^B_0,r^B_\infty}0$. 

Therefore, it suffices to construct a continuous path in $\cU(\fD_\Theta)$ from $1$ to $w^B_1\oplus \left(w^{B}_1\right)^*$. For this we apply Kasparov's technical theorem (\cite[Theorem B.1]{ARKIKU23}) to the following septuple $(I,\cJ, \cJ_i,\sigma_i, \cA_1, \cA_{2,i}, \Delta_i )$: 
\begin{itemize}[leftmargin=*]%\setlength{\leftskip}{-2em}
\item 
$I = \Irr(\cC) \cup \{ 0 \}$, 
\item
$\cJ = \cJ_0 \coloneqq B^\rs$ with $\sigma_0\coloneqq\id_{B^\rs}$,
\item
$\cJ_X \coloneqq \cK(\beta^\rs(X)\oplus B^\rs)$ with $\sigma_X\coloneqq \beta^\rs_X\oplus\id_{B^\rs}\colon B^\rs\to \cK(\beta^\rs(X)\oplus B^\rs)$, where $\beta^\rs_X\colon B^\rs\to\cK(\beta^\rs(X))$ is the left $B^\rs$-action on $\beta^\rs(X)$ for $X\in\Irr(\cC)$, 
\item
$\cA_1 \coloneqq \theta^B_{0,1_\cC}(B^\rs\otimes C[0,1]) + B^\rs \subset \cM(B^\rs)$, 
\item
$\cA_{2,0} \coloneqq C^*
\bigl( (\Ad(w^B_t) - \id_{B^\rs}) \theta^B_{0,1_\cC} (b) \,\big|\, t\in[0,1], b\in \Phi_{1_{\cC}}(A)\cup\Psi_{1_{\cC}}(A) \bigr) \subset \cM(B^\rs)$, 
\item
$\cA_{2,X} \coloneqq C^*\biggl( (\Ad(w^B_t) - \id_{\cJ_X}) \biggl(\begin{array}{cc}0&\theta^{B}_{0,X}(\xi)\\\theta^{B}_{0,X}(\xi)^*&0\end{array}\biggr) \,\bigg|\, t\in[0,1], \xi\in \Phi_X(\alpha(X))\cup\Psi_X(\alpha(X)) \biggr) \subset \cL(\beta^\rs(X)\oplus B^\rs)$ for $X\in\Irr(\cC)$, 
\item
$\Delta_0 \coloneqq \overline{\spa}\{\theta^B_{0,1_\cC}(\Phi(A)) \cup \theta^B_{0,1_\cC}(\Psi(A)) \cup \{ w^B_t \,|\, t\in[0,1] \}\} \subset \cM(B^\rs)$, and 
\item
$\Delta_X \coloneqq \biggl\{ \biggl(\begin{array}{cc}0&\theta^{B}_{0,X}(\xi)\\\theta^{B}_{0,X}(\xi)^*&0\end{array}\biggr) \,\bigg|\, \xi\in\Phi_X(\alpha(X))\cup\Psi_X(\alpha(X)) \biggr\} \subset \cL(\beta^\rs(X)\oplus B^\rs)$ for $X\in\Irr(\cC)$. 
\end{itemize}
Where we have used the extension of $\theta_{0,X}^B$ in the notation above as defined in Remark \ref{rmk:compnondegcocyclereps}. After checking the conditions required in \cite[Theorem B.1]{ARKIKU23} and noting that $\sigma_X$ extends canonically to a map from $\cM(B^\rs)$,\footnote{This follows from nondegeneracy of $\beta^\rs(X)$.} we obtain $M\in\cM(B^\rs)$ with $0\leq M\leq 1$ such that 
\begin{align*}
    &
    M\cA_1,\ (1-M)\cA_{2,0}, [M,\Delta_0] \subset B^\rs,
\end{align*}
\begin{align}\label{eq:lem:KKarantine1}
    &\sigma_X(1-M)\cA_{2,X}\subset \cJ_X,
\end{align}
\begin{align}\label{eq:lem:KKarantine2}
    &[\sigma_X(M),\Delta_X]\subset \cJ_X
\end{align}
for $X\in\Irr(\cC)$. 
Here, note that \eqref{eq:lem:KKarantine2} implies $M\in \fD_{\theta^B_0\circ\Psi}\cap\fD_{\theta^B_0\circ\Phi}$. 
Moreover, as $M\cA_1\subset B^\rs$, or equivalently $\cA_1 M\subset B^\rs$, one has that 
\begin{align}\label{eq:lem:KKarantine3}\begin{aligned}
    &\theta^B_{0,X}(\beta^\rs(X)\boxtimes C[0,1])\lhd M 
    \\={}& \theta^B_{0,X}(\beta^\rs(X)\boxtimes C[0,1])\lhd \theta^B_{0,1_\cC}(B^\rs\boxtimes C[0,1])M 
    \\\subset{}& \cK(B^s,\beta^\rs(X)) 
\end{aligned}\end{align}
by applying Cohen factorisation.

We now check the conditions of \cite[Theorem B.1]{ARKIKU23}.
First, as \[[w^B_t,\theta^B_{0,1_\cC}(B^\rs\otimes C[0,1])]\subset B^\rs,\] we see that for $b\in\Phi_{1_{\cC}}(A)\cup\Psi_{1_{\cC}}(A)$ and $x\in B^\rs\otimes C[0,1]$, 
\begin{align*}
&
\theta^B_{0,1_\cC}(x) [w^B_t,\theta^B_{0,1_\cC}(b)] 
=
[ \theta^B_{0,1_\cC}(x) , w^B_t ]\theta^B_{0,1_\cC}(b)
+ [ w^B_t , \theta^B_{0,1_\cC}(x b) ] 
\in
B^\rs, 
\end{align*}
which shows that $\cA_1\cA_{2,0}\subset B^\rs$.
Then, it is not hard to see 
\[[ \cA_1 , \Delta_0 ] 
\subset \theta^B_{0,1_\cC}(B^\rs\otimes C[0,1]) + B^\rs 
=\cA_1.\]
Moreover, for $X\in\Irr(\cC)$, $\xi\in\Phi_X(\alpha(X))\cup\Psi_X(\alpha(X))$, and $x\in B^\rs\otimes C[0,1]$, using $[w^B_t,\theta^{B}_{0,1_\cC}(x)] \in B^\rs$ and $x\rhd \xi\in \cK(B^\rs\otimes C[0,1],\beta^\rs(X)\boxtimes C[0,1])$, we see 
\begin{align*}
&
\theta^{B}_{0,1_\cC}(x)\rhd [w^B_t , \theta^{B}_{0,X}(\xi)] 
\\={}&
[\theta^{B}_{0,1_\cC}(x) ,w^B_t] \rhd \theta^{B}_{0,X}(\xi) + [w^B_t , \theta^{B}_{0,X}(x \rhd \xi) ]
\\\in{}&
\cK(B^s,\beta^\rs(X)), 
\end{align*}
and similarly $[w^B_t,\theta^B_{0,X}(\xi)] \lhd \theta^{B}_{0,1_\cC}(x) \in\cK(B^\rs,\beta^\rs(X))$, which shows that $\sigma_X(\cA_1)\cA_{2,X}\subset \cJ_X$.
Finally, for $x\in B^\rs\otimes C[0,1]$, $X\in\Irr(\cC)$, and $\xi\in\Phi_X(\alpha(X))\cup \Psi_X(\alpha(X))$, by using Cohen factorisation, there are $y,y',z,z'\in B^\rs\otimes C[0,1]$ and $\eta,\eta'\in \cK(B^\rs\otimes C[0,1],\beta^\rs(X)\boxtimes C[0,1])$ such that $x\rhd \xi=y\rhd\eta\lhd z$ and $\xi\lhd x =y'\rhd \eta'\lhd z'$ in $\cK(B^\rs,\beta^\rs(X))$. 
This shows 
\begin{align*}
    &
    [\theta^B_{0,1_\cC}(x),\theta^B_{0,X}(\xi)] 
    = \theta^B_{0,X} (y\rhd\eta\lhd z - y'\rhd \eta'\lhd z') 
    \\\in{}&
    \mathrm{span}\ \cA_1\rhd \cL(B^\rs,\beta^\rs(X))\lhd \cA_1
\end{align*}
and thus 
\[ [\sigma_X(\cA_1),\Delta_X] \subset \overline{\mathrm{span}}\ \sigma_X(\cA_1)\cM(\cJ_X)\sigma_X(\cA_1) . \]

We set $U\coloneqq\Ad (s_1,s_2) \left(\begin{array}{cc}\sqrt{1-M}&\sqrt{M}\\-\sqrt{M}&\sqrt{1-M}\end{array}\right)$ and claim that \[(1\oplus_{s_1,s_2} w_t^*)U(w_t\oplus_{s_1,s_2} 1)\in\cU(\fD_{\Theta}).\]  
Indeed, for $X\in\cC$ and $\xi\in \alpha(X)$, using $[M,w_t^B] \in [M,\Delta_0]\subset B^\rs$, we have 
\begin{align*}
&
\left(\begin{array}{cc}\sqrt{1-M}&\sqrt{M}\\-\sqrt{M}&\sqrt{1-M}\end{array}\right)
\rhd 
\Ad\left(\begin{array}{cc}w^B_t&0\\0&1\end{array}\right)
\left(\begin{array}{cc}\theta^{B}_{0,X}\Psi_X(\xi)&0\\0&\theta^{B}_{0,X}\Phi_X(\xi)\end{array}\right)
\\={}&
\left(\begin{array}{cc} \sqrt{1-M} w^B_t\rhd \theta^{B}_{0,X}\Psi_X(\xi) \lhd w^{B*}_t & \sqrt{M}\rhd \theta^{B}_{0,X}\Phi_X(\xi) \\ -\sqrt{M}w^B_t\rhd \theta^{B}_{0,X}\Psi_X(\xi) \lhd w^{B*}_t & \sqrt{1-M}\rhd \theta^{B}_{0,X}\Phi_X(\xi) \end{array}\right)
%\\\equiv{}&
%\left(\begin{array}{cc} w^B_t\sqrt{1-M}\rhd \theta^{B}_{0,X}\Psi_X(\xi) \lhd w^{B*}_t & \sqrt{M}\rhd \theta^{B}_{0,X}\Phi_X(\xi) \\ -w^B_t\sqrt{M}\rhd \theta^{B}_{0,X}\Psi_X(\xi) \lhd w^{B*}_t & \sqrt{1-M}\rhd \theta^{B}_{0,X}\Phi_X(\xi) \end{array}\right)
\\\stackrel{\eqref{eq:lem:KKarantine2}}{\equiv}&
\left(\begin{array}{cc} w^B_t\rhd \theta^{B}_{0,X}\Psi_X(\xi)\lhd w^{B*}_t\sqrt{1-M} & \theta^{B}_{0,X}\Phi_X(\xi) \lhd \sqrt{M} \\ -w^B_t\rhd \theta^{B}_{0,X}\Psi_X(\xi) \lhd w^{B*}_t\sqrt{M} & \theta^{B}_{0,X}\Phi_X(\xi) \lhd \sqrt{1-M} \end{array}\right)
\\\equiv{}&
\left(\begin{array}{cc} \theta^{B}_{0,X}\Psi_X(\xi) \lhd \sqrt{1-M} & \theta^{B}_{0,X}\Psi_X(\xi) \lhd \sqrt{M} \\ -w^B_t\rhd \theta^{B}_{0,X}\Phi_X(\xi) \lhd w^{B*}_t\sqrt{M} & w^B_t\rhd \theta^{B}_{0,X}\Phi_X(\xi) \lhd w^{B*}_t\sqrt{1-M} \end{array}\right)
\\+{}&
\left(\begin{array}{cc} 0 & \theta^{B}_{0,X}(\Phi_X(\xi) - \Psi_X(\xi)) \lhd \sqrt{M} \\ -w^B_t\rhd \theta^{B}_{0,X} (\Psi_X(\xi) - \Phi_X(\xi)) \lhd w^{B*}_t\sqrt{M} & 0 \end{array}\right)
\\\stackrel{\eqref{eq:lem:KKarantine3}}{\equiv}&
\left(\begin{array}{cc} \theta^{B}_{0,X}\Psi_X(\xi) \lhd \sqrt{1-M} & \theta^{B}_{0,X}\Psi_X(\xi) \lhd \sqrt{M} \\ -w^B_t\rhd \theta^{B}_{0,X}\Phi_X(\xi) \lhd w^{B*}_t\sqrt{M} & w^B_t\rhd \theta^{B}_{0,X}\Phi_X(\xi) \lhd w^{B*}_t\sqrt{1-M} \end{array}\right)
\\={}&
\left(\Ad\left(\begin{array}{cc}1&0\\0&w^B_t\end{array}\right)
\left(\begin{array}{cc}\theta^{B}_{0,X}\Psi_X(\xi)&0\\0&\theta^{B}_{0,X}\Phi_X(\xi)\end{array}\right)
\right) \lhd \left(\begin{array}{cc}\sqrt{1-M}&\sqrt{M}\\-\sqrt{M}&\sqrt{1-M}\end{array}\right)
\end{align*}
mod $\cK(B^\rs,\beta^\rs(X))\otimes M_2$, where in the third equivalence we used the following equivalences mod $\cK(B^\rs,\beta^\rs(X))$, 
\begin{align*}
    w^B_t\rhd \theta^{B}_{0,X}\Psi_X(\xi)\lhd w^{B*}_t\sqrt{1-M} 
    &\equiv 
    \theta^{B}_{0,X}\Psi_X(\xi) \lhd\sqrt{1-M} ,
    \\
    w^B_t\rhd \theta^{B}_{0,X}\Phi_X(\xi) \lhd w^{B*}_t\sqrt{1-M}
    &\equiv
    \theta^{B}_{0,X}\Phi_X(\xi) \lhd \sqrt{1-M} ,
\end{align*}
which hold by focusing on the upper right entry of the inclusion $\cA_{2,X}\sigma_X(\sqrt{1-M})\subset \cJ_X$ following from \eqref{eq:lem:KKarantine1}. 
As in the proof of \cite[Lemma 2.8]{GASZ22}, we can define the desired norm-continuous path $V\colon [0,3]\to\cU(\fD_{\Theta})$ by 
\[V_t=\left\{\begin{array}{cl}
\exp(t\log U) & (\text{if $0\leq t\leq 1$}) \\
(1\oplus_{s}w_{t-1}^{B*})U(w_{t-1}^B\oplus_{s}1) & (\text{if $1\leq t\leq 2$}) \\
(1\oplus_{s}w_1^{B*})\exp((3-t)\log U)(w_1^B\oplus_{s}1) & (\text{if $2\leq t\leq 3$}) \\
\end{array}\right. ,\]
where $\log U\in \fD_{\Theta}$ is well-defined with an appropriate holomorphic branch of $\log\colon\{z\in\bC^\times \,|\, \mathop{\mathrm{Re}}z\geq 0 \}\to \bC$ thanks to the fact that $U+U^*\geq 0$. 
\end{proof}

\begin{theorem}\label{thm:KKCtostableoh}
Let $(A,\alpha,\fu)$ and $(B,\beta,\fv)$ be $\cC$-C$^*$-algebras with $A$ separable and $B$ $\sigma$-unital, and $(\phi,\psi)\in\bE^\cC((\alpha,\fu),(\beta,\fv))$. 
If $[\phi,\psi]=0\in\KK^\cC((\alpha,\fu),(\beta,\fv))$, then $\phi$ and $\psi$ are stably $\cC$-operator homotopic. 
\end{theorem}

\begin{proof}
The conclusion follows by applying \cref{lem:KKarantine} to the element in $\bE^{\cC}((\alpha,\fu),(\beta\otimes\id_{C[0,1]},\fv\otimes 1))$ inducing the homotopy between $[\phi,\phi]$ and $[\phi,\psi].$ 
\end{proof}

\section{Absorbing \texorpdfstring{$\cC$}{C}-cocycle representations}\label{sec:absorbing}

In this section we will show the existence of absorbing representations between separable $\cC$-C$^*$-algebras. Before we introduce absorbing representations we need a notion of asymptotic unitary equivalence analogous to \cite[Definition 2.1]{DADEIL01} (see also \cite[Definition 3.1]{GASZ22}).

\begin{defn}\label{defn:asympunitaryeq}
    Let $\phi,\psi\colon(A,\alpha,\fu)\rightarrow (B,\beta,\fv)$ be $\cC$-cocycle representations.
    We write $\phi\sim_{\cC}\psi$ if there exists a norm-continuous path $u\colon[1,\infty)\rightarrow \cU(\cM(B))$ such that for all $X\in \cC$ and $\xi\in \alpha(X)$
    \begin{enumerate}[label=(\roman*)]
        \item $\psi_X(\xi)=\lim_{t\rightarrow\infty} u_t\rhd \phi_X(\xi)\lhd u_t^*$, and \label{item:conv}
       \item $\psi_X(\xi)-u_t\rhd \phi_X(\xi)\lhd u_t^*\in \cK(B,\beta(X))$ for all $t\geq 1$, \label{item:Dphiu}
    \end{enumerate}
    where the limit in condition \ref{item:conv} is with respect to the norm topology on $\cL(B,\beta(X))$.
\end{defn}
It is straightforward to check that the relation $\sim_{\cC}$ is an equivalence relation. Furthermore, if $\theta,\phi,\psi\colon(A,\alpha,\fu)\rightarrow (B,\beta,\fv)$ are $\cC$-cocycle representations with $\theta\sim_{\cC}\phi$ through a unitary path $u_t$ and $B$ is stable, then the path given by $u_t\oplus 1$ witnesses that $\theta\oplus \psi\sim_{\cC}\phi\oplus \psi$.

\begin{rmk}\label{rmk: AsympGroupCase}
In the case of genuine group actions by a countable discrete group $G$ and cocycle representations $(\phi,\mathbbm{u}),(\psi,\mathbbm{v})$ in the sense of \cite{GASZ22}, the relation $\sim_\cC$ defined above is genuinely different than $\sim_{\asy}$ from \cite[Definition 3.1]{GASZ22}. Using Remark \ref{rmk: GroupCase}, we see that $\phi\sim_\cC\psi$ if there exists a norm-continuous path $u\colon[1,\infty)\rightarrow \cU(\cM(B))$ such that for any $g\in G$ and $a\in A$, we have that $$\mathbbm{v}_g^*\psi(a)=\lim_{t\rightarrow\infty} \beta_g(u_t) \mathbbm{u}_g^*\phi(a) u_t^*$$ and 
$$\mathbbm{v}_g^*\psi(a)- \beta_g(u_t) \mathbbm{u}_g^*\phi(a) u_t^*\in B, \ t\geq 1.$$A straightforward calculation shows that this is equivalent to the following conditions to hold for any $a\in A$, $g\in G$, and $t\geq 1$:
\begin{itemize}
    \item $\psi(a)=\lim_{t\rightarrow\infty} u_t\phi(a) u_t^*$;
    \item $\psi(a)\mathbbm{v}_g=\lim_{t\rightarrow\infty} \psi(a)u_t\mathbbm{u}_g\beta_g(u_t^*)$;
    \item $\psi(a)-u_t\phi(a)u_t^*\in B$;
    \item $\psi(a)(\mathbbm{v}_g-u_t\mathbbm{u}_g\beta_g(u_t^*))\in B$.
\end{itemize}
For example if $\phi$ and $\psi$ are the zero $^*$-homomorphisms and $\bu$ is a $\beta$-1-cocycle then $(0,\bu)\sim_\cC(0,\bf{1})$ always as all the bullet points above are vacuous (whereby $\bf{1}$ denotes the trivial $\beta$-1-cocycle).
Whereas, if $B$ is stable then $(0,\bu)\sim_{\asy}(0,\bf{1})$ implies that $\bu$ is an asymptotic coboundary in the sense of \cite[Definition~2.2]{SZ17}.
\end{rmk}

\begin{rmk}\label{rmk: asympcomments}
If $\phi,\psi\colon (A,\alpha,\fu)\rightarrow (B^\rs,\beta^\rs,\fv^\rs)$ and the pair $(\phi,\psi)$ forms a $\cC$-Cuntz pair from $(\alpha,\fu)$ to $(\beta,\fv)$, then condition \ref{item:Dphiu} above is precisely saying that $u_t\in \fD_{\phi}$ for all $t\geq 1$. In this case, we have that $[\Ad(u_1)\circ \phi,\psi]=0$ via the homotopy
\begin{equation*}
    (\phi^{(t)},\psi^{(t)})=\begin{cases}
        (\psi,\psi),\quad \quad \quad \quad t=0,\\
        (\Ad(u_{1/t})\phi,\psi),\quad t\in (0,1].
    \end{cases}
\end{equation*}\end{rmk}
Note that by Proposition \ref{prop:phiXmaps} it suffices to quantify  only over $X\in \Irr(\cC)$ in Definition \ref{defn:asympunitaryeq}. We may now introduce absorbing representations.
\begin{defn}
  Let $(A,\alpha,\fu)$ and $(B,\beta,\fv)$ be $\cC$-C$^*$-algebras with $B$ stable and let $\fC$ be a class of $\cC$-cocycle representations from $(\alpha,\fu)$ to $(\beta,\fv)$. We say that a $\cC$-cocycle representation $\theta\colon(A,\alpha,\fu)\rightarrow (B,\beta,\fv)$ is \emph{absorbing for $\fC$} if for all $\phi\in\fC$ 
  \[\theta\oplus\phi\sim_{\cC}\theta.\] 
  If $\fC$ is the class of all $\cC$-cocycle representations from $(\alpha,\fu)$ to $(\beta,\fv)$, we simply call $\theta$ \emph{absorbing}.
\end{defn}

\begin{rmk}\label{rmk:uniquenessabsorbing}
    Note that absorbing representations, if they exist, are unique up to $\sim_{\cC}$. Indeed if $\theta,\vartheta\colon(A,\alpha,\fu)\rightarrow (B,\beta,\fv)$ are two absorbing representations then $\theta\sim_{\cC} \theta\oplus \vartheta\sim_{\cC} \vartheta$.
\end{rmk}

A key technical tool to construct absorbing representations is the notion of an infinite direct sum of $\cC$-cocycle representations. We briefly outline the construction. Let $(A,\alpha,\fu)$ and $(B,\beta,\fv)$ be $\cC$-C$^*$-algebras with $B$ stable and $\phi^{(n)}\colon(A,\alpha,\fu)\rightarrow (B,\beta,\fv)$ be a sequence of $\cC$-cocycle representations for $n\in \bN$. Let $r_n\in \cM(B)$ be a sequence of isometries such that $\sum_{n=1}^\infty r_nr_n^*=1$ in the strict topology. We then define the infinite direct sum of $\{\phi^{(n)}\}_{n\in \bN}$ by \[\bigoplus_{n=1}^\infty \phi^{(n)}_X=\sum_{n=1}^\infty r_n\rhd \phi^{(n)}_X\lhd r_n^*, \ X\in\cC.\]  
To see that this is indeed well-defined, we first need to show that \[\bigoplus_{n=1}^\infty \phi^{(n)}_X(\xi)\in\cL(B,\beta(X))\] for any $X\in\cC$ and $\xi\in\alpha(X)$. Consider the unitary $U\in \cL(B,B^{\infty})$ defined by $b\mapsto (r_1^*b,r_2^*b,\ldots)$ and the map $\Phi_X(\xi)\in \cL(B^{\infty},\beta(X)^{\infty})$ for $\xi\in \alpha(X)$ defined by $\Phi_X(\xi)(b_1,b_2,\ldots)=(\phi_X^{(1)}(\xi)(b_1),\phi_X^{(2)}(\xi)(b_2),\ldots)$. Then  \[\bigoplus_{n=1}^\infty\phi_X^{(n)}(\xi)=(U^*\otimes \id_{\beta(X)})\circ \Phi(\xi)\circ U.\] Notice that for this last composition to make sense we have used the canonical identification of $\beta(X)^{\infty}$ with $B^{\infty}\otimes \beta(X)$. Furthermore, a standard computation shows that $\bigoplus_{n=1}^\infty\phi^{(n)}$ satisfies the conditions of Proposition \ref{prop:phiXmaps} and hence is a $\cC$-cocycle representation from $(A,\alpha,\fu)$ to $(B,\beta,\fv)$. Moreover, as in the case of finite direct sums, $\bigoplus_{n=1}^\infty\phi^{(n)}(\xi)$ only depends on the sequence $r_n$ up to unitary equivalence. Precisely, if $v_n\in\mathcal{M}(B)$ is another sequence of isometries satisfying the same relation, then the unitary defined by $u =\sum\limits_{n=1}^\infty r_nv_n^*$ implements this equivalence. If the sequence $\phi^{(n)}$ is constant and equal to a $\cC$-cocycle representation $\phi$ for all $n\in \bN$ then we denote $\bigoplus_{n=1}^\infty \phi$ by $\phi^{\infty}$ and call it the \emph{infinite repeat} of $\phi$.
%\begin{rmk}\label{rmk:actionnormvsevaluationnorm}
    %Note that if $T\in \cL(B,E)$ with $E$ a non-degenerate $B$-correspondence and $b,c\in B$, then the norm of the operator $b\rhd T\lhd c\in \cL(B,E)$ coincides with the norm of the element $b\rhd T(c)$ in $E$. Indeed 
    %\[\|b\rhd T\lhd c\|=\sup_{a\in A_1}\|b\rhd T\lhd c(a)\|=\sup_{a\in A_1}\|b\rhd T(c)\lhd a\|\leq \|b\rhd T(c)\|,\] 
    %and conversely, taking an approximate unit of contractions $u_{\lambda}\in B$, it follows from the continuity of $T$ that
    %\[\|b\rhd T(c)\|=\lim_{\lambda\in \Lambda}\|b\rhd T(cu_{\lambda})\|\leq \sup_{a\in A_1}\|b\rhd T\lhd c(a)\|=\|b\rhd T\lhd c\|.\] 
    %Following the same argument it is also true that if you take $n\in \bN$, $T_i\in \cL(B,E)$ and $b_i$, $c_i\in B$ for $1\leq i \leq n$ the equality
    %$\|\sum_{i=1}^n b_i\rhd T_i\lhd c_i\|=\|\sum_{i=1}^nb_i\rhd T_i(c_i)\|$ holds.
    %We will use these norm equalities often during the remainder of this section.
%\end{rmk}

\par We may now introduce two notions of $\cC$-weak containment of $\cC$-cocycle representations that will be crucial in showing the existence of absorbing $\cC$-cocycle representations. These notions are motivated by \cite[Definition 3.3]{GASZ22} and \cite[Definition 2.11]{DADEIL02}(see also \cite[Definition 3.5]{GASZ22}). Recall in the definitions below that $B_1$ is the set of positive contractions in $B$.

\begin{defn}\label{defn:weakcontainment}
 Let $\phi\colon(A,\alpha,\fu)\rightarrow (B,\beta,\fv)$ be a $\cC$-cocycle representation and $\fC$ be a family of $\cC$-cocycle representations from $(A,\alpha,\fu)$ to $(B,\beta,\fv)$. We say that $\phi$ is \emph{$\cC$-weakly contained} in $\fC$, denoted $\phi\preccurlyeq_{\cC} \fC$, if for all $\varepsilon>0$, $b\in B_1$, finite $K\subset \Irr(\cC)$, and compact sets $\cF_X\subset \alpha(X)$ for $X\in K$, there exist $\psi^{(1)}, \psi^{(2)},\ldots ,\psi^{(n)}\in \fC$ and elements $\{c_{j,k} \,|\, j=1,2,\ldots ,n,\ k=1,2,\ldots ,N\}\subset B$ such that for all $X\in K$
 \begin{equation}\label{eqn:weak}
     \max_{\xi\in \cF_X}\|b^*\rhd \phi_X(\xi)\lhd b-\sum_{j=1}^n\sum_{k=1}^N c_{j,k}^*\rhd \psi_X^{(j)}(\xi)\lhd c_{j,k}\|\leq \varepsilon
 \end{equation}
 and
 \begin{equation}\label{eqn:unitweak}
     \|b^*b-\sum_{j=1}^n\sum_{k=1}^Nc_{j,k}^*c_{j,k}\|\leq \varepsilon.
 \end{equation}
\end{defn}

\begin{defn}\label{defn:weakcontaimentinfinity}
    Let $\phi,\psi\colon(A,\alpha,\fu)\rightarrow (B,\beta,\fv)$ be $\cC$-cocycle representations. We say that $\phi$ is \emph{$\cC$-contained in} $\psi$ \emph{at infinity} if for all $\varepsilon>0$, $b\in B_1$, finite $K\subset \Irr(\cC)$, and compact sets $\cF_X\subset \alpha(X)$ for $X\in K$, there exists an element $x\in B$ such that for all $X\in K$
     \begin{equation}\label{eqn:weakinf}
     \max_{\xi\in \cF_X}\|b^*\rhd \phi_X(\xi)\lhd b-x^*\rhd \psi_X(\xi)\lhd x\|\leq \varepsilon,
 \end{equation}
 \begin{equation}\label{eqn:unitcondweak}
      \|b^*b-x^*x\|\leq \varepsilon,
 \end{equation}
 and
 \begin{equation}\label{eqn:orthogonalityweak}
    \|x^*b\|\leq \varepsilon.
 \end{equation}
\end{defn}
In the definitions above we could also have taken $b\in B$ by rescaling it by its norm.
\begin{rmk}
When restricted to group actions, with $(A,\alpha)$ being $\bC$ with the trivial $G$-action, and $(B,\beta)$ being $\bK$ with the trivial $G$-action, $\Hilb(G)^{\op}$-cocycle representations $\bC\rightarrow \bK$, with a non-degenerate underlying $^*$-homomorphism, coincide with unitary representations of $G$ on $\ell^2(\bN)$. In this case, our notion of $\cC$-weak containment coincides with that of weak containment of group representations (see \cite[Remark 3.4]{GASZ22}). We warn the reader though that our notions of $\cC$-weak containment and $\cC$-containment at infinity do not seem to coincide with those of \cite{GASZ22} when restricted to group actions and cocycle representations between them; not even a priori in the case that the cocycle representations are non-degenerate. Precisely, if we take $(\phi,\mathbbm{u})$ to be a cocycle representation in the sense of \cite{GASZ22} and $\mathfrak{C}$ be a class of cocycle representations in the sense of \cite{GASZ22}, then, with the same choices of sets as in Definition \ref{defn:weakcontainment}, $\phi\preccurlyeq_{\cC} \fC$ if there exist $(\psi^{(1)},\mathbbm{v}^{(1)}), (\psi^{(2)},\mathbbm{v}^{(2)}),\ldots ,(\psi^{(n)},\mathbbm{v}^{(n)})\in \fC$ and elements $\{c_{j,k} \,|\, j=1,2,\ldots ,n,\ k=1,2,\ldots ,N\}\subset B$ such that for all $g\in K$
 \begin{equation*}
     \max_{a\in \cF}\|\beta_g(b^*)\mathbbm{u}_g^*\phi(a)b-\sum_{j=1}^n\sum_{k=1}^N \beta_g(c_{j,k}^*)(\mathbbm{v}_g^{(j)})^*\psi^{(j)}(a) c_{j,k}\|\leq \varepsilon
 \end{equation*}
 and \eqref{eqn:unitweak} holds. Likewise, one can see that $\cC$-containment at infinity restricts to a different notion in general than the one in \cite[Definition 3.5]{GASZ22}. The same example as considered in Remark \ref{rmk: AsympGroupCase} shows that these two notions are genuinely different to those considered in \cite{GASZ22} in the case of degenerate cocycle representations.

Finally, we also note that conditions \eqref{eqn:unitweak} and \eqref{eqn:unitcondweak} can be removed if the $\cC$-cocycle representations considered in the hypothesis are unital. 
\end{rmk}

If the class $\fC$ in Definition \ref{defn:weakcontainment} consists of one $\cC$-cocycle representation $\psi$ then we simply write $\phi\preccurlyeq_{\cC} \psi$. In this case $\preccurlyeq_{\cC}$ is reflexive and transitive. Clearly $\cC$-containment at infinity implies $\cC$-weak containment.

 \begin{rmk}\label{rmk:simplificationsweak}
     To show either $\cC$-weak containment or $\cC$-containment at infinity, it suffices to quantify only over compact sets $\cF_X\subset \alpha(X)_1$. Moreover, one may assume that $c_{j,k}$ and $x$ in Definition \ref{defn:weakcontainment} and Definition \ref{defn:weakcontaimentinfinity} are contractions. By considering approximate units and using that each $\beta(X)$ is non-degenerate, Definition \ref{defn:weakcontaimentinfinity} is equivalent to the a priori stronger statement as follows:
     \par 
     For all $\varepsilon>0$, $b\in B_1$, finite $K\subset\Irr(\cC)$, compact sets $\cF_X\subset \alpha(X)$ and  $\cG_X\subset \beta(X)$ for $X\in K$ one may choose $x$ satisfying \eqref{eqn:weakinf}, \eqref{eqn:unitcondweak}, \eqref{eqn:orthogonalityweak} and that $\sup_{\eta\in \cG_X}\|x^*\rhd \eta\|<\varepsilon$ for all $X\in K$. 
     
     To see this, note that using an approximate unit of $B$ we can take a contractive positive element $e\in B$ such that $\|eb- b\|<\frac{\varepsilon}{3+3\|\xi \|}$ and $\| e\eta -\eta \|<\frac{\varepsilon}{3}$ for all $X\in K$, $\xi\in\cF_X$, and $\eta\in\cG_X$. Applying the original Definition \ref{defn:weakcontaimentinfinity} to $(e,\varepsilon/3)$ instead of $(b,\varepsilon)$ we obtain $y\in B$ such that 
     \begin{equation*}
     \max_{\xi\in \cF_X}\|e^*\rhd \phi_X(\xi)\lhd e-y^*\rhd \psi_X(\xi)\lhd y\|\leq \varepsilon/3,
     \end{equation*}
     $\|e^*e-y^*y\|\leq \varepsilon/3$,
     and
     $\|y^*e\|\leq \varepsilon/3$.
     Then, $x\coloneqq yb$ satisfies the desired conditions. 
 \end{rmk}

\begin{rmk}\label{rmk:topologyHomcC}
For $\cC$-C$^*$-algebras $(A,\alpha,\fu)$ and $(B,\beta,\fv)$, we equip each $\cB(\alpha(X),\cL(B,\beta(X)))$ with the topology given by the family of seminorms $\phi_X \mapsto \| \phi_X(\xi) (b) \|$ over all $\xi\in \alpha(X), b\in B$, which induces a topology on $\mathfrak{X}\coloneqq \prod_{X\in\Irr(\cC)}\cB(\alpha(X),\cL(B,\beta(X)))$ by the product topology. Therefore, we obtain a topology on $\Hom^\cC((\alpha,\fu),(\beta,\fv))$ via the canonical injection into $\mathfrak{X}$. 
If $A$ and $B$ are separable, this topology on the unit ball of $\cB(\alpha(X),\cL(B,\beta(X)))$ is second countable and metrisable when $\cB(\alpha(X),\cL(B,\beta(X)))_1$ is equipped with the metric \[d(\phi_X,\psi_X)\coloneqq \sum\limits_{m,n=1}^{\infty} 2^{-m-n}\| \phi_{X}(\xi_{m})\lhd b_n - \psi_{X}(\xi_{m})\lhd b_n \|,\] where we chose dense subsets $\{\xi_{m}\}_{m\in\bN}\subset \alpha(X)_1$ and $\{ b_n \}_n \subset B_1$ in the unit balls. If moreover $\Irr(\cC)$ is countable, as $\Hom((\alpha,\fu),(\beta,\fv)))$ is contained $\prod_{X\in\Irr(\cC)}\cB(\alpha(X),\cL(B,\beta(X)))_1$ by \cref{prop:phiXmaps}, it follows that the topology on $\Hom^\cC((\alpha,\fu),(\beta,\fv))$ is second countable and metrisable.

We say $\phi,\psi\in\Hom^\cC((\alpha,\fu),(\beta,\fv))$ are approximately unitarily equivalent and write $\phi\au\psi$ if $\phi$ is contained in the closure of $\{ \Ad(u) \circ \psi \,|\, u\in \cU(\cM(B)) \}$. 
\end{rmk}

\begin{comment}
    As for second countability, for any $m,n\in\bN$, we have a countable dense subset $\{g_{m,n,l}\}_l\subset \cB(\mathrm{span}\{ \xi_1,\cdots,\xi_m \}\otimes \mathrm{span}\{b_1,\cdots,b_n\}, \beta(X))$ with respect to the norm, and if we apply Hahn--Banach theorem twice to $g_{m,n,l}$, we obtain $\widetilde{g}_{m,n,l}\in\cB(\alpha(X),\cL(B,\beta(X)))$, which forms a countable dense subset. 
\end{comment}

\begin{lemma}\label{lem:densecontainment}
    Let $\phi,\psi\colon(A,\alpha,\fu)\rightarrow (B,\beta,\fv)$ be two $\cC$-cocycle representations and $\fC$ be a collection of $\cC$-cocycle representations such that $\phi$ is contained in the closure of $\fC$. Then $\phi\preccurlyeq_{\cC} \fC$. In particular, if $\phi\au \psi$ then $\phi\preccurlyeq_{\cC} \psi$.
\end{lemma}

\begin{proof}
    Let $\varepsilon>0$, $K\subset \Irr(\cC)$ finite, $b\in B_1$, and $\cF_X\subset \alpha(X)$ for $X\in K$ be compact. Then as $\phi$ is in the closure of $\fC$ there exists $\psi\in \fC$ such that for all $X\in K$
    \[\max_{\xi\in \cF_X}\|\phi_X(\xi)\lhd b-\psi_X(\xi)\lhd b\|\leq \varepsilon.\] 
    As $b^*$ is a contraction, we get that
    \[\max_{\xi\in \cF_X}\|b^*\rhd \phi_X(\xi)\lhd b-b^*\rhd\psi_X(\xi)\lhd b\|\leq \max_{\xi\in \cF_X}\|\phi_X(\xi)\lhd b-\psi_X(\xi)\lhd b\|\leq \varepsilon.\] 
    Therefore, the conditions of Definition \ref{defn:weakcontainment} are satisfied with $l=1$, $N=1$, $\psi^{(1)}=\psi$ and $c_{1,1}=b$. 
    %\nn{Shall we omit this last part? I think it's clear -RN}If there exists a net $u_\lambda\in \cU(\cM(B))$ so that $\Ad(u_\lambda)\psi\rightarrow \phi$ then with $\varepsilon, K, b$ and $\cF_X$ as above, the same argument for $\fC=\{\Ad(u_\lambda)\psi\colon\lambda\in \Lambda\}$ yields that there exists a unitary $u\in \cU(\cM(B))$ such that for all $X\in K$
    %\[\max_{\xi\in \cF_X}\|b^*\rhd \phi_X(\xi)(b)-(u^*b)^*\rhd\psi_X(\xi)(u^*b)\|<\varepsilon.\] 
    %Thus, taking $l=1$, $N=1$, $\psi^{(1)}=\psi$ and $c_{1,1}=u^*b$ in Definition \ref{defn:weakcontainment} one has that $\phi\preccurlyeq_{\cC} \psi$.
\end{proof}

\begin{lemma}\label{lem:containmentinsum}
    Let $\phi,\psi^{(n)}\colon(A,\alpha,\fu)\rightarrow (B,\beta,\fv)$ be $\cC$-cocycle representations for $n\in \bN$. If we further assume that $B$ is stable, then 
    \[\phi\preccurlyeq_{\cC} \{\psi^{(n)}\}_{n\in\bN}\iff \phi\preccurlyeq_{\cC} \bigoplus_{n=1}^\infty\psi^{(n)}.\footnote{Here $\bigoplus_{n=1}^\infty\psi^{(n)}$ is interpreted as the representation with the associated collection of linear maps given by $\sum_{n=1}^\infty r_n\rhd \psi_X^{(n)}\lhd r_n^*$ for a fixed sequence of isometries $r_n\in\cM(B)$ such that $\sum_{n=1}^\infty r_nr_n^*=1$.}\]
    In particular, $\phi^{\infty}\preccurlyeq_{\cC} \phi$ and $\phi\preccurlyeq_{\cC} \phi^\infty$.
\end{lemma}
\begin{proof}
Let $r_n\in \cM(B)$ be a sequence of isometries such that $\sum_{n=1}^\infty r_nr_n^*=1$ in the strict topology and let $\bigoplus_{n=1}^\infty\psi^{(n)}_X=\Psi_X$ for any $X\in\cC$. We start by showing the only if direction. By hypothesis, for all $K\subset \Irr(\cC)$ finite, $\cF_X\subset \alpha(X)$ for $X\in K$ compact, $\varepsilon>0$, and $b\in B$ there exist $\{c_{j,k} \,|\, 1\leq j \leq n, 1\leq k\leq N\}\subset B$  such that for all $\xi \in \cF_X$ and $X\in K$
     \[\|b^*\rhd \phi_X(\xi)\lhd b-\sum_{j=1}^n\sum_{k=1}^N c_{j,k}^*\rhd \psi_X^{(j)}(\xi)\lhd c_{j,k}\|\leq \varepsilon\]  
     and
     \[\|b^*b-\sum_{j=1}^n\sum_{k=1}^Nc_{j,k}^*c_{j,k}\|\leq\varepsilon.\] 
     Let $d_{j,k}=r_jc_{j,k}$ for $1\leq j \leq n$ and $1\leq k \leq N$. Then, for any $X\in K$ and $\xi\in \cF_X$,
     \begin{align*}
         b^*\rhd \phi_X(\xi)\lhd b&=_\varepsilon \sum_{j=1}^n\sum_{k=1}^Nc_{j,k}^*\rhd \psi_X^{(j)}(\xi)
     \lhd c_{j,k}\\
     &=\sum_{k=1}^N\sum_{j=1}^nd_{j,k}^*\rhd \Psi_X(\xi)\lhd d_{j,k}.
     \end{align*}
     Moreover, $b^*b=_\varepsilon \sum_{j=1}^n\sum_{k=1}^Nd_{j,k}^*d_{j,k}$ as required.
     We now turn to the if direction. For any $K\subset \Irr(\cC)$ finite, $\cF_X\subset \alpha(X)_1$ compact for $X\in K$, $\varepsilon>0$, and $b\in B_1$ choose contractions $\{c_{k} \,|\, 1\leq k\leq N\}$ such that for any $X\in K$ and $\xi\in\cF_X$, 
     \[\|b^*\rhd \phi_X(\xi)\lhd b-\sum_{k=1}^Nc_k^*\rhd \Psi_X(\xi)\lhd c_k\|\leq \varepsilon\] 
     and 
     \[\|b^*b-\sum_{k=1}^Nc_k^*c_k\|\leq \varepsilon.\] 
     Choose $n\in \bN$ such that $\sum_{j=1}^nr_jr_j^*c_k=_{\varepsilon/2N} c_k$ for all $1\leq k \leq N$. Then letting $f=\sum_{j=1}^nr_jr_j^*$ and $d_{j,k}=r_j^*c_k$, it follows that for $\xi\in \cF_X$ for $X\in K$
     \begin{align*}
         b^*\rhd\phi_X(\xi)\lhd b&=_\varepsilon \sum_{k=1}^Nc_k^*\rhd \Psi_X(\xi)\lhd c_k\\
         &=_\varepsilon \sum_{k=1}^N(c_k^*f^*)\rhd\Psi_X(\xi)\lhd (fc_k)\\
         &=\sum_{j=1}^n\sum_{k=1}^N d_{j,k}^*\rhd \psi^{(j)}_X(\xi)\lhd d_{j,k}
     \end{align*}
     and 
     \[b^*b=_\varepsilon \sum_{k=1}^Nc_k^*c_k=_\varepsilon \sum_{j=1}^n\sum_{k=1}^Nd_{j,k}^*d_{j,k}\] 
     as required.
\end{proof}

We now show a $\cC$-equivariant version of \cite[Lemma 3.9]{GASZ22} which is a useful sufficient criterion to obtain absorption of $\cC$-cocycle representations.

\begin{lemma}\label{lem:isometryabsorption}
    Let $\phi,\psi\colon(A,\alpha,\fu)\rightarrow (B,\beta,\fv)$ be $\cC$-cocycle representations with $B$ stable. Suppose that there exists a sequence of isometries $s_n\in\cM(B)$ such that for all $\xi\in \alpha(X)$ and $X\in \cC$
    \begin{enumerate}[label=\textit{(\roman*)}]
    \item\label{item:lem:isometryabsorption1} 
    $\|s_n\rhd \psi_X(\xi)-\phi_X(\xi)\lhd s_n\|\rightarrow 0$,
    \item\label{item:lem:isometryabsorption2} 
    $s_n\rhd \psi_X(\xi)-\phi_X(\xi)\lhd s_n\in \cK(B,\beta(X))$,
    \item\label{item:lem:isometryabsorption3} 
    $s_{n+1}^*s_n=0,\ \forall n\geq 1$,
    \end{enumerate}
    and moreover that there exists a unital inclusion $\iota\colon\cO_2\rightarrow \cM(B)$ with $\iota(a)\rhd\psi_X(\xi)=\psi_X(\xi)\lhd \iota(a)$ for all $\xi\in \alpha(X)$, $X\in \cC$, and $a\in \cO_2$. Then $\phi\sim_{\cC} \psi\oplus \phi$.
\end{lemma}
\begin{proof}
    We follow the strategy in \cite[Lemma 2.3, Lemma 2.4]{DADEIL01} (see also \cite[Lemma 3.8, Lemma 3.9]{GASZ22}). Firstly $s_n$ can be extended to a norm-continuous path $(s_t)_{t\geq 1}$ of isometries still satisfying \ref{item:lem:isometryabsorption1} and \ref{item:lem:isometryabsorption2} by setting $s_{n+t}=(1-t)^{1/2}s_n+t^{1/2}s_{n+1}$ for all $t\in [0,1]$. Moreover, note that for any $X\in \cC$ and $\xi\in \alpha(X)$, combining  \cref{lemma:barnotation,lem:usingisometryofstar} gives that
    \begin{align*}
    \|s_t^*\rhd \phi_X(\xi)-\psi_X(\xi)\lhd s_t^*\|
    &\leq d_X^{1/2}\|\phi_{\overline{X}}(\overline{\xi})\lhd s_t-s_t\rhd \psi_{\overline{X}}(\overline{\xi})\|\longrightarrow 0
    \end{align*}
    as $t\rightarrow \infty$. Using Lemma \ref{lem:usingisometryofstar} again, it follows that \[s_t^*\rhd \phi_X(\xi)-\psi_X(\xi)\lhd s_t^*\in \beta(\mu_X^{-1})\circ\overline{\cK(B,\beta(\overline{X}))}\subset \cK(B,\beta(X)).\] Let $r_1,r_2\in \cM(B)$ be the images of the generators of $\cO_2$ through $\iota$. Then consider the norm-continuous path $u\colon[1,\infty)\rightarrow \cU(\cM(B))$ given by
    \[u_t=(r_1r_1^*+r_2s_tr_2^*)s_t^*+r_2(1-s_ts_t^*).\] 
    Now, using the properties of $s_t$, for any $\varepsilon>0$, $X\in \cC$, and $\xi\in \alpha(X)$ one can choose $n_0$ large enough so that for any $t\geq n_0$
    \[
    \lVert s_t\rhd\phi_X(\xi)-\phi_X(\xi)\lhd s_t\rVert+\lVert s_t^*\rhd\phi_X(\xi)-\phi_X(\xi)\lhd s_t^*\rVert<\varepsilon.
    \]
    Thus, as $s_t$ are isometries for any $t\geq n_0$ one has
    {\small
    \begin{align*}
        & u_t\rhd \phi_X(\xi)\lhd u_t^*\\
        &=(r_1r_1^*+r_2s_tr_2^*)s_t^*\rhd \phi_X(\xi)\lhd s_t(r_1r_1^*+r_2s_t^*r_2^*)\\ 
        &\quad +(r_1r_1^*+r_2s_tr_2^*)s_t^*
         \rhd \phi_X(\xi)\lhd (1-s_ts_t^*)r_2^*\\
         &\quad +r_2(1-s_ts_t^*)\rhd \phi_X(\xi)\lhd s_t(r_1r_1^*+r_2s_t^*r_2^*)\\
        &\quad +r_2(1-s_ts_t^*)\rhd \phi_X(\xi)\lhd (1-s_ts_t^*)r_2^*\\
        &=_{4\varepsilon}(r_1r_1^*+r_2s_tr_2^*)\rhd \psi_X(\xi)\lhd(r_1r_1^*+r_2s_t^*r_2^*)+r_2\rhd \phi_X(\xi)\lhd (1-s_ts_t^*)r_2^*\\
        &= r_1\rhd \psi_X(\xi)\lhd r_1^*+r_2s_t\rhd \phi_X(\xi)\lhd s_t^*r_2^*+r_2\rhd \phi_X(\xi)\lhd (1-s_ts_t^*)r_2^*\\
        &=_{\varepsilon}(\psi_X\oplus\phi_X)(\xi).
    \end{align*}
    }
    As $\varepsilon$ is arbitrary and each of the differences in the computation above are contained in $\cK(B,\beta(X))$, the statement follows.
    \end{proof}

The following lemma is a $\cC$-equivariant adaptation of \cite[Lemma 3.11]{GASZ22}.

\begin{lemma}\label{lem:biglemsec3}
    Let $A$ be a separable $\cC$-C$^*$-algebra, $B$ be a $\sigma$-unital and stable $\cC$-C$^*$-algebra, and $\phi,\psi\colon(A,\alpha,\fu)\rightarrow (B,\beta,\fv)$ be $\cC$-cocycle representations. Then the following are equivalent:
    \begin{enumerate}[label=\textit{(\roman*)}]
        \item $\psi$ is $\cC$-contained in $\phi$ at infinity;\label{item:1lem3.11}
        \item $\psi^{\infty}$ is $\cC$-contained in $\phi$ at infinity;\label{item:2lem3.11}
        \item There exists an isometry $S\in \cM(B)$ such that
        \[\psi_X^{\infty}(\xi)-S^*\rhd \phi_X(\xi)\lhd S\in \cK(B,\beta(X))\] 
        for all $X\in \cC$ and $\xi\in\alpha(X)$;\label{item:3lem3.11}
        \item $\phi\sim_{\cC} \psi^{\infty}\oplus \phi$.\label{item:4lem3.11}
    \end{enumerate}
\end{lemma}

\begin{proof}
    Let $r_n\in \cM(B)$ be isometries such that $\sum_{n=1}^\infty r_nr_n^* = 1$ in the strict topology.
    \ref{item:1lem3.11} $\Rightarrow$ \ref{item:2lem3.11} Choose a finite set $K\subset \Irr(\cC)$, compact subsets $\cF_X\subset \alpha(X)$ for $X\in K$, $\varepsilon>0$ and $b\in B_1$. Without loss of generality we may assume that (up to isomorphism of objects) $1_{\cC}\in K=\overline{K}$ and that any $\xi\in \cF_X$ satisfies $\|\xi\|\leq 1$ and $\overline{\xi}\in \cF_{\overline{X}}$. Let $C=\max_{X\in K}d_X^{1/2}$.
    Following the proof of \cite[Lemma 3.11]{GASZ22} (see also Remark \ref{rmk:simplificationsweak}) one can choose $N\in \bN$, and $b_0,x_1,\ldots, x_N\in B_1$ with $b_0$ positive such that letting $b_{0,N}=\sum_{k=1}^N r_kb_0r_k^*$ (note that $b_{0,N}$ is a positive contraction as $b_0$ is) one has that $\|b-b_{0,N}b\|\leq \varepsilon/4$, and for all $1\leq k \leq N$, $l>k$, $\xi\in\cF_X$, $X\in K$
    
    \begin{equation}\label{eq: longlemma1}
    \|b_0\rhd \psi_X(\xi)\lhd b_0-x_k^*\rhd \phi_X(\xi)\lhd x_k\|\leq \frac{\varepsilon}{8N},\end{equation}
    
    \begin{equation}\label{eq: longlemma2}
        \|b_0^2-x_k^*x_k\|\leq \frac{\varepsilon}{4N},\ \text{and}
    \end{equation}
    \begin{equation}\label{eqn:orthogonal}\|x_l^*b\|+\|x_l^*x_k\|+\|x_l^*\rhd \phi_X(\xi)\lhd x_k\|\leq \frac{\varepsilon}{4N^2C}.\end{equation}
   Similarly, it follows from \cref{lemma:barnotation,lem:usingisometryofstar}, and \eqref{eqn:orthogonal} that for $1\leq l< k\leq N$, $X \in K$ and $\xi\in \cF_X$
   \begin{align*}
       \|x_l^*\rhd \phi_X(\xi)\lhd x_k\|
       %&\leq d_X^{1/2}\|x_k^*\rhd \phi_{\overline{X}}(\overline{\xi})\lhd x_l\|\\ 
       \leq C\|x_k^*\rhd \phi_{\overline{X}}(\overline{\xi}) \lhd x_l \| \leq \frac{\varepsilon}{4N^2}.
   \end{align*}
   Set $x=\sum_{j=1}^Nx_jr_j^*$. For all $X\in K$ and $\xi\in \cF_X$ it follows that 
   \begin{align*}
   x^*\rhd \phi_X(\xi)\lhd x&\ = \ \sum_{k,l=1}^N (r_kx_k^*)\rhd \phi_X(\xi)\lhd (x_lr_l^*)\\
   &\stackrel{\eqref{eqn:orthogonal}}{=}_{\!\!\varepsilon/4} \sum_{k=1}^N (r_kx_k^*)\rhd \phi_X(\xi)\lhd (x_kr_k^*)\\
   &\stackrel{\eqref{eq: longlemma1}}{=}_{\!\!\varepsilon/4} \sum_{k=1}^N (r_kb_0)\rhd \psi_X(\xi)\lhd (b_0r_k^*)\\
   %&= \sum_{k,l=1}^N\sum_{j=1}^\infty (t_kb_0t_k^*t_j)\rhd \psi_X(\xi)\lhd (t_j^*t_{l}b_0t_{l}^*)\\
   &\ = \ b_{0,N}\rhd \psi^\infty_X(\xi)\lhd b_{0,N}
   \end{align*}
   Therefore, as $b_{0,N}b=_{\varepsilon/4} b$, we have that
   \[\|b^*\rhd \psi_X^\infty(\xi)\lhd b-(xb)^*\rhd \phi_X(\xi)\lhd (xb)\|\leq \varepsilon\] 
   for all $X\in K$ and $\xi\in \cF_X$. By \eqref{eqn:orthogonal}, we also have that $\|(xb)^*b\|\leq \varepsilon$. Furthermore, we get that
   \begin{align*}
       b^*b&\ =_{\varepsilon/2} b^*b_{0,N}b_{0,N}b
       =\sum_{j=1}^Nb^*r_jb_0^2r_j^*b
       \stackrel{\eqref{eq: longlemma2}}{=}_{\!\!\varepsilon/4}\sum_{j=1}^Nb^*r_jx_j^*x_jr_j^*b
       \stackrel{\eqref{eqn:orthogonal}}{=}_{\!\!\varepsilon/4} b^*x^*xb.
   \end{align*}
   \ref{item:2lem3.11} $\Rightarrow$ \ref{item:3lem3.11}. Let $K_n\subset \Irr(\cC)$ be increasing finite sets such that (up to isomorphism) $1_\cC\in K_n=\overline{K_n}$ and $\bigcup_{n\in \bN} K_n = \Irr(\cC)$. Moreover, by separability of $A$, each $\alpha(X)$ is a separable Banach space. Therefore, we may choose compact sets $\cF_X\subset \alpha(X)_1$ for $X\in \Irr(\cC)$ such that $\overline{\xi}\in \cF_{\overline{X}}$ whenever $\xi\in \cF_X$ and the linear span of each $\cF_X$ is dense in $\alpha(X)$. We also let $C_n=\max_{X\in K_n} d_X^{1/2}$.
   By Lemma \ref{lem:approxunit}, there exists a countable, increasing, approximate unit $(e_n)_{n\geq 1}$ in $B$ that is quasicentral for $\psi_X^{\infty}(\xi)$ for any $X\in \Irr(\cC)$ and $\xi \in \alpha(X)$.
   Fix $0<\varepsilon<1/2$. By passing to a subsequence of $e_n$ we may assume that
   \begin{equation}\label{eqn:approxunitlem3.11}
       \max_{\xi\in \cF_X}\|[\psi_{X}^{\infty}(\xi),f_n]\|\leq 2^{-n}\varepsilon,\ \forall X\in K_n,
   \end{equation}
   where $f_1=e_1^{1/2}$ and $f_n=(e_n-e_{n-1})^{1/2}$ for all $n\geq 2$. 
    Note also that $f_n$ are positive contractions and that $\sum_{n=1}^\infty f_n^2=1$ in the strict topology. We may apply the hypothesis inductively (see also Remark \ref{rmk:simplificationsweak}) to get contractions $x_n\in B$ satisfying
    \begin{equation}\label{eqn:inductive1}
        \max_{\xi\in \cF_X}\|f_n\rhd \psi_X^{\infty}(\xi)\lhd f_n-x_n^*\rhd \phi_X(\xi)\lhd x_n\|\leq 2^{-(n+1)}\varepsilon, \forall X\in K_n,
    \end{equation}
    \begin{equation}\label{eqn:inductive2}
        \|x_n^*x_n-f_n^2\|\leq 2^{-(n+1)}\varepsilon, 
    \end{equation}
    \begin{equation}\label{eqn:inductive3}
        \max_{\xi\in \cF_X}\|x_n^*\rhd \phi_X(\xi)\lhd x_j\|\leq C_n^{-1}2^{-(2+j+n)}\varepsilon,\ \forall X\in K_n,\ 1\leq j<n.
    \end{equation}
    \begin{equation}\label{eqn:inductive3.5}
        \|x_n^*e_j\|+\|x_n^*x_j\|\leq 2^{-(2+j+n)}\varepsilon,\ \forall  1\leq j<n.
    \end{equation}
 Applying the same argument as in \ref{item:1lem3.11}$\Rightarrow$\ref{item:2lem3.11}, \eqref{eqn:inductive3} also yields that for any $X\in K_n$ and $1\leq j<n$,
 \begin{equation}\label{eqn:inductive4}
     \max_{\xi\in \cF_X}\{\|x_j^*x_n\|,\|x_j^*\rhd \phi_X(\xi)\lhd x_n\|\}\leq 2^{-(2+j+n)}\varepsilon.
 \end{equation}
 \par We let $R=\sum_{n=1}^\infty x_n$, making sense of the convergence of this sum in the strict topology. That $R$ is well-defined follows exactly as in the proof of \ref{item:2lem3.11} $\Rightarrow$ \ref{item:3lem3.11} in \cite[Lemma 3.11]{GASZ22} by additionally making use of \eqref{eqn:inductive2} (to show this note that conditions \eqref{eqn:inductive1} and \eqref{eqn:inductive3.5} are used). 
 \par We are going to check that $R$ satisfies the required properties for $S$ apart from being an isometry. First note that by \eqref{eqn:inductive3} and \eqref{eqn:inductive4}, $R^*\rhd \phi_X(\xi)\lhd R-\sum_{n=1}^\infty x_n^*\rhd \phi_X(\xi)\lhd x_n$ is contained in $\cK(B,\beta(X))$ for all $X\in \Irr(\cC)$ and $\xi\in \cF_X$. Indeed, this difference is an absolutely convergent sum of elements of the form $x_n^*\rhd \phi_X(\xi)\lhd x_k \in \cK(B,\beta(X))$ (any element of this form is contained in $\cK(B,\beta(X))$ as $x_k\in B=\cK(B)$ and the right action is precomposition) and $\cK(B,\beta(X))$ is a Banach space. We will use this same argument a few times in the next two displayed computations. Let $X \in \Irr(\cC)$ and $\xi \in \cF_X$ then we have 
 \begin{equation*}
 \begin{array}{rcl}
     R^*\rhd \phi_X(\xi)\lhd R 
     %=\sum_{k,n<j}x_k^*\rhd \phi_X(\xi)\lhd x_n &+\sum_{n\geq j} x_n^*\rhd \phi_X(\xi)\lhd x_n &+\sum_{n\neq k\geq j} x_k^*\rhd \phi_X(\xi)\lhd x_n\\
     &\equiv& \sum_{n=1}^\infty x_n^*\rhd \phi_X(\xi)\lhd x_n\\
     &\stackrel{\eqref{eqn:inductive1}}{\equiv}& \sum_{n=1}^\infty f_n\rhd \psi_X^{\infty}(\xi)\lhd f_n\\
     &\stackrel{\eqref{eqn:approxunitlem3.11}}{\equiv}& \sum_{n=1}^\infty f_n^2\rhd \psi_X^{\infty}(\xi) \\
     &=& \psi_X^{\infty}(\xi) 
 \end{array}
 \end{equation*}
 mod $\cK(B,\beta(X))$. Furthermore, we have that
 \begin{align*}
     R^*R &=\sum_{n,k=1}^\infty x_n^*x_k
     \stackrel{\eqref{eqn:inductive3},\eqref{eqn:inductive4}}{\equiv_{\varepsilon}}
     \sum_{n=1}^\infty x_n^*x_n 
     \stackrel{\eqref{eqn:inductive2}}{{\equiv}_{\varepsilon}}  \sum_{n=1}^\infty f_n^2 
     = 1
 \end{align*}
 modulo $B$.
 In particular, $|R|^2=R^*R\in 1+B$.
 From $\|R^*R-1\|\leq 2\varepsilon<1$, we see that $|R|= (R^*R)^{1/2}$ is invertible, and hence $|R|^{-1}\in 1+B$ or $1-|R|^{-1}\in B$.
 Then $S\coloneqq R|R|^{-1}\in \cM(B)$ is an isometry and 
 \begin{align*}
     &\psi_X^{\infty}(\xi)-S^*\rhd \phi_X(\xi)\lhd S \in \cK(B,\beta(X))
 \end{align*}
 for all $X\in \Irr(\cC)$ thanks to the properties of $R$ and that $R-S=R(1-|R|^{-1})\in B$. As the claimed result holds for all $X\in \Irr(\cC)$ then it also holds for all $X\in \cC$ by using the semisimplicity of $\cC$ combined with the naturality of $\cC$-cocycle representations (\ref{item:naturality} Proposition \ref{prop:phiXmaps}).

    \ref{item:3lem3.11} $\Rightarrow$ \ref{item:1lem3.11}. Let $S$ be as in the hypothesis. For $b\in B_1$, $X\in \cC$ and $\xi\in \alpha(X)$, consider the sequence $x_n=Sr_nb$. First $x_n^*b=b^*r_n^*S^*b\rightarrow 0$ as $n\rightarrow \infty$ because $r_n^*c\rightarrow 0$ for all $c\in B$. Also, $x_n^*x_n=b^*b$ for all $n\in \bN$. Finally, as $\beta(X)$ is non-degenerate, any operator $T\in \cK(B,\beta(X))$ satisfies $r_n^*\rhd T\lhd r_n\rightarrow 0$ (this is immediate for finite rank operators and their span is dense in $\cK(B,\beta(X))$ in the norm topology). Therefore
    \begin{align*}
        &b^*\rhd \psi_X(\xi)\lhd b-x_n^*\rhd \phi_X(\xi)\lhd x_n\\
        ={}&b^*\rhd (r_n^*\rhd (\psi_X^{\infty}(\xi)-S^*\rhd \phi_X(\xi)\lhd S)\lhd r_n)\lhd b\longrightarrow 0.
    \end{align*}
    To perform this estimation when quantifying over arbitrary finite sets $K\subset \Irr(\cC)$ and compact sets $\cF_X\subset \alpha(X)$ for $X\in K$ we note that it is sufficient to choose $n$ large enough for the approximations above for finitely many elements $\xi\in\alpha(X)$ for $X\in K$ and use that $\psi_X$ and $\phi_X$ are continuous.
    
    \ref{item:4lem3.11} $\Rightarrow$ \ref{item:3lem3.11}. By hypothesis, there exist $u\in \cM(B)$ and isometries $s_1,s_2\in \cM(B)$ with $1_{\cM(B)}=s_1s_1^*+s_2s_2^*$ such that \[u\rhd \phi_X(\xi)\lhd u^*-s_1\rhd \psi_X^{\infty}(\xi)\lhd s_1^{*}-s_2\rhd \phi_X(\xi)\lhd s_2^*\in \cK(B,\beta(X))\] for all $X\in\cC$, $\xi\in\alpha(X)$. Then $S=u^*s_1$ is the required isometry.
    
    \ref{item:3lem3.11} $\Rightarrow$ \ref{item:4lem3.11}. The double infinite repeat $(\psi^{\infty})^{\infty}$ is unitarily equivalent to $\psi^{\infty}$. Therefore choose an isometry $S\in \cM(B)$ so that
    \[S^*\rhd\phi_X(\xi)\lhd S-(\psi_X^\infty)^\infty(\xi)\in \cK(B,\beta(X))\] 
    for all $\xi\in \alpha(X)$ and $X\in \cC$. Conjugating by $r_n^*$ and setting $s_n=Sr_n$ one has that
    \begin{equation}\label{eqn:convergence}
    s_n^*\rhd\phi_X(\xi)\lhd s_n-\psi_X^\infty(\xi)\longrightarrow 0
    \end{equation}
     as $n\rightarrow \infty$ for all $\xi\in \alpha(X)$ and $X\in \cC$ (see also the proof of \ref{item:3lem3.11} $\Rightarrow$ \ref{item:1lem3.11}). Note that the sequence $s_n$ satisfies the conditions of Lemma \ref{lem:isometryabsorption}. Indeed, $s_{n+1}^*s_n=0$ for any $n\geq 1$. Moreover, combining \eqref{eqn:convergence} and the properties of $\cC$-cocycle representations one has that 
     \begin{align*}
         &(s_n\rhd \psi_X^{\infty}(\xi)-\phi_X(\xi)\lhd s_n)^*\circ (s_n\rhd \psi_X^{\infty}(\xi)-\phi_X(\xi)\lhd s_n)\\
         ={}&\psi_X^\infty(\xi)^*\circ (\psi_X^{\infty}(\xi)-s_n^*\rhd \phi_X(\xi)\lhd s_n)\\
         -{}&s_n^*\rhd \phi_X(\xi)^*\lhd s_n\circ \psi_X^{\infty}(\xi)+s_n^*\rhd \phi_X(\xi)^*\circ \phi_X(\xi)\lhd s_n\\
         ={}&\psi_X^\infty(\xi)^*\circ (\psi_X^{\infty}(\xi)-s_n^*\rhd \phi_X(\xi)\lhd s_n)\\
         +{}&(s_n^*\phi_{1_{\cC}}(\langle \xi,\xi\rangle)s_n-\psi^\infty_{1_\cC}(\langle\xi,\xi\rangle))\\
         +{}&(\psi_X^{\infty}(\xi)-s_n^*\rhd \phi_X(\xi)\lhd s_n)^*\circ\psi_X^{\infty}(\xi)\\
         &\hspace{-1.5em}\longrightarrow 0.
     \end{align*}
     Therefore, as each term on the right hand side of the equation above is in $\cK(B,\beta(X))$ for every $n\in \bN$, so is the product $(s_n\rhd \psi_X^{\infty}(\xi)-\phi_X(\xi)\lhd s_n)^*\circ (s_n\rhd \psi_X^{\infty}(\xi)-\phi_X(\xi)\lhd s_n)$. Now $s_n\rhd \psi_X^{\infty}(\xi)-\phi_X(\xi)\lhd s_n$ is an element of the C$^*$-algebra $\cK(B\oplus\beta(X))$ and so by polar decomposition (see for example \cite[Proposition II.3.2.1]{BlackadarEncycBook}) 
     one has that $s_n\rhd \psi_X^{\infty}(\xi)-\phi_X(\xi)\lhd s_n\in \cK(B,\beta(X))$ for all $n\in \bN$. Moreover, the computation above shows that $\|s_n\rhd \psi_X^{\infty}(\xi)-\phi_X(\xi)\lhd s_n\|\longrightarrow 0$. The result now follows from Lemma \ref{lem:isometryabsorption} as the elements
     \[\widetilde{r}_i=\sum_{n=1}^\infty r_{2n-i}r_n^*\in \cM(B),\ i=0,1\] 
     give the required copy of $\cO_2$.
\end{proof}

\begin{cor}\label{cor:weakcontgivesasym}
    Let $\phi,\psi\colon(A,\alpha,\fu)\rightarrow (B,\beta,\fv)$ be $\cC$-cocycle representations with $A$ separable and $B$ $\sigma$-unital and stable.\ Then $\psi\preccurlyeq_{\cC} \phi$ if and only if $\phi^{\infty}\sim_{\cC}\psi^{\infty}\oplus \phi^{\infty}$.
\end{cor}

\begin{proof}
    Let $r_n\in \cM(B)$ be isometries such that $\sum_{n=1}^\infty r_nr_n^*=1$ in the strict topology; all infinite repeats in this proof will be constructed with respect to this sequence. We start by showing the if direction. Note that $\phi^{\infty}\sim_{\cC}\psi^{\infty}\oplus \phi^{\infty}$ implies that $\phi^{\infty}\au\psi^{\infty}\oplus \phi^{\infty}$. Then, $\psi^{\infty}\oplus \phi^{\infty}\preccurlyeq_{\cC} \phi^{\infty}$ by Lemma \ref{lem:densecontainment}. Therefore, one has that
    \[\psi\preccurlyeq_{\cC} \psi^{\infty}\preccurlyeq_{\cC} \psi^{\infty}\oplus \phi^{\infty}\preccurlyeq_{\cC} \phi^{\infty}\preccurlyeq_{\cC} \phi,\] 
    where the first and last $\cC$-weak containments follow from Lemma \ref{lem:containmentinsum}. As $\preccurlyeq_{\cC}$ is transitive it follows that $\psi\preccurlyeq_{\cC} \phi$.
    
    \par We now turn to the only if direction. By Lemma \ref{lem:biglemsec3} it suffices to show that $\psi$ is $\cC$-contained in $\phi^{\infty}$ at infinity. Let $ K\subset \Irr(\cC)$ be finite, $\cF_X\subset \alpha(X)$ for $X\in K$ be compact, $\varepsilon>0$, and $b\in B_1$. As $\psi\preccurlyeq_{\cC} \phi$, there exist $d_i$ for $1\leq i \leq k$ so that for all $X\in K$
    \[\sup_{\xi\in \cF_X}\|b^*\rhd \psi_X(\xi)\lhd b-\sum_{i=1}^k d_i^*\rhd \phi_X(\xi)\lhd d_i\|\leq \varepsilon\] and
    \[\|b^*b-\sum_{i=1}^k d_i^*d_i\|\leq \varepsilon.\] 
    We now let $N$ be so that $\|b-\sum_{i=1}^N r_ir_i^*b\|\leq \varepsilon$ and set $x=\sum_{i=1}^kr_{N+i}d_i$. It follows that for all $\xi\in\cF_X$ and $X\in K$
    \[x^*\rhd \phi_X^{\infty}(\xi)\lhd x=\sum_{i=1}^kd_i^*\rhd \phi_X(\xi)\lhd d_i=_{\varepsilon} b^*\rhd \psi_X(\xi)\lhd b.\] 
    A similar estimation shows that $\|x^*x-b^*b\|\leq \varepsilon$, which gives that $\|x\|\leq 1+\varepsilon$. Moreover, $\|x^*b\|= \|x^*(b-\sum_{i=1}^N r_ir_i^*b)\|\leq \varepsilon+\varepsilon^2$. This shows that $\psi$ is $\cC$-contained in $\phi^{\infty}$ at infinity, which finishes the proof.
\end{proof}

We may now show the existence of absorbing $\cC$-cocycle representations between separable $\cC$-C$^*$-algebras.

\begin{theorem}\label{thm:absorbingexists}
    Let $\cC$ be a unitary tensor category with countably many isomorphism classes of simple objects and $(A,\alpha,\fu)$, $(B,\beta,\fv)$ be separable $\cC$-C$^*$-algebras with $B$ stable. Then there exists an absorbing $\cC$-cocycle representation from $(A,\alpha,\fu)$ to $(B,\beta,\fv)$.
\end{theorem}

\begin{proof}
When $A$ and $B$ are separable and $\cC$ has countably many isomorphism classes of simple objects, the topology on $\Hom^\cC((\alpha,\fu),(\beta,\fv))$ is second countable by \cref{rmk:topologyHomcC}. 
Therefore, there exists a sequence of $\cC$-cocycle representations $\theta^{(n)}\colon(A,\alpha,\fu)\rightarrow (B,\beta,\fv)$ which is dense in the set of $\cC$-cocycle representations from $(A,\alpha,\fu)$ to $(B,\beta,\fv)$. We denote by $\theta$ the infinite direct sum $\bigoplus_{n=1}^\infty \theta^{(n)}$. Combining Lemma \ref{lem:densecontainment} and Lemma \ref{lem:containmentinsum} we have that any $\cC$-cocycle representation $\phi\colon(A,\alpha,\fu)\rightarrow (B,\beta,\fv)$ satisfies $\phi\preccurlyeq_{\cC} \theta$. It follows from Corollary \ref{cor:weakcontgivesasym} that 
\[\theta^{\infty}\sim_{\cC} \phi^{\infty}\oplus \theta^{\infty}=\phi\oplus (\phi^{\infty}\oplus \theta^{\infty})\sim_{\cC} \phi\oplus \theta^{\infty}.\qedhere\]
\end{proof}

\begin{rmk}\label{rmk:extabsorbing}
    We note that our arguments, as in \cite[Theorem 3.16]{GASZ22}, also show the existence of absorbing $\cC$-cocycle representations in $\sigma$-additive sets of $\cC$-cocycle representations between separable $\cC$-C$^*$-algebras. This may be relevant for future applications.  
\end{rmk}

\begin{cor}\label{cor:KKthroughabsorbing}
 Let $A$ and $B$ be separable $\cC$-C$^*$-algebras and let $\theta\colon(A,\alpha,\fu)\rightarrow (B^\rs,\beta^\rs,\fv^\rs)$ be an absorbing $\cC$-cocycle representation. Then every element $z\in \KK^{\cC}((\alpha,\fu),(\beta,\fv))$ may be expressed as a $\cC$-Cuntz pair $[\phi,\theta]$ for some absorbing $\cC$-cocycle representation $\phi\colon(A,\alpha,\fu)\rightarrow (B^\rs,\beta^\rs,\fv^\rs)$.
\end{cor}

\begin{proof}
 Let $z=[\psi,\kappa]$ for some $\cC$-Cuntz pair from $(\alpha,\fu)$ to $(\beta,\fv)$. As $\theta$ is absorbing, one has that $\kappa\oplus\theta\sim_{\cC}\theta$ so there exists a path of unitaries $w\colon[1,\infty)\rightarrow \cU(\cM(B^\rs))$ such that $[\Ad(w_1)\circ (\kappa\oplus\theta),\theta]$ forms a degenerate $\cC$-Cuntz pair.
 Therefore, $[\Ad(w_1)\circ (\kappa\oplus\theta),\theta]=0$ in $\KK^{\cC}((\alpha,\fu),(\beta,\fv))$ by Lemma~\ref{lemma: CuntzSumGroup}.
 Moreover, also using that any unitary in $\cU(\cM(B^\rs))$ is path connected to the identity, one has that
 \begin{align*}
     z&=[\psi,\kappa]\\
     &=[\psi\oplus \theta,\kappa\oplus \theta]\\
     &=[\Ad(w_1)\circ (\psi\oplus \theta),\Ad(w_1)\circ (\kappa\oplus \theta)]+[\Ad(w_1)\circ (\kappa\oplus \theta),\theta]\\
     &=[\Ad(w_1)\circ (\psi\oplus \theta),\theta].
 \end{align*}
 Where the last equality follows as $[\psi,\phi]+[\phi,\kappa]=[\psi,\kappa]$ for any $\cC$-Cuntz pairs (this follows from the argument of \cite[(ii) Lemma 1.10]{GASZ22}). Since $\theta$ is absorbing, so is $\Ad(w_1)\circ (\psi\oplus \theta)$, so the conclusion follows.
\end{proof}

\section{Obtaining asymptotic unitary equivalence}\label{sect: AsymptEquiv}

Similarly to the results in \cite[Section 4]{GASZ22}, the goal of this section is to show that if two $\cC$-cocycle representations are $\cC$-operator homotopic, then the unitary path witnessing the homotopy can be changed to take values in the minimal unitisation of $B$ instead of its multiplier algebra. Before proving this result, let us state an appropriate notion of asymptotic unitary equivalence.

\begin{defn}\label{defn:properasympunitaryeq}
    Let $\phi,\psi\colon(A,\alpha,\fu)\rightarrow (B,\beta,\fv)$ be $\cC$-cocycle representations. We say that $\phi$ and $\psi$ are \emph{$\cC$-properly asymptotically unitarily equivalent} if there exists a norm-continuous path $u\colon[0,\infty)\rightarrow \cU(1+B)$ such that for all $X\in \cC$ and $\xi\in \alpha(X)$
  \[\psi_X(\xi)=\lim_{t\rightarrow\infty} u_t\rhd \phi_X(\xi)\lhd u_t^*.\] If we can further arrange that $u_0 = 1$, then we say that $\phi$ and $\psi$ are \emph{$\cC$-strongly asymptotically unitarily equivalent}. Moreover, if $B$ is stable, we say that $\phi$ and $\psi$ are \emph{stably $\cC$-properly (resp. strongly) asymptotically unitarily equivalent} if there exists a $\cC$-cocycle representation $\theta\colon (A, \alpha,\fu) \to (B, \beta,\fv)$ such that $\phi \oplus \theta$ is $\cC$-properly (resp. strongly) asymptotically unitarily equivalent to $\psi\oplus\theta$.
\end{defn}

The following two lemmas are the main technical tools in showing that $\cC$-operator homotopy implies $\cC$-strong asymptotic unitary equivalence. The proofs essentially follow the strategy in \cite[Lemma 4.2, Lemma 4.3]{GASZ22}, with a key technical difference being the use of Lemma \ref{lem:approxunit}.

\begin{lemma}\label{lemma: ChangeUnitPath1}
Let $A$ be a separable $\cC$-C$^*$-algebra, $B$ be a $\sigma$-unital $\cC$-C$^*$-algebra, and $D_X\subseteq \cL(B,\beta(X))$ be a separable closed linear subspace for each $X\in\Irr(\cC)$. Let $\phi\colon(A,\alpha,\fu)\rightarrow (B,\beta,\fv)$ be a $\cC$-cocycle representation. Let $U\colon[0,\infty)\to \cU(\fD_\phi)$ be a norm-continuous path such that $U_0=1$, $U|_{[1,\infty)}$ is constant, $\|U_t-1\|<2$ for any $t\in[0,1]$, and $[U_t,D_X]\subseteq\cK(B,\beta(X))$ for any $X\in\cC$. Then for all sequences $\varepsilon_n>0$ over $n\geq 0$, all finite sets $K_n\subseteq \Irr(\cC)$, and all compact sets $\cF_n(X)\subseteq D_X$ for any $X\in K_n$, there exists a norm-continuous path $v \colon [0, \infty) \to \cU(1 + B)$ such that \[\|[v_t^*U_t,\eta]\|\leq \varepsilon_n\] for any $n\geq 0$, $t\in[n,n+1]$, $X\in K_n$, and any $\eta\in\cF_n(X)$. Moreover, we can choose the path $v$ such that $v_t=1$ whenever $U_t=1$. 
%Let $A$ be a separable $\cC$-C$^*$-algebra, $B$ be a $\sigma$-unital $\cC$-C$^*$-algebra, and $D_X\subseteq \cL(B,\beta(X))$ be a separable closed linear subspace for each $X\in\cC$. Let $\phi\colon(A,\alpha,\fu)\rightarrow (B,\beta,\fv)$ be a cocycle representation. Let $U\colon[0,\infty)\to \cU(\fD_\phi)$ be a norm-continuous path such that $U_0=1$, $U|_{[1,\infty)}$ is constant, $\|U_t-1\|<2$ for any $t\in[0,1]$, and $[U_t,D_X]\subseteq\cK(B,\beta(X))$ for any $X\in\cC$. Then for all sequences $\varepsilon_n>0$, all finite sets $K_n\subseteq \Irr(\cC)$, all finite sets $\cG_n(X)\subseteq \alpha(X)$ for any $X\in K_n$, and all compact sets $\cF_n(X)\subseteq D_X$ for any $X\in K_n$, there exists a norm-continuous path $v \colon [0, \infty) \to \cU(1 + B)$ such that \[\|[v_t^*U_t,\phi_X(\xi)]\|\leq\varepsilon_n\] and \[\|[v_t^*U_t,\eta]\|\leq \varepsilon_n\] for any $n\geq 0$, $t\in[n,n+1]$, $X\in K_n$, $\xi\in\cG_n(X)$, and any $\eta\in\cF_n(X)$. Moreover, we can choose the path $v$ such that $v_t=1$ whenever $U_t=1$. 
\end{lemma}

%\nn{I made the assertion a bit shorter by suppressing $\cG_X$. (This will not lose the generality if we let $\phi_X(\cG_X)\subset \cF_X$.) KK}

\begin{proof}
We start by enlarging $D_X$ such that it is defined for all irreducible objects $X\in \cC$ and such that $\overline{D_{X}}\subset D_{\overline{X}}$. For any irreducible object $Y\in\cC$ and an isomorphism $\nu\colon X\to Y$ with $X\in\Irr(\cC)$, we put the linear subspace $D_Y\coloneqq\beta(\nu)\circ D_X \subset \cL(B,\beta(Y))$, which is independent on the choice of $\nu$ whose ambiguity is only up to scalar multiplication. 
%\marginpar{What is this trying to say? The $D_X$ sets are given, so we implicitly assume they are 'the same' for objects in the same isom class. RN \\ You are right, but the point (if I'm not too pedantic) is to make 'the same' rigorous. For example, if we accidentally chose $\overline{\overline{X}}=X$ (not just isomorphic), we would need $D_X=\mu_XD_X$. This subtlety can be ignored as $D_Y\coloneqq\nu \circ D_X$ is independent of $\nu$. KK}
%We may assume $\phi_X(\alpha(X))\subset D_X$ by taking $\phi_X(\alpha(X))\cup D_X$ if necessary. 
Moreover, we replace $D_X$ with $\overline{\spa}(D_{X} \cup \beta(\mu_X^{-1})\circ \overline{D_{\overline{X}}})$ which satisfies the required conditions. Indeed, by putting 
$\nu\coloneqq (R_{\overline{\overline{X}}}^*\otimes \id_{\overline{X}}) (\id_{\overline{\overline{\overline{X}}}}\otimes\mu_X\otimes \id_{\overline{X}}) (\id_{\overline{\overline{\overline{X}}}}\otimes \overline{R}_{X})\neq 0$, we have
\begin{align*}
    \overline{ D_{X} \cup \beta(\mu_X^{-1})\circ \overline{D_{\overline{X}}} }
    ={}&
    \overline{ \beta(\mu_X^{-1})\circ ( D_{\overline{\overline{X}}} \cup \overline{D_{\overline{X}}} ) }\\
    \stackrel{\eqref{eqn:barinnerprod2}}{=}&
    \beta(\nu) \circ \overline{ D_{\overline{\overline{X}}} \cup \overline{D_{\overline{X}}} }
    \\
    ={}&
    \beta(\mu_{\overline{X}}^{-1}) \circ \overline{ D_{\overline{\overline{X}}} \cup \overline{D_{\overline{X}}} }\\
    ={}&
     D_{\overline{X}}\cup \beta(\mu_{\overline{X}}^{-1})\circ \overline{D_{\overline{\overline{X}}}}. 
\end{align*}
Note that the second equality holds by expanding the definition of the bar notation and using the conjugate equations, whereas in the third equality we use that $\beta(\nu)$ and $\beta(\mu_{\overline{X}}^{-1})$ only differ by a scalar due to the irreducibility of $\overline{X}$.
Then we have the C$^*$-subalgebra $\fD_D\subset \cM(B)$ given by \cref{lem:fD*alg}. 
%which is contained in $\fD_\phi$. Using the notation in the statement of the lemma, we can assume that the sets $\cG_k(X)$ and $\cF_k(X)$ consists of contractions and satisfy $\phi_X(\cG_k(X))\subset \cF_k(X)$. 
Using the notation in the statement of the lemma, we can assume that the sets $\cF_k(X)$ consist of contractions. 
Moreover, $\|U_t-1\|<2$ for any $t\in[0,1]$, so functional calculus yields a norm-continuous path of self-adjoint elements $a\colon[0,\infty)\to\fD_{D}$ such that \[U_t=\exp(ia_t),\ \|a_t\|<\pi,\ \text{and}\ [a_t, D_X]\subseteq\cK(B,\beta(X)),\] for any $t\in[0,\infty)$ and $X\in\cC$. In particular, $a_0=0$ and $a|_{[1,\infty)}$ is constant.

Then, by Lemma \ref{lem:approxunit}, there exists an increasing approximate unit $e_n\in B$ such that
\begin{equation}\label{eq: approxunit1}
e_n\rhd T\to T \  \text{and}\ T\lhd e_n\to T\ \text{for any} \ T\in \cK(B,\beta(X)),\  X\in \cC,   
\end{equation}
\begin{comment}\label{eq: approxunit2}
\lim\limits_{n\to\infty}\|[e_n,\phi_X(\xi)]\|=0\ \text{for any} \ X\in\cC, \ \xi\in\alpha(X).
\end{comment} 
\begin{equation}\label{eq: approxunit3}
\lim\limits_{n\to\infty}\|[\eta,e_n]\|=0 ,
\end{equation}
and
\begin{equation}\label{eq: approxunit4}
\lim\limits_{n\to\infty}\max\limits_{t\in[0,1]}\|[a_t,e_n]\|=0 
\end{equation}
for any $X\in \cC$ and $\eta\in \cF_X$. %Furthermore, picking elements in the convex hull if necessary, we can ensure quasicentrality of the approximate unit $e_n$. In particular, we assume that

Extending linearly, we find an increasing norm-continuous path of positive contractions $e\colon[0,\infty)\to B$ satisfying the same asymptotic conditions as the sequence $e_n$. Therefore, we get that 
\begin{align*}
&\lim\limits_{s\to\infty}\max\limits_{t\in[0,1]}\|[a_t-e_sa_te_s,\eta]\|\\
&\stackrel{\eqref{eq: approxunit3},\eqref{eq: approxunit4}}{=} \lim\limits_{s\to\infty}\max\limits_{t\in[0,1]}\|(1-e_s^2)[a_t,\eta]\|\stackrel{\eqref{eq: approxunit1}}{=}0
\end{align*} for any $X\in \cC$, any $\eta\in D_X$, where the last equality uses that $a_t\in\fD_D$. 
%A similar calculation using \eqref{eq: approxunit1}, \eqref{eq: approxunit3}, \eqref{eq: approxunit4}, and that $[a_t, D_X]\subseteq\cK(B,\beta(X))$ shows that \[\lim\limits_{s\to\infty}\|[\eta,a_t-e_sa_te_s]\|=0\] for any $X\in\cC$, $\eta\in D_X$, and uniformly over $t\in[0,1]$.

Choose a sequence of positive numbers $\delta_n$ as in \cite[Lemma 4.2]{GASZ22}. Reparametrising our paths if necessary, we can assume that the following estimates hold:
\begin{equation}\label{eq: est1}
\max\limits_{t\in[n,n+1]}\|[a_t,e_ta_te_t]\|\leq \delta_n;
\end{equation}

\begin{comment}\label{eq: est2}
\max\limits_{t\in[n,n+1]}\max\limits_{X\in K_n,\xi\in\cG_n(X)}\|[a_t-e_ta_te_t,\phi_X(\xi)]\|\leq\delta_n;   
\end{comment}

\begin{equation}\label{eq: est3}
\max\limits_{t\in[n,n+1]}\max\limits_{X\in K_n,\eta\in\cF_n(X)}\|[a_t-e_ta_te_t,\eta]\|\leq\delta_n.  
\end{equation}

To finish the proof, we define $v_t=\exp(ie_ta_te_t)\in\cU(1+B)$ for any $t\in[0,\infty)$ and show that it satisfies the required conditions. Fix $n\geq 0$, $t\in[n,n+1]$, $X\in K_n$, and $\eta\in\cF_n(X)$. 
%$\supset\phi_X(\cG_n(X))$. 
Then \[\|[v_t^*U_t,\eta]\|\stackrel{\eqref{eq: est1}}{=}_{\!\!2\varepsilon_n/8}\|[\exp(i(a_t-e_ta_te_t)),\eta]\|\stackrel{\eqref{eq: est3}}{=}_{\!\!\varepsilon_n/8}0.\] The last statement follows as $v_t=1$ whenever $a_t=0$.
\end{proof}

\begin{lemma}\label{lemma: ChangeUnitPath2}
Let $A$ be a separable $\cC$-C$^*$-algebra, $B$ be a $\sigma$-unital $\cC$-C$^*$-algebra, and $\phi\colon(A,\alpha,\fu)\rightarrow (B,\beta,\fv)$ be a $\cC$-cocycle representation. Let $U\colon[0,\infty)\to \cU(\fD_\phi)$ be a norm-continuous path such that $U_0=1$. Then there exists a norm-continuous path $v\colon[0,\infty)\to\cU(1+B)$ with $v_0=1$ such that \[\lim\limits_{t\to\infty}\|[v_t^*U_t,\phi_X(\xi)]\|=0\] for any $X\in\cC$ and any $\xi\in\alpha(X)$.
\end{lemma}

\begin{proof}
Since $U$ is uniformly continuous on bounded intervals, there exists an increasing sequence $0=t_0<t_1<t_2<\ldots$ with $t_n\to\infty$ such that 
\begin{equation}\label{eq: CutUnitPath}
\max\limits_{t_n\leq s\leq t_{n+1}}\|U_s-U_{t_n}\|<2,\ n\geq 0.
\end{equation}
Reparametrising the path if necessary, we can assume that $t_n=n$ for any $n\in\mathbb{N}$. We then define a norm-continuous path $U^{(n)}\colon[0,\infty)\to\cU(\fD_\phi)$ by \[  U_t^{(n)}= 
\begin{cases} 
      U_{n+1}U_n^* & t>n+1 \\
      U_tU_n^* & n\leq t\leq n+1 \\
      1 & t<n.
   \end{cases}
\]In particular, using that $U_0=1$, we have that \[U_t=U_t^{(n)}U_t^{(n-1)}\ldots U_t^{(0)}\] for any $n\in\mathbb{N}$ and any $t\leq n+1$.

Let $K_n$ be an increasing sequence of finite sets whose union is $\Irr(\cC)$. For any $X\in \Irr(\cC)$, let $\cG_k(X)\subseteq\alpha(X)$ for $k\geq\min\{ n\,|\, X\in K_n \}$ be an increasing sequence of finite sets of contractions such that the union $\bigcup_{k}\cG_k(X)$ is dense in the unit ball of $\alpha(X)$. Then, for $n\leq k$ and $X\in K_n$, we consider the compact sets \[\cG_n^k(X)=\{U_t^{(n-1)}\ldots U_t^{(0)}\rhd\phi_X(\xi)\lhd U_t^{(0)*}\ldots U_t^{(n-1)*}\colon \xi\in \cG_k(X),\ t\leq k+1\}.\] 

We now apply Lemma \ref{lemma: ChangeUnitPath1} for every $n\in\mathbb{N}$ to the path $U_{t+n}^{(n)}$ for $t\in[0,\infty)$. Note that $U^{(n)}_n=1$, while \eqref{eq: CutUnitPath} shows that $\|U_t^{(n)}-1\|<2$ for any $t\in[n,n+1]$. Furthermore, as $U_t^{(m)}\in\fD_\phi$ for any $m\geq 0$, we have that \[U_t^{(n)}\rhd\eta\lhd U_t^{(n)*}-\eta\in\cK(B,\beta(X))\] for any $k\geq n$, $X\in K_k$, $t\geq 0$, and $\eta\in\cG_n^k$. Thus, Lemma \ref{lemma: ChangeUnitPath1} yields a norm-continuous path $v^{(n)} \colon [0, \infty) \to \cU(1 + B)$ such that 
%\begin{equation}\label{eq: UnitPath1}
%\|[v_t^{(n)*}U_t^{(n)},\phi_X(\xi)]\|\leq 2^{-(n+k)}    
%\end{equation} and
\begin{equation}\label{eq: UnitPath2}
\|[v_t^{(n)*}U_t^{(n)},\eta]\|\leq 2^{-(n+k)}    
\end{equation}
for every $k\geq n$, any $t\in[k,k+1], X\in K_k$, $\xi\in \cG_k(X)$, and $\eta\in\cG_n^k(X)$. Since $U_t^{(n)}=1$ for any $t\leq n$, Lemma \ref{lemma: ChangeUnitPath1} also gives that $v^{(n)}|_{[0,n]}=1$.

Using the family of paths $v^{(n)}$, we can now define a norm-continuous path $v\colon[0,\infty)\to\cU(1+B)$ by \[v_t=v_t^{(n)}v_t^{(n-1)}\ldots v_t^{(0)}, \quad n\leq t\leq n+1.\] We prove that the path $v$ satisfies the required properties. First note that $v_0=v_0^{(0)}=1$. Then, for any $k\geq 0$, $t\in[k,k+1]$, $X\in K_k$, and any $\xi\in\cG_k(X)$, we have that 
\begin{align*}
&\|[v_t^*U_t,\phi_X(\xi)]\| \\
={}& \|v_t^{(k)*}U_t^{(k)}\ldots U_t^{(0)}\rhd \phi_X(\xi)\lhd U_t^{(0)*}\ldots U_t^{(k)*}v_t^{(k)}\\
-{}& v_t^{(k-1)}\ldots v_t^{(0)}\rhd \phi_X(\xi)\lhd v_t^{(0)*}\ldots v_t^{(k-1)*}\\
+{}& U_t^{(k-1)}\ldots U_t^{(0)}\rhd\phi_X(\xi)\lhd U_t^{(0)*}\ldots U_t^{(k-1)*}\\
-{}& U_t^{(k-1)}\ldots U_t^{(0)}\rhd\phi_X(\xi)\lhd U_t^{(0)*}\ldots U_t^{(k-1)*}\|\\
\leq{}& 
\|[v_t^{(k)*}U_t^{(k)},U_t^{(k-1)}\ldots U_t^{(0)}\rhd \phi_X(\xi)\lhd U_t^{(0)*}\ldots U_t^{(k-1)*}]\|\\
+&\|[v_t^{(0)*}\ldots v_t^{(k-1)*}U_t^{(k-1)}\ldots U_t^{(0)},\phi_X(\xi)]\|\\
\stackrel{\eqref{eq: UnitPath2}}{\leq}& 
2^{-2k}+\|[v_t^{(0)*}\ldots v_t^{(k-1)*}U_t^{(k-1)}\ldots U_t^{(0)},\phi_X(\xi)]\|.
\end{align*} 
Repeating this estimation $k-1$ more times, we obtain that \[\|[v_t^*U_t,\phi_X(\xi)]\|\leq \sum\limits_{i=k}^{2k}2^{-i}\leq 2^{1-k},\] so the conclusion follows.
\end{proof}

We can now obtain $\cC$-strong asymptotic unitary equivalence from $\cC$-operator homotopy.

\begin{theorem}\label{thm: OpHomToAsympUnitEq}
Let $(A,\alpha,\fu)$ be a separable $\cC$-C$^*$-algebra, $(B,\beta,\fv)$ be a $\sigma$-unital $\cC$-C$^*$-algebra, and $\phi,\psi\colon(A,\alpha,\fu)\rightarrow (B,\beta,\fv)$ be $\cC$-cocycle representations. If $\phi$ and $\psi$ are $\cC$-operator homotopic then they are $\cC$-strongly asymptotically unitarily equivalent.
\end{theorem}

\begin{proof}
By definition, there exists a norm-continuous path $U\colon[0,\infty)\to \cU(\fD_\phi)$ such that $U_0=1$, $U|_{[1,\infty)}$ is constant, and \[U_1\rhd\phi_X(\xi)\lhd U_1^*=\psi_X(\xi)\] for any $X\in\cC$ and $\xi\in\alpha(X)$. Then, we can apply Lemma \ref{lemma: ChangeUnitPath2} to find a norm-continuous path $v\colon[0,\infty)\to\cU(1+B)$ with $v_0=1$ such that \[\lim\limits_{t\to\infty}\|[v_t^*U_t,\phi_X(\xi)]\|=0\] for any $X\in\cC$ and any $\xi\in\alpha(X)$. Thus, for any $X\in\cC$ and $\xi\in\alpha(X)$, we have that 
\begin{align*}
\psi_X(\xi)&=\lim\limits_{t\to\infty}U_t\rhd\phi_X(\xi)\lhd U_t^*\\&=\lim\limits_{t\to\infty}v_t(v_t^*U_t)\rhd\phi_X(\xi)\lhd (v_t^*U_t)^*v_t^*\\&=\lim\limits_{t\to\infty}v_t\rhd\phi_X(\xi)\lhd v_t^*,
\end{align*} which finishes the proof.
\end{proof}

\section{The \texorpdfstring{$\cC$}{C}-equivariant stable uniqueness theorem}
The lemma below follows exactly as in \cite[Proposition 1.13]{GASZ22}.

\begin{lemma}\label{lem:unitisation}
    Let $\phi,\psi\colon (A,\alpha,\fu)\rightarrow (B^\rs,\beta^\rs,\fv^\rs)$ be $\cC$-cocycle representations such that $(\phi,\psi)$ forms a $\cC$-Cuntz pair. If $u\in \cU(1+B^\rs)$, then the $\cC$-Cuntz pair $(\phi,\psi)$ is homotopic to $(\Ad(u)\phi,\psi)$.
\end{lemma}
We can now state and prove the $\cC$-equivariant version of the stable uniqueness theorem (cf. \cite[Theorem 3.8]{DADEIL01} and \cite[Theorem 5.4]{GASZ22}).

\begin{theorem}\label{thm: StableUniq}
Let $\cC$ be a unitary tensor category with countably many isomorphism classes of simple objects, $A$ be a separable $\cC$-C$^*$-algebra, and $B$ be a $\sigma$-unital $\cC$-C$^*$-algebra. Let \[\phi,\psi\colon(A,\alpha,\fu)\to (B^\rs,\beta^\rs,\fv^\rs)\] be two $\cC$-cocycle representations that form a $\cC$-Cuntz pair. Then the following are equivalent:
\begin{enumerate}[label=\textit{(\roman*)}]
\item $[\phi,\psi]=0$ in $\KK^{\cC}((\alpha,\fu),(\beta,\fv))$;\label{SU1} 

\item $\phi$ and $\psi$ are stably $\cC$-operator homotopic;\label{SU2}

\item $\phi$ and $\psi$ are stably $\cC$-strongly asymptotically unitarily equivalent;\label{SU3}

\item $\phi$ and $\psi$ are stably $\cC$-properly asymptotically unitarily equivalent.\label{SU4}
\end{enumerate}
If $B$ is further assumed to be separable, then these statements are also equivalent to the following condition.

\begin{enumerate}[label=\textit{(\roman*)},resume]
\item $\phi\oplus\theta$ is $\cC$-strongly asympotically unitarily equivalent to $\psi\oplus\theta$ for every absorbing $\cC$-cocycle representation $\theta\colon (A,\alpha,\fu) \to (B^\rs,\beta^\rs,\fv^\rs)$.\label{SU5}
\end{enumerate}
\end{theorem}

\begin{proof}
The fact that \ref{SU1} implies \ref{SU2} is the content of Theorem \ref{thm:KKCtostableoh}. Then the implication \ref{SU2}$\implies$\ref{SU3} is a direct consequence of Theorem \ref{thm: OpHomToAsympUnitEq}, while \ref{SU3} implies \ref{SU4} by definition. We now show the implication \ref{SU4}$\implies$\ref{SU1}. Let $\theta\in\Hom^\cC((\alpha,\fu),(\beta,\fv))$ and $w_t\colon[0,\infty)\rightarrow \cU(1+B^\rs)$ be a norm continuous unitary path such that $\Ad(w_t)(\phi_X(\xi)\oplus\theta_X(\xi))\rightarrow \psi_X(\xi)\oplus\theta_X(\xi)$ for all $X\in \cC$ and $\xi\in \alpha(X)$ as $t\rightarrow \infty$. Set
\[\varphi^{(t)}_X(\xi)=\begin{cases}
    \Ad(w_{1/t})(\phi_X(\xi)\oplus\theta_X(\xi)), &t \in (0,1],\\
    \psi_X(\xi)\oplus\theta_X(\xi), &t=0.
\end{cases}\] 
for all $X\in \cC$ and $\xi\in \alpha(X)$. Then, as each $w_t\in \cU(1+B^{s})\subset \fD_{\phi\oplus\theta}$, $t\in [0,1]\mapsto (\varphi^{(t)},\psi\oplus\theta)$ forms a homotopy from a degenerate $\cC$-Cuntz pair to $(\Ad(w_1)(\phi\oplus\theta),\psi\oplus\theta)$. Now, by Lemma \ref{lem:unitisation} the Cuntz-pair $(\Ad(w_1)(\phi\oplus\theta),\psi\oplus\theta)$ is homotopic to the pair $(\phi\oplus\theta,\psi\oplus\theta)$ and so combining these two homotopies $[\phi,\psi]=0$.

\par We now assume that $B$ is separable and show that \ref{SU3} and \ref{SU5} are equivalent. By Theorem \ref{thm:absorbingexists} there exists an absorbing $\cC$-cocycle representation $\theta\colon (A,\alpha,\fu) \to (B^\rs,\beta^\rs,\fv^\rs)$, so \ref{SU5} implies \ref{SU3} by definition. To finish the proof, let $\kappa\colon(A,\alpha,\fu)\to(B^\rs,\beta^\rs,\fv^\rs)$ be a $\cC$-cocycle representation such that $\phi\oplus\kappa$ is strongly asymptotically unitarily equivalent to $\psi\oplus\kappa$. Therefore, $\phi\oplus\kappa\oplus\theta$ is strongly asymptotically unitarily equivalent to $\psi\oplus\kappa\oplus\theta$. Since $\theta$ is absorbing, we have that $\kappa\oplus\theta\sim_{\cC}\theta$, so the conclusion follows using a similar calculation to \cite[Lemma 5.3]{GASZ22}. 
\end{proof}

\bibliographystyle{abbrv}
\bibliography{KK}

@article{ARKIKU23,
  title={Tensor category equivariant {KK}-theory},
  author={Arano, Yuki and Kitamura, Kan and Kubota, Yosuke},
  journal={Adv. Math.},
  volume={453},
  pages={109848},
  year={2024}
}

@article {GASZ22,
    AUTHOR = {Gabe, James and Szab\'o, G\'abor},
     TITLE = {The stable uniqueness theorem for equivariant {K}asparov
              theory},
   JOURNAL = {Amer. J. Math.},
  FJOURNAL = {American Journal of Mathematics},
    VOLUME = {147},
      YEAR = {2025},
    NUMBER = {6},
     PAGES = {1527--1576},
}

@article {cocyclecategszabo,
    AUTHOR = {Szab\'{o}, G\'{a}bor},
     TITLE = {On a categorical framework for classifying {C$^*$}-dynamics up to cocycle conjugacy},
   JOURNAL = {J. Funct. Anal.},
    VOLUME = {280},
      YEAR = {2021},
    NUMBER = {8},
     PAGES = {Paper No. 108927, 66},
}

@book {blackadar,
    AUTHOR = {Blackadar, Bruce},
     TITLE = {{$K$}-theory for operator algebras},
    SERIES = {Mathematical Sciences Research Institute Publications},
    VOLUME = {5},
   EDITION = {Second},
 PUBLISHER = {Cambridge University Press, Cambridge},
      YEAR = {1998},
     PAGES = {xx+300},  
}

@article{intertwining,
  title={An {E}lliott intertwining approach to classifying actions of {C$^*$}-tensor categories},
  author={Gir{\'o}n Pacheco, Sergio and Neagu, Robert},
  journal={arXiv:2310.18125},
  year={2023}
}

@article {DADEIL01,
    AUTHOR = {Dadarlat, Marius and Eilers, S\oren},
     TITLE = {Asymptotic unitary equivalence in {$KK$}-theory},
   JOURNAL = {$K$-Theory},
  FJOURNAL = {$K$-Theory. An Interdisciplinary Journal for the Development,
              Application, and Influence of $K$-Theory in the Mathematical
              Sciences},
    VOLUME = {23},
      YEAR = {2001},
    NUMBER = {4},
     PAGES = {305--322},
      ISSN = {0920-3036,1573-0514},
   MRCLASS = {19K35 (46L80)},
  MRNUMBER = {1860859},
MRREVIEWER = {Emmanuel\ C.\ Germain},
       DOI = {10.1023/A:1011930304577},
       URL = {https://doi.org/10.1023/A:1011930304577},
}

@article {DADEIL02,
    AUTHOR = {Dadarlat, Marius and Eilers, S\oren},
     TITLE = {On the classification of nuclear {C$^*$}-algebras},
   JOURNAL = {Proc. London Math. Soc. (3)},
  FJOURNAL = {Proceedings of the London Mathematical Society. Third Series},
    VOLUME = {85},
      YEAR = {2002},
    NUMBER = {1},
     PAGES = {168--210},
      ISSN = {0024-6115,1460-244X},
   MRCLASS = {19K35 (46L05 46L35 46L80)},
  MRNUMBER = {1901373},
MRREVIEWER = {Emmanuel\ C.\ Germain},
       DOI = {10.1112/S0024611502013679},
       URL = {https://doi.org/10.1112/S0024611502013679},
}

@article {KA88,
    AUTHOR = {Kasparov, G. G.},
     TITLE = {Equivariant {$\KK$}-theory and the {N}ovikov conjecture},
   JOURNAL = {Invent. Math.},
  FJOURNAL = {Inventiones Mathematicae},
    VOLUME = {91},
      YEAR = {1988},
    NUMBER = {1},
     PAGES = {147--201},
      ISSN = {0020-9910,1432-1297},
   MRCLASS = {58G12 (19K33 19K56 46L80 46M20 53C20 57R67)},
  MRNUMBER = {918241},
MRREVIEWER = {Jonathan\ M.\ Rosenberg},
       DOI = {10.1007/BF01404917},
       URL = {https://doi.org/10.1007/BF01404917},
}

@article {AFclass,
    AUTHOR = {Chen, Quan and Hern\'{a}ndez Palomares, Roberto and Jones,
              Corey},
     TITLE = {K-theoretic {C}lassification of {I}nductive {L}imit {A}ctions
              of {F}usion {C}ategories on {AF}-algebras},
   JOURNAL = {Comm. Math. Phys.},
  FJOURNAL = {Communications in Mathematical Physics},
    VOLUME = {405},
      YEAR = {2024},
    NUMBER = {3},
     PAGES = {Paper No. 83},
      ISSN = {0010-3616,1432-0916},
   MRCLASS = {46L80 (18M20 19L47)},
  MRNUMBER = {4717816},
       DOI = {10.1007/s00220-024-04969-w},
       URL = {https://doi.org/10.1007/s00220-024-04969-w},
}

@article {THO05,
    AUTHOR = {Thomsen, Klaus},
     TITLE = {Duality in equivariant {$KK$}-theory},
   JOURNAL = {Pacific J. Math.},
  FJOURNAL = {Pacific Journal of Mathematics},
    VOLUME = {222},
      YEAR = {2005},
    NUMBER = {2},
     PAGES = {365--397},
      ISSN = {0030-8730,1945-5844},
   MRCLASS = {19K35 (46L80)},
  MRNUMBER = {2225077},
MRREVIEWER = {Otgonbayar\ Uuye},
       DOI = {10.2140/pjm.2005.222.365},
       URL = {https://doi.org/10.2140/pjm.2005.222.365},
}

@book {NETU13,
    AUTHOR = {Neshveyev, Sergey and Tuset, Lars},
     TITLE = {Compact quantum groups and their representation categories},
    SERIES = {Cours Sp\'{e}cialis\'{e}s [Specialized Courses]},
    VOLUME = {20},
 PUBLISHER = {Soci\'{e}t\'{e} Math\'{e}matique de France, Paris},
      YEAR = {2013},
     PAGES = {vi+169},
      ISBN = {978-2-85629-777-3},
   MRCLASS = {46L65 (17B37 46L89 81-02 81R05 81R10 81R50)},
  MRNUMBER = {3204665},
MRREVIEWER = {Julien\ Bichon},
}

@article {CHHPJOPE22,
    AUTHOR = {Chen, Quan and Hern\'{a}ndez Palomares, Roberto and Jones,
              Corey and Penneys, David},
     TITLE = {Q-system completion for {C$^*$} 2-categories},
   JOURNAL = {J. Funct. Anal.},
  FJOURNAL = {Journal of Functional Analysis},
    VOLUME = {283},
      YEAR = {2022},
    NUMBER = {3},
     PAGES = {Paper No. 109524, 59},
      ISSN = {0022-1236,1096-0783},
   MRCLASS = {46L37 (18M15 18M20 18M30 18N10 46L80)},
  MRNUMBER = {4419534},
MRREVIEWER = {Yasuyuki\ Kawahigashi},
       DOI = {10.1016/j.jfa.2022.109524},
       URL = {https://doi.org/10.1016/j.jfa.2022.109524},
}

@book {Hilbertmodules,
    AUTHOR = {Lance, E. C.},
     TITLE = {Hilbert {C$^*$}-modules. {A} toolkit for operator algebraists},
    SERIES = {London Mathematical Society Lecture Note Series},
    VOLUME = {210},
 PUBLISHER = {Cambridge University Press, Cambridge},
      YEAR = {1995},
     PAGES = {x+130},
}

@article {KWP04,
    AUTHOR = {Kajiwara, Tsuyoshi and Pinzari, Claudia and Watatani, Yasuo},
     TITLE = {Jones index theory for {H}ilbert {C$^*$}-bimodules and its
              equivalence with conjugation theory},
   JOURNAL = {J. Funct. Anal.},
  FJOURNAL = {Journal of Functional Analysis},
    VOLUME = {215},
      YEAR = {2004},
    NUMBER = {1},
     PAGES = {1--49},
      ISSN = {0022-1236,1096-0783},
   MRCLASS = {46L08 (46L05 46L60)},
  MRNUMBER = {2085108},
MRREVIEWER = {Michael\ Frank},
       DOI = {10.1016/j.jfa.2003.09.008},
       URL = {https://doi.org/10.1016/j.jfa.2003.09.008},
}

@article{MEY00, 
    AUTHOR = {Meyer, Ralf},
     TITLE = {Equivariant {K}asparov theory and generalized homomorphisms},
      YEAR = {2000},
   JOURNAL = {{$K$}-Theory},
    VOLUME = {21},
    NUMBER = {3},
     PAGES = {201--228},
}

@incollection {SEL11,
    AUTHOR = {Selinger, P.},
     TITLE = {A survey of graphical languages for monoidal categories},
 BOOKTITLE = {New structures for physics},
    SERIES = {Lecture Notes in Phys.},
    VOLUME = {813},
     PAGES = {289--355},
 PUBLISHER = {Springer, Heidelberg},
      YEAR = {2011},
      ISBN = {978-3-642-12820-2},
   MRCLASS = {18D10},
  MRNUMBER = {2767048},
MRREVIEWER = {Volodymyr\ V.\ Lyubashenko},
       DOI = {10.1007/978-3-642-12821-9\_4},
       URL = {https://doi.org/10.1007/978-3-642-12821-9_4},
}

@book{HigsonRoe,
    author = {Higson, Nigel and Roe, John},
    title = {Analytic {K}-{H}omology},
    publisher = {Oxford University Press},
    year = {2000},
}

@article{Tho01,
  title={On absorbing extensions},
  author={Thomsen, Klaus},
  journal={Proc. Amer. Math. Soc.},
  volume={129},
  number={5},
  pages={1409--1417},
  year={2001}
}

@book {BlackadarEncycBook,
    AUTHOR = {Blackadar, Bruce},
     TITLE = {Operator algebras},
    SERIES = {Encyclopaedia of Mathematical Sciences},
    VOLUME = {122},
      NOTE = {Theory of {C$^*$}-algebras and von {N}eumann algebras,
              {O}perator {A}lgebras and {N}on-commutative {G}eometry, {III}},
 PUBLISHER = {Springer-Verlag, Berlin},
      YEAR = {2006},
     PAGES = {xx+517},
}

@article {DynamicalKP,
    AUTHOR = {Gabe, James and Szab\'{o}, G\'{a}bor},
     TITLE = {The dynamical {K}irchberg-{P}hillips theorem},
   JOURNAL = {Acta Math.},
  FJOURNAL = {Acta Mathematica},
    VOLUME = {232},
      YEAR = {2024},
    NUMBER = {1},
     PAGES = {1--77},
}

@inproceedings {BDF1,
    AUTHOR = {Brown, Lawrence G. and Douglas, Ronald G. and Fillmore, Peter A.},
     TITLE = {Unitary equivalence modulo the compact operators and
              extensions of {C$^*$}-algebras},
 BOOKTITLE = {Proceedings of a {C}onference on {O}perator {T}heory
              ({D}alhousie {U}niv., {H}alifax, {N}.{S}., 1973)},
     PAGES = {58--128. Lecture Notes in Math., Vol. 345},
      YEAR = {1973},
}

@article {BDF2,
    AUTHOR = {Brown, Lawrence G. and Douglas, Ronald G. and Fillmore, Peter A.},
     TITLE = {Extensions of {C$^*$}-algebras and {K}-homology},
   JOURNAL = {Ann. of Math. (2)},
  FJOURNAL = {Annals of Mathematics. Second Series},
    VOLUME = {105},
      YEAR = {1977},
    NUMBER = {2},
     PAGES = {265--324},   
}

@article {elliott,
    AUTHOR = {Elliott, George A.},
     TITLE = {On the classification of inductive limits of sequences of
              semisimple finite-dimensional algebras},
   JOURNAL = {J. Algebra},
  FJOURNAL = {Journal of Algebra},
    VOLUME = {38},
      YEAR = {1976},
    NUMBER = {1},
     PAGES = {29--44},
}

@article {phillipsclass,
    AUTHOR = {Phillips, N. Christopher},
     TITLE = {A classification theorem for nuclear purely infinite simple
              {C$^*$}-algebras},
   JOURNAL = {Doc. Math.},
  FJOURNAL = {Documenta Mathematica},
    VOLUME = {5},
      YEAR = {2000},
     PAGES = {49--114},
}

@misc{kirchbergclass,
author ={Kirchberg, Eberhard},
title = {The classification of purely infinite {C$^*$}-algebras using {K}asparov's theory},
note ={manuscript available at https://ivv5hpp.uni-muenster.de/u/echters/ekneu1.pdf},
year = "1994"
}

@inproceedings {Kirch00,
    AUTHOR = {Kirchberg, Eberhard},
     TITLE = {Exact {C$^*$}-algebras, tensor products, and the
              classification of purely infinite algebras},
 BOOKTITLE = {Proceedings of the {I}nternational {C}ongress of
              {M}athematicians, {V}ol. 1, 2 ({Z}\"{u}rich, 1994)},
     PAGES = {943--954},
 PUBLISHER = {Birkh\"{a}user, Basel},
      YEAR = {1995},
}

@article{oinftyclass,
    AUTHOR = {Gabe, James},
     TITLE = {Classification of {$\mathcal{O}_\infty$}-stable {C$^*$}-algebras},
   JOURNAL = {Mem. Amer. Math. Soc.},
  FJOURNAL = {Memoirs of the American Mathematical Society},
    VOLUME = {293},
      YEAR = {2024},
    NUMBER = {1461},
     PAGES = {v+115},
}

@article{classif,
  title={Classifying $^*$-homomorphisms {I}: {U}nital simple nuclear {C$^*$}-algebras},
  author={Carri{\'o}n, Jos{\'e} R. and Gabe, James and Schafhauser, Christopher and Tikuisis, Aaron and White, Stuart},
  journal={arXiv:2307.06480},
  year={2023}
}

@article {CS20,
    AUTHOR = {Schafhauser, Christopher},
     TITLE = {Subalgebras of simple {AF}-algebras},
   JOURNAL = {Ann. of Math. (2)},
  FJOURNAL = {Annals of Mathematics. Second Series},
    VOLUME = {192},
      YEAR = {2020},
    NUMBER = {2},
     PAGES = {309--352},
}

@incollection {Cu83,
    AUTHOR = {Cuntz, Joachim},
     TITLE = {Generalized homomorphisms between {C$^* $}-algebras and
              {$\KK$}-theory},
 BOOKTITLE = {Dynamics and processes ({B}ielefeld, 1981)},
    SERIES = {Lecture Notes in Math.},
    VOLUME = {1031},
     PAGES = {31--45},
 PUBLISHER = {Springer, Berlin},
      YEAR = {1983},
}

@incollection {Cu84,
    AUTHOR = {Cuntz, Joachim},
     TITLE = {{$K$}-theory and {C$^*$}-algebras},
 BOOKTITLE = {Algebraic {$K$}-theory, number theory, geometry and analysis
              ({B}ielefeld, 1982)},
    SERIES = {Lecture Notes in Math.},
    VOLUME = {1046},
     PAGES = {55--79},
 PUBLISHER = {Springer, Berlin},
      YEAR = {1984},
}

@article {Th98,
    AUTHOR = {Thomsen, Klaus},
     TITLE = {The universal property of equivariant {$\KK$}-theory},
   JOURNAL = {J. Reine Angew. Math.},
  FJOURNAL = {Journal f\"{u}r die Reine und Angewandte Mathematik. [Crelle's
              Journal]},
    VOLUME = {504},
      YEAR = {1998},
     PAGES = {55--71},
}

@article {Lin02,
    AUTHOR = {Lin, Huaxin},
     TITLE = {Stable approximate unitary equivalence of homomorphisms},
   JOURNAL = {J. Operator Theory},
  FJOURNAL = {Journal of Operator Theory},
    VOLUME = {47},
      YEAR = {2002},
    NUMBER = {2},
     PAGES = {343--378},
}

@article {TWW,
    AUTHOR = {Tikuisis, Aaron and White, Stuart and Winter, Wilhelm},
     TITLE = {Quasidiagonality of nuclear {C$^*$}-algebras},
   JOURNAL = {Ann. of Math. (2)},
  FJOURNAL = {Annals of Mathematics. Second Series},
    VOLUME = {185},
      YEAR = {2017},
    NUMBER = {1},
     PAGES = {229--284},    
}

@article{jonessubfactors,
  title={Index for subfactors},
  author={Jones, Vaughan FR},
  journal={Invent. Math.},
  volume={72},
  number={1},
  pages={1--25},
  year={1983},
  publisher={Springer}
}

@article {Popa94,
    AUTHOR = {Popa, Sorin},
     TITLE = {Classification of amenable subfactors of type {II}},
   JOURNAL = {Acta Math.},
  FJOURNAL = {Acta Mathematica},
    VOLUME = {172},
      YEAR = {1994},
    NUMBER = {2},
     PAGES = {163--255},
}

@article {ATSI68,
    AUTHOR = {Atiyah, M. F. and Singer, I. M.},
     TITLE = {The index of elliptic operators. {I}},
   JOURNAL = {Ann. of Math. (2)},
  FJOURNAL = {Annals of Mathematics. Second Series},
    VOLUME = {87},
      YEAR = {1968},
     PAGES = {484--530},
      ISSN = {0003-486X},
   MRCLASS = {57.50},
  MRNUMBER = {236950},
MRREVIEWER = {F.\ Hirzebruch},
       DOI = {10.2307/1970715},
       URL = {https://doi.org/10.2307/1970715},
}

@article {Ka80,
    AUTHOR = {Kasparov, G. G.},
     TITLE = {Hilbert {C$^*$}-modules: theorems of {S}tinespring and
              {V}oiculescu},
   JOURNAL = {J. Operator Theory},
  FJOURNAL = {Journal of Operator Theory},
    VOLUME = {4},
      YEAR = {1980},
    NUMBER = {1},
     PAGES = {133--150},
}

@article {Voi76,
    AUTHOR = {Voiculescu, Dan},
     TITLE = {A non-commutative {W}eyl-von {N}eumann theorem},
   JOURNAL = {Rev. Roumaine Math. Pures Appl.},
  FJOURNAL = {Acad\'{e}mie de la R\'{e}publique Populaire Roumaine. Revue Roumaine
              de Math\'{e}matiques Pures et Appliqu\'{e}es},
    VOLUME = {21},
      YEAR = {1976},
    NUMBER = {1},
     PAGES = {97--113},
}

@article {ChrisTWW,
    AUTHOR = {Schafhauser, Christopher},
     TITLE = {A new proof of the {T}ikuisis-{W}hite-{W}inter theorem},
   JOURNAL = {J. Reine Angew. Math.},
  FJOURNAL = {Journal f\"{u}r die Reine und Angewandte Mathematik. [Crelle's
              Journal]},
    VOLUME = {759},
      YEAR = {2020},
     PAGES = {291--304},
}

@article {EllKu01,
    AUTHOR = {Elliott, George A. and Kucerovsky, Dan},
     TITLE = {An abstract {V}oiculescu-{B}rown-{D}ouglas-{F}illmore
              absorption theorem},
   JOURNAL = {Pacific J. Math.},
  FJOURNAL = {Pacific Journal of Mathematics},
    VOLUME = {198},
      YEAR = {2001},
    NUMBER = {2},
     PAGES = {385--409},
}

@book {CMR07,
    AUTHOR = {Cuntz, Joachim and Meyer, Ralf and Rosenberg, Jonathan M.},
     TITLE = {Topological and bivariant {$K$}-theory},
    SERIES = {Oberwolfach Seminars},
    VOLUME = {36},
 PUBLISHER = {Birkh\"{a}user Verlag, Basel},
      YEAR = {2007},
     PAGES = {xii+262},
}

@article {SZ17,
    AUTHOR = {Szab\'o, G\'abor},
     TITLE = {Strongly self-absorbing {C}$^*$-dynamical systems,
              {III}},
   JOURNAL = {Adv. Math.},
  FJOURNAL = {Advances in Mathematics},
    VOLUME = {316},
      YEAR = {2017},
     PAGES = {356--380},
      ISSN = {0001-8708,1090-2082},
   MRCLASS = {46L55},
  MRNUMBER = {3672909},
MRREVIEWER = {Geoffrey\ Price},
       DOI = {10.1016/j.aim.2017.06.008},
       URL = {https://doi.org/10.1016/j.aim.2017.06.008},
}

@book {MU90,
    AUTHOR = {Murphy, Gerard J.},
     TITLE = {{$C^*$}-algebras and operator theory},
 PUBLISHER = {Academic Press, Inc., Boston, MA},
      YEAR = {1990},
     PAGES = {x+286},
      ISBN = {0-12-511360-9},
   MRCLASS = {46Lxx (46-01)},
  MRNUMBER = {1074574},
MRREVIEWER = {E.\ Gerlach},
}
\end{document}